\documentclass[preperint,oneside]{amsart}
\usepackage{lipsum} 
\usepackage{amsfonts}
\usepackage{graphicx}
\usepackage{epstopdf} 
\usepackage{overpic}
\usepackage{todonotes}
\ifpdf
\DeclareGraphicsExtensions{.eps,.pdf,.png,.jpg}
\else
\DeclareGraphicsExtensions{.eps}
\fi
\usepackage{amsopn}
\usepackage{paralist}
\usepackage{graphics} 
\usepackage{epsfig} 
\usepackage{xspace}
\usepackage{marginnote} 

\usepackage{array}

\usepackage{appendix}
\usepackage[english]{babel} 
\usepackage[T1]{fontenc} 
\usepackage[utf8]{inputenc} 
\usepackage{geometry}
\usepackage{graphicx}
\usepackage{tabto}
\usepackage{amsmath}
\usepackage{calligra}
\usepackage{bbm}
\usepackage{cancel}
\usepackage{mathrsfs}
\usepackage{mathtools}
\usepackage{version}
\usepackage{arydshln}
\usepackage[math]{cellspace}

\usepackage{multicol}

\usepackage{bm}
\usepackage{subfig}

\usepackage{tikz}

\usepackage{amssymb}
\newcommand{\numberset}{\mathbb}

\newcommand{\R}{\numberset{R}}
\newcommand{\F}{\numberset{F}}
\newcommand{\I}{\numberset{I}}
\newcommand{\PP}{\numberset{P}}

\newcommand{\TT}{\numberset{T}}

\newcommand{\VV}{\numberset{V}}

\DeclareFontFamily{U}{matha}{\hyphenchar\font45}
\DeclareFontShape{U}{matha}{m}{n}{
	<-6> matha5 <6-7> matha6 <7-8> matha7
	<8-9> matha8 <9-10> matha9
	<10-12> matha10 <12-> matha12
}{}
\DeclareSymbolFont{matha}{U}{matha}{m}{n}
\DeclareFontFamily{U}{mathx}{\hyphenchar\font45}
\DeclareFontShape{U}{mathx}{m}{n}{
	<-6> mathx5 <6-7> mathx6 <7-8> mathx7
	<8-9> mathx8 <9-10> mathx9
	<10-12> mathx10 <12-> mathx12
}{}
\DeclareSymbolFont{mathx}{U}{mathx}{m}{n}
\DeclareMathDelimiter{\vvvert} {0}{matha}{"7E}{mathx}{"17}%
\newcommand{\norm}[1]{\left\lVert#1\right\rVert}
\newcommand{\abs}[1]{\left\lvert#1\right\rvert}

\newtheorem{teo}{Theorem}

\newtheorem{pro}{Problem}
\newtheorem{lem}{Lemma}
\newtheorem{prop}{Proposition}
\newtheorem{rem}{Remark}

\newtheorem{assump}{Assumption}

\let\div\undefined\DeclareMathOperator{\div}{div} 
\let\curl\undefined\DeclareMathOperator{\curl}{curl} 
\DeclareMathOperator*{\Grad}{\boldsymbol\nabla}
\DeclareMathOperator*{\Grads}{{\underline{\boldsymbol{\varepsilon}}}}
\DeclareMathOperator{\grads}{\boldsymbol\nabla_s}
\newcommand{\derivative}[2]{\frac{\partial #1}{\partial #2}}
\newcommand\Huo{\mathbf{H}^1_0(\Omega)}
\newcommand\Ldo{L^2_0(\Omega)} 
\newcommand\Hub{\mathbf{H}^1(\B)} 
\newcommand\HuOm{\mathbf{H}^1(\Omega)} 
\newcommand\Hubd{(\Hub)^\prime} 
\newcommand\LdBd{\mathbf{L}^2(\B)} 
\newcommand\LdOd{\mathbf{L}^2(\Omega)} 
\newcommand\Winftyd{\mathbf{W}^{1,\infty}(\B)}
\newcommand\Oft{\Omega^f_t} 
\newcommand\Ost{\Omega^s_t} 
\newcommand\Of{\Omega^f} 
\newcommand\Os{\Omega^s} 
\newcommand\B{\mathcal B} 
\renewcommand\u{\mathbf{u}} 
\renewcommand\a{\mathbf{a}} 
\renewcommand\v{\mathbf{v}} 
\renewcommand\d{\mathbf{d}} 
\renewcommand\c{\mathbf{c}} 
\newcommand\f{\mathbf{f}} 
\newcommand\g{\mathbf{g}} 
\newcommand\n{\mathbf{n}} 
\newcommand\w{\mathbf{w}} 
\newcommand\x{\mathbf{x}} 
\newcommand\s{\mathbf{s}} 
\renewcommand\S{\mathbf{S}} 

\newcommand\V{\mathcal{V}} 
\newcommand\U{\mathcal{U}} 
\newcommand\Bil{\mathcal{L}} 

\newcommand\X{\mathbf{X}} 
\newcommand\Xbar{\overline{\X}} 
\newcommand\Y{\mathbf{Y}} 
\newcommand\Z{\mathbf{Z}} 

\newcommand\LL{{\boldsymbol{\Lambda}}} 
\newcommand\Vline{\mathbf{V}} 
\newcommand\Sline{\S} 

\newcommand\ssigma{\boldsymbol\sigma} 
\newcommand\llambda{\boldsymbol\lambda} 
\newcommand\mmu{\boldsymbol\mu} 
\newcommand\ds{\mathrm{d}\s} 
\newcommand\dx{\mathrm{d}\x} 
\newcommand\dt{\Delta t} 
\newcommand\dr{\delta\rho} 
\newcommand\T{\mathcal{T}} 
\newcommand\quadnode{\mathbf{p}}
\newcommand\quadweigth{\omega}

\newcommand\Pcal{\mathcal{P}}
\newcommand\tri{\mathrm{T}} 
\newcommand\tris{\tri_\mathrm{s}} 

\newcommand{\Amatr}{\mathsf{A}}
\newcommand{\Bmatr}{\mathsf{B}}
\newcommand{\Cmatr}{\mathsf{C}}

\newcommand{\Sys}{\mathcal{A}}

\newcommand{\cond}[1]{\mathsf{cond}(#1)}

\usepackage{listings}
\newcommand\LG{\begin{color}{black}}     
\newcommand\GL{\end{color}} 
\newcommand\fc{\begin{color}{black}}     
\newcommand\cf{\end{color}}

\newcommand{\FC}{\begin{color}{black}}
\newcommand{\CF}{\end{color}}

\renewcommand\lg{\begin{color}{black}}
	\newcommand\gl{\end{color}}
\newcommand\blue{\begin{color}{black}}
	\newcommand\noblue{\end{color}}

\newcommand\red{\begin{color}{black}}
	\newcommand\nored{\end{color}}

\newcommand\DB{\begin{color}{black}}
	\newcommand\BD{\end{color}}

\begin{document}

	\title[Conditioning of a fictitious domain method for FSI]{On the stability and conditioning\\of a fictitious domain formulation\\for fluid-structure interaction problems}
	
	\author[Daniele Boffi]{Daniele Boffi}
	\address{Computer, Electrical and Mathematical Sciences and Engineering Division, King Abdullah University of Science and Technology, Thuwal 23955, Saudi Arabia and Dipartimento di Matematica \textquoteleft F. Casorati\textquoteright, Universit\`a degli Studi di Pavia, via Ferrata 5, 27100, Pavia, Italy}
	\email{daniele.boffi@kaust.edu.sa}
	\urladdr{kaust.edu.sa/en/study/faculty/daniele-boffi}
	
	\author{Fabio Credali}
	\address{Computer, Electrical and Mathematical Sciences and Engineering Division, King Abdullah University of Science and Technology, Thuwal 23955, Saudi Arabia}
	\email{fabio.credali@kaust.edu.sa}
	\urladdr{cemse.kaust.edu.sa/profiles/fabio-credali}
	
	\author{Lucia Gastaldi}
	\address{Dipartimento di Ingegneria Civile, Architettura, Territorio, Ambiente e di Matematica, Universit\`a degli Studi di Brescia, via Branze 43, 25123, Brescia, Italy}
	\email{lucia.gastaldi@unibs.it}
	\urladdr{lucia-gastaldi.unibs.it}
	
	\begin{abstract}

\fc We consider a distributed Lagrange multiplier formulation for
fluid-structure interaction problems in the spirit of the fictitious
domain approach. This is an unfitted method, which does not require the construction of meshes conforming to the interface. We focus on the stationary problem arising from the time discretization and we analyze the behavior of
the condition number with respect to mesh refinement. At the numerical level, the computation of the term coupling the fluid and the solid mesh requires the knowledge of the intersection between fluid and mapped solid elements and it might happen that a portion of the intersected elements is very small. We show that our formulation is stable independently of such intersections and that the conditioning is not affected by the
interface position.\cf

\

\noindent\textbf{Keywords:} Fluid-structure interactions, Fictitious domain, Stability, Condition number, Non-matching grids.

\

\noindent\textbf{Mathematics Subject Classification:} 65N30, 65N12, 74F10
		
\end{abstract}

	\maketitle	
	
	\section{Introduction}
	
	The scientific community has a large interest in developing effective methods for solving partial differential equations modeling multi-physics and multi-bodies problems. Such equations are often defined on time dependent domains characterized by complex interfaces. In our research, we focus on fluid-structure interaction problems, which have several applications in biology, medicine, structural mechanics, and several other fields.
	
	The discretization of these problems is challenging since it requires an accurate representation of the underlying geometry. To this aim, several methods have been developed during the last decades. They can be classified as \textit{fitted} or \textit{unfitted} depending on the fact that the computational mesh is aligned or not with the interface.
	\fc Among fitted approaches, the most popular is the Arbitrary Lagrangian--Eulerian (ALE)~\cite{hirt1974arbitrary,hughes1981lagrangian,donea1982arbitrary}. Fluid and solid domains are discretized by two meshes evolving around a shared interface. Despite the advantage of having the kinematic constraints imposed by construction, strong mesh distortions may appear. \cf
	Unfitted approaches may overcome \fc this issue\cf. In this case, fluid and solid are discretized independently, with meshes not fitting at the interface. Thus, the complexity of mesh generation is minimized. \fc In some cases, \cf this comes at the price of reducing the accuracy of the scheme unless special techniques are used. \fc Among unfitted approaches for FSI and multiphysics problems\cf, we mention, for instance, the level set method~\cite{SUSSMAN1994146,chang1996level}, the immersed boundary method~\cite{peskin1972flow,peskin2002immersed}, the fictitious domain approach~\cite{glowinski1997lagrange,glowinski2001fictitious,yu2005dlm}, \fc the immersogeometric method\cf~\cite{iga-heart}, the shifted boundary method~\cite{sbm-flow}, the fat boundary method~\cite{fat-interface}, and, finally, the Nitsche
	method. \FC Within the family of Nitsche-type methods, we find schemes based on a fictitious domain approach as in~\cite{burman1,burman2}, the CutFEM~\cite{wadbro2013uniformly,
		HANSBO201490, burman2015cutfem, cutStokes}, the Finite Cell
	Method~\cite{finitecell1, finitecell2, finitecell3} and the Nitsche--XFEM~\cite{alauzet2016nitsche, burman2014unfitted}.\CF
	
	\FC
	Our approach falls within the class of unfitted methods for FSI. It originated from the finite element immersed boundary method~\cite{pesk} and further evolved into a fictitious domain method based on the use of a Lagrange multiplier~\cite{2015,stat}. 
	\CF
	We consider the simulation of fluid-structure interaction problems which couple an incompressible Newtonian fluid with an incompressible visco-elastic solid body. This setting is motivated by the fact that many biological tissues satisfy this assumption, see e.g.~\cite{peskin1972flow}. Actually, our technique can be extended also to the case of compressible solids as in~\cite{compressible}. See also~\cite{heltai-costanzo}, where the finite element immersed boundary method is generalized to an incompressible Newtonian fluid interacting with general hyper-elastic structures.
	
	The main idea is to fictitiously extend the fluid domain to the region occupied by the immersed structure
	so that we solve the Navier--Stokes equations using the Eulerian framework in a fixed domain.
	For the solid deformation we use the standard Lagrangian framework and solve the governing elasticity equation on a reference domain mapped, at each time instant, into the
	actual configuration. 
	A  distributed Lagrange multiplier is introduced to deal with the kinematic constraint describing the interaction
	between fluid and solid at variational level.
\fc
The fluid and the solid equations are then coupled by additional
terms depending on the Lagrange multiplier.
In particular, in the fluid equation, this term involves functions defined on the two domains,
hence we refer to it as \emph{coupling term}.

At the computational level, the main advantage of our approach is that we can
use totally independent meshes which do not evolve in time. However, the
computation of the coupling term requires to integrate on the elements of the
reference mesh mapped on the background fluid mesh.
	Such computation requires the knowledge of the intersections between the fluid mesh
	and the mapped solid one. Hence, very small cells contained in the support of both fluid and mapped solid basis functions may originate. We refer to them as \emph{small cut cells}. Then, the matrix associated with
	the coupling term can  be evaluated exactly by employing a
	quadrature rule on each intersected element. An inexact procedure can
	also be considered, using a single quadrature rule on each mapped element.
	In this case, it might happen that no quadrature nodes fall inside a
	small cut cell.  
	We estimated the error arising from the use of inexact quadrature 
	in our recent works~\cite{boffi2022interface,BCG24}. In this paper, we discuss the effect of the presence of small cut cells on the stability and conditioning of the considered formulation. \FC On the other hand, in the case of Nitsche-type \CF methods,
the term \emph{cut cell} refers to a mesh element which is cut by the
interface and it might happen that one portion of the cut cell is small. It is
known that this fact can produce instabilities in the solution, as well as ill-conditioning, unless stabilization/penalization terms, such as the ghost penalty technique~\cite{burman-ghost}, are added to the variational formulation.
	
	In this work, we recall the main features of our distributed Lagrange
	multiplier formulation and we prove upper bounds for the condition number, depending only on the choice of the coupling term and not on the integration technique or
	the interface position. 
	On the other hand, the inf-sup stability is guaranteed at the discrete
	level without the need for any artificial penalization term dealing with
	small cut cells. \cf
	
	Our theoretical results on the stability and conditioning are confirmed by
	extensive numerical investigations. 
	\fc A final numerical test involving a time dependent problem shows that the method is unconditionally stable in time, independently of the presence of small cut cells.\cf
	
	\FC
	The paper is organized as follows. After recalling the functional analysis notation, we derive our fictitious domain formulation starting from the physical background (see Section~\ref{sec:pb_setting}). In Section~\ref{sec:time}, we recall the time semi-discretization and the related stability properties, whereas computational aspects regarding the coupling term are recalled in Section~\ref{sec:coupling_term}. We prove the upper bounds for the condition number in Section~\ref{sec:conditioning}. Our numerical investigation is finally presented in Section~\ref{sec:tests}. The derivation of the model and the results of Sections~\ref{sec:time} and \ref{sec:coupling_term} are essentially known. We summarized them here since we need them for our analysis of Section~\ref{sec:conditioning} with the intention of being self-contained.\CF
	
	\
	
	\textbf{Notation.} Let $D\subset\R^2$ be an open bounded domain. The space of square-integrable functions on $D$ is denoted by $L^2(D)$ and endowed with the scalar product $(\cdot,\cdot)_D$, inducing the norm $\|\cdot\|_{0,D}$. We denote by~$L^2_0(D)$ the subspace of zero mean valued functions. Sobolev spaces are denoted by $W^{s,p}(D)$, with $s\in\R$ being the differentiability exponent and $p\in[1,\infty]$ the integrability exponent. When $p=2$, we adopt the notation ${H^s(D)=W^{s,2}(D)}$. The space $H^s(D)$ is endowed with the norm $\|\cdot\|_{s,D}$. Finally, the space $H^1_0(D)\subset H^1(D)$ contains functions with zero trace on $\partial D$. Vector spaces and vector/tensor valued functions are denoted by boldface letters. 
	
	\section{Problem setting}\label{sec:pb_setting}
	
	We focus on fluid-structure interaction problems consisting of an elastic body immersed in a Newtonian fluid. The configuration of the system is described by two regions $\Ost$ and $\Oft$ representing the position of solid and fluid at time $t$, respectively, \fc and sharing a common interface $\partial\Ost\cap\partial\Oft$\cf. More precisely, both $\Ost$ and $\Oft$ can be either two- or three-dimensional domains.  We point out that our approach can also deal with codimension one (see e.g.~\cite{2015,annese}) and codimension two solids (see e.g.~\cite{ALZETTA,heltaizunino}). We then introduce a third domain $\Omega=\Ost\cup\Oft$, acting as a container for the system and independent of time. We assume \fc that the solid domain $\Ost$ is totally immersed in the fluid one, so that its closure does not touch the boundary of the container~$\partial\Omega$. Therefore, the interface $\partial\Ost\cap\partial\Oft$ coincides with $\partial\Ost$.\cf
	
	The fluid is described in the Eulerian framework, whereas the solid deformation is described in the Lagrangian setting. Hence, we introduce a solid reference domain $\B$; at each time instant, $\B$ is mapped into the actual position of the body $\Ost$ through the action of a map $\X:\B\rightarrow\Ost$. Thus, $\X$ represents the position of the solid and each point $\x\in\Ost$ can be expressed in terms of $\s$ as~$\x = \X(\s,t)$.
	
	We denote by $\F$ the deformation gradient $\grads\X$, and $|\F|$ its determinant. In case of incompressible materials, $|\F|$ is constant during time. In particular, $|\F|=1$ when the reference domain $\B$ coincides with the initial configuration of the immersed body. In addition, the time derivative $\partial\X/\partial t$ is equal to the material velocity $\u^s$ of the solid, i.e.
	\begin{equation}\label{eq:kinematik_eq_solid}
		\derivative{\X}{t}(\s,t) = \u^s(\x,t)\qquad\text{for }\x=\X(\s,t).
	\end{equation}
	An example of configuration is sketched in Figure~\ref{fig:geo}.
	\begin{figure}
		\centering
		\begin{tikzpicture}
			\draw[fill=cyan!20, thick] (0,0) -- (4,0) -- (4,4) -- (0,4) -- (0,0);
			\draw (4.5,3.7) node {$\Omega$};
			\draw (0.5,3.7) node {$\Oft$};	
			
			\draw (-4.5,3.7) node {$\B$};
			\coordinate (c1) at (-4.5,0.5);
			\coordinate (c2) at (-3.5,0.7);
			\coordinate (c3) at (-2,1.8);
			\coordinate (c4) at (-2,3);
			\coordinate (c5) at (-3,3.5);
			\coordinate (c8) at (-4.5,3);
			\coordinate (c7) at (-4.7,2);
			\coordinate (c6) at (-4.8,1);
			\draw[fill = yellow!30, thick] plot [smooth cycle] coordinates {(c1) (c2) (c3) (c4) (c5) (c8) (c7) (c6)};			
			
			\coordinate (c1) at (1.5,0.5);
			\coordinate (c2) at (2.5,0.7);
			\coordinate (c3) at (3,1.8);
			\coordinate (c4) at (3,3);
			\coordinate (c5) at (2,3.5);
			\coordinate (c8) at (1.5,3);
			\coordinate (c7) at (1,2);
			\coordinate (c6) at (0.6,1);
			\draw[fill = yellow!30, thick] plot [smooth cycle] coordinates {(c1) (c2) (c3) (c4) (c5) (c8) (c7) (c6)};
			\draw (2.2,3) node {$\Ost$};
			
			\draw (-3.2,2.3) node {$\s$};
			\draw (2.1,2) node {$\x$};
			
			\draw (-1.1,2.4) node {$\X(\s,t)$};
			
			\draw[->,thick,violet] (-3,2.3)  .. controls (0,2)  ..  (1.9,2);
		\end{tikzpicture}
		\caption{Our geometrical setting. $\Omega$ and $\B$ are fixed domains, independent of time. A Lagrangian point $\s\in\B$ is mapped into $\x\in\Ost$ through the map $\X$.}
		\label{fig:geo}
	\end{figure}
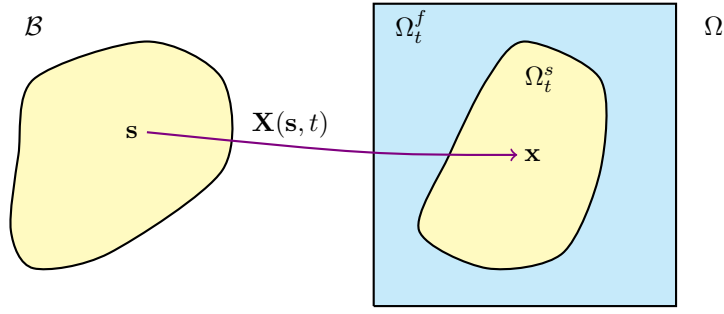
	
	Assuming that both fluid and solid materials are incompressible, the equations governing the evolution of our system are
	\begin{equation}\label{eq:strong}
		\begin{aligned}
			&\rho_f\bigg(\derivative{\u^f}{t}+\u^f\cdot\nabla\u^f\bigg)=\div\ssigma^f&&\quad\text{in }\Oft\\
			&\div\u^f = 0&&\quad\text{in }\Oft\\
			&\rho_s\frac{\partial^2\X}{\partial^2 t} = \div_\s\TT &&\quad\text{in }\B\\ 
			&|\F| = 1 &&\quad\text{in }\B.
		\end{aligned}
	\end{equation}
	Here, $\u^f$ is the fluid velocity, while $\rho_f,\,\rho_s$ are fluid and solid densities. The symbols $\ssigma_f$ and $\TT$ denote the Cauchy stress tensors for fluid and solid, respectively. More precisely, $\ssigma^f$ is given by
	\begin{equation*}
		\ssigma^f = -p^f \I + \nu_f\Grads(\u^f),
	\end{equation*}
	where $\nu_f>0$ is the fluid viscosity, $p^f$ the pressure, and $\Grads(\u^f)=(\Grad\u^f+(\Grad\u^f)^\top)/2$ is the symmetric gradient.
	
	We consider a visco-elastic solid material; the associated Cauchy stress tensor is composed of two terms
	\begin{equation}
		\TT = |\F| \ssigma_v^s \F^{-\top} + \PP\qquad\text{with }\ssigma_v^s = -p^s \I + \nu_s\Grads(\u^s),
	\end{equation}
	where $\ssigma_v^s$ represents the viscous contribution, while $\u^s$, $p^s$ are velocity and pressure of the solid, respectively, and $\nu_s>0$ denotes the viscosity. The pressure $p^s$ is the Lagrange multiplier associated to the incompressibility condition. $\PP$ is the first Piola--Kirchhoff stress tensor, which can be expressed in terms of a potential energy density~$W(\F,\s,t)$ as $\PP(\F,\s,t) = \displaystyle\derivative{W}{\F}(\F,\s,t)$.
	
	The equations in~\eqref{eq:strong} are completed by suitable interface, boundary, and initial conditions. Velocity and Cauchy stress are imposed to be continuous at the interface by the transmission conditions
	\begin{equation*}
		\begin{aligned}
			&\u^f = \u^s && \text{ on } \partial\Ost\\
			&\ssigma_f\n_f = - \big(|\F|^{-1}\TT\F^\top)\n_s && \text{ on } \partial\Ost,
		\end{aligned}
	\end{equation*}
	where $\n_s$ and $\n_f$ are the outer unit normal vectors to $\Ost$ and $\Oft$, respectively. We finally add homogeneous boundary condition for the fluid velocity $\u^f =0$ on $\partial\Omega$ and the initial conditions
	\begin{equation*}
		\u^f(0) = \u_0^f\ \text{ in }\Of_0,
		\qquad
		\u^s(0) = \u_0^s\ \text{ in }\Os_0,
		\qquad
		\X(0) = \X_0\ \text{ in } \B.
	\end{equation*}
	
	In order to present the fictitious domain formulation of the above equations, we extend the fluid variables to the entire $\Omega$ incorporating the solid domain
	\begin{equation}
		\u = \begin{cases}
			\u^f \quad\text{ in }\Oft\\
			\u^s \quad\text{ in }\Ost,
		\end{cases}
		\qquad
		p = \begin{cases}
			p^f \quad\text{ in }\Oft\\
			p^s \quad\text{ in }\Ost.
		\end{cases}
	\end{equation}
	Since $\u$ must coincide with $\u^s$ in the solid region $\Ost$, we impose the following kinematic constraint in $\B$ corresponding to~\eqref{eq:kinematik_eq_solid}
	\begin{equation}\label{eq:kinematic_fictitious}
		\derivative{\X}{t}(\s,t) = \u(\X(\s,t),t) \quad \text{ for } \s\in\B.
	\end{equation}
	We enforce~\eqref{eq:kinematic_fictitious} weakly by means of a distributed Lagrange multiplier. To this aim, we consider a suitable Hilbert space $\LL$ and we introduce a continuous bilinear form $\c:\LL\times\Hub\rightarrow\R$ with the property
	\begin{equation}
		\c(\mmu,\Z) = 0 \quad\forall\mmu\in\LL \quad\text{implies}\quad \Z=0.
	\end{equation}
	Two possible choices of $\LL$ and $\c$ are used in our model. We can set $\LL_0=\Hubd$, i.e. the dual space of $\Hub$, and define $\c_0$ as the duality pairing
	\begin{equation}\label{eq:c_dual}
		\c_0(\mmu,\Y) = \langle \mmu,\Y \rangle \quad \forall\mmu\in\LL_0,\;\forall\Y\in\Hub,
	\end{equation}
	or we can choose $\LL_1=\Hub$ and set $\c_1$ as the inner product in $\Hub$
	\begin{equation}\label{eq:c_h1}
		\c_1(\mmu,\Y) = (\mmu,\Y)_\B + (\grads\mmu,\grads\Y)_\B \quad \forall\mmu\in\LL_1,\,\forall\Y\in\Hub.
	\end{equation}
	
	From now on, we will use the generic notation $(\LL,\c)$ and we specify
	$(\LL_0,\c_0)$ or $(\LL_1,\c_1)$ when necessary.
	Finally, the weak formulation of our problem reads as follows (see~\cite{2015} for its full derivation).
	
	\begin{pro}
		\label{pro:problem_fictitious}
		Given $\u_0\in\Huo$ and $\X_0:\B\longrightarrow\Omega$, $\forall t\in (0,T)$, find $(\u(t),p(t))\in\Huo \times \Ldo$, $\X(t)\in \Winftyd$ and $\llambda(t)\in \LL$, such that for almost every $t\in(0,T)$ it holds
		\begin{subequations}
			\begin{align}
				&\notag \rho_f\frac{\partial}{\partial t}(\u(t),\v)+b(\u(t),\u(t),\v)+a(\u(t),\v)\\
				&\qquad\qquad\qquad\quad-(\div\v,p(t))+\c(\llambda(t),\v(\X(t)))=0 
				&&\forall\v\in\Huo\\
				&(\div\u(t),q)=0  &&\forall q\in \Ldo\\
				& \label{eq:fict_solid} \dr \bigg(\frac{\partial^2\X}{\partial t^2},\Y \bigg)_\B+(\PP(\F(\s,t)),\grads\Y )_\B-\c(\llambda(t),\Y)=0&&\forall\Y\in \Hub \\
				& \c\bigg(\mmu,\u(\X(\cdot,t),t)-\frac{\partial\X}{\partial t}(t) \bigg)=0 &&\forall\mmu\in\LL \\
				&\u(\x,0)=\u_0(\x) &&\forall\x\in\Omega\\
				&\X(\s,0)=\X_0(\s) &&\forall\s\in\B.
			\end{align}
		\end{subequations}
	\end{pro}
	Here, we adopted the following notation: $\dr = \rho_s-\rho_f$ and
	\begin{equation*}
		a(\u,\v)=\int_\Omega \nu\Grads(\u):\Grads(\v)\,\dx,\qquad b(\u,\v,\w)=\frac{\rho_f}{2}((\u\cdot\nabla\v,\w)-(\u\cdot\nabla\w,\v)).
	\end{equation*}
	The extended viscosity $\nu$ is equal to $\nu_f$ in $\Oft$ and $\nu_s$ in $\Ost$. For simplicity, in the remainder of the paper, we neglect the non-linear convective term $b(\u,\v,\w)$.
	
	\FC
	\begin{rem} 
		By exploiting the incompressibility of both the fluid and solid materials, it is straightforward to impose the divergence free condition to the extended velocity. Notice that, in this case, the variable $p$ represents the Lagrange multiplier associated to such constraint, as usual for the Navier--Stokes equations. This might be considered as a limitation of our method. However, it is possible to extend our fictitious domain approach to the case of compressible solids. This is done for instance in~\cite{compressible}, where a fictitious pressure field without any physical meaning is introduced in the solid domain, and it is weakly imposed to be zero.
		Moreover, in our analysis we strongly rely on the assumption that the solid material has a viscous component. The method can be generalized to hyper-elastic structures as done in~\cite{heltai-costanzo}, where the authors take due account of the additional viscous term that is fictitiously added to the extended fluid equation.
	\end{rem}
	\CF
	
	\section{Time discretization and saddle point problem}\label{sec:time}
	
	In this section, we introduce the time semi-discretization of Problem~\ref{pro:problem_fictitious}. We consider a positive integer $N$ and we subdivide the time domain $[0,T]$ into $N$ equal sub-intervals of size $\dt$. For $n=0,\dots,N$, we set $t_n=n\dt$. We consider a backward approximation of time derivatives, i.e. given the generic function $f$
	\begin{equation}
		\derivative{f}{t}(t_{n+1})\approx\frac{f^{n+1}-f^{n}}{\dt},\qquad
		\frac{\partial^2f}{\partial t^2}(t_{n+1})\approx\frac{f^{n+1}-2f^{n}+f^{n-1}}{\dt^2},
	\end{equation}
	where $f^n=f(t_n)$.
	
	The time semi-discretization of Problem~\ref{pro:problem_fictitious} is the following.
	
	\begin{pro}
		\label{pro:problem_time_dis}
		Given $\u_0\in \Huo$ and $\X_0\in\Winftyd$, for $n=1,\dots,N$ find $(\u^n,p^n)\in\Huo\times \Ldo$, $\X^n\in \Hub$, and $\llambda^n\in\LL$, such that
		\begin{subequations}
			\begin{align}
				&\notag\rho_f\bigg(\frac{\u^{n+1}-\u^{n}}{\dt},\v\bigg)+a(\u^{n+1},\v)\\
				&\hspace*{2.7cm}-(\div\v,p^{n+1})+\c(\llambda^{n+1},\v(\X^{n}))=0&&\forall\v\in\Huo\\
				&(\div\u^{n+1},q)=0  &&\forall q\in \Ldo\\
				&\notag \dr \bigg(\frac{\X^{n+1}-2\X^n+\X^{n-1}}{\dt^2},\Y \bigg)_\B+(\PP(\F^{n+1}),\grads\Y )_\B\\
				&\hspace*{5.8cm}-\c(\llambda^{n+1},\Y)=0 &&\forall\Y\in \Hub\\
				& \label{eq:initiTime}\c\bigg(\mmu,\u^{n+1}(\X^{n})-\frac{\X^{n+1}-\X^{n}}{\dt} \bigg)=0  &&\forall\mmu\in\LL
			\end{align}
		\end{subequations}
	\end{pro}
	
	The coupling terms are treated in a semi-implicit way by taking into account the position of the solid body at the previous time instant as in $\c(\llambda^{n+1},\v(\X^{n}))$ and $\c(\mmu,\u^{n+1}(\X^{n}))$. Moreover, the term $\X^{-1}$ can be computed by solving the following equation
	\begin{equation*}
		\frac{\X^0-\X^{-1}}{\Delta t} = \u_0^s \quad \text{ in }\B.
	\end{equation*}
	
	We recall that the method is unconditionally stable, without any restriction on the choice of time step $\dt$, see~\cite[Prop. 3]{2015}.
	
	The semi-discrete Problem~\ref{pro:problem_time_dis} can be interpreted as a sequence of stationary saddle point problems, see~\cite{stat}. Before presenting such formulation, we assume that the solid material is described by a linear constitutive law, hence we set $\PP(\F)=\kappa\F$. At a fixed time instant, unknowns and physical constants can be rearranged as
	\begin{gather*}
		\Xbar = \X^{n},\quad\u = \u^{n+1},\quad p = p^{n+1},\quad\X = \frac{\X^{n+1}}{\dt},\quad\llambda = \llambda^{n+1},\\
		\f = \frac{\rho_f}{\dt}\u^n,\quad\g = \frac{\dr}{\Delta t^2}(2\X^n-\X^{n-1}),\quad\d = \frac{1}{\Delta t}\X^n,\quad
		\alpha = \frac{\rho_f}{\Delta t},\quad\beta = \frac{\dr}{\Delta t},\quad\gamma = \kappa\Delta t,
	\end{gather*}
	so that we obtain the following saddle point problem.
	\begin{pro}
		\label{pro:stationary_general}
		Let $\Xbar\in\Winftyd$ be invertible with Lipschitz inverse. Given
		$\f\in\LdOd$, $\g\in \LdBd$, and $\d\in\Hub$, find $(\u,p)\in\Huo\times\Ldo$,
		$\X\in\Hub$, and $\llambda\in\LL$ such that
		\begin{subequations}
			\begin{align}
				&\a_f(\u,\v)-(\div\v,p)_\Omega+\c(\llambda,\v(\Xbar))=(\f,\v)_\Omega&&
				\forall\v\in\Huo\\
				&(\div\u,q)_\Omega=0&&\forall q\in \Ldo\\
				&\a_s(\X,\Y)-\c(\llambda,\Y)=(\g,\Y)_\B&&\forall\Y\in\Hub\\
				&\c(\mmu,\X-\u(\Xbar))=\c(\mmu,\d)&&\forall\mmu\in\LL
			\end{align}
		\end{subequations}
		where
		\begin{equation}\label{eq:forme}
			\begin{aligned}
				&\a_f(\u,\v)=\alpha(\u,\v)_\Omega+(\nu\Grads(\u),\Grads(\v))_\Omega&& \forall\u,\v\in\Huo\\
				&\a_s(\X,\Y) = \beta(\X,\Y)_\B+\gamma (\grads\X,\grads\Y)_\B && \forall\X,\Y\in\Hub.
			\end{aligned}
		\end{equation}
	\end{pro}
	
	\subsection{Known theoretical results on continuous and discrete problems}\label{sec:analysis_c}
	
	In this section, we recall the theoretical results regarding Problem~\ref{pro:stationary_general} and its finite element discretization. We also introduce the notation used in the rest of the paper. We first define the product space $\VV = \Huo\times\Hub\times\LL\times\Ldo$ and the norm 
	\begin{equation}
		\|\V\|_\VV^2 = \norm{\v}^2_{1,\Omega} + \norm{\Y}^2_{1,\B} + \norm{\mmu}^2_{\LL}+\|q\|_{0,\Omega}^2.
	\end{equation}
	The tuple $\V=(\v,\Y,\mmu,q)$ is a generic element of $\VV$. For a fixed $\Xbar$, we introduce the bilinear form $\Bil:\VV\times\VV\rightarrow\R$ defined as
	\begin{equation}\label{eq:def_Bil}
		\begin{aligned}
			\Bil(\U,\V;\Xbar) = & \,\a_f(\u,\v) + \a_s(\X,\Y)\\
			&+ \c(\mmu,\X-\u(\Xbar)) - \c(\llambda,\Y-\v(\Xbar))\\
			&-(\div\v,p)_\Omega + (\div\u,q)_\Omega,
		\end{aligned}
	\end{equation}
	so that we can rewrite Problem~\ref{pro:stationary_general} as follows.
	
	\begin{pro}\label{pro:stat_gen_unified}
		Let $\Xbar\in\Winftyd$ be invertible with Lipschitz inverse, given $\f\in \LdOd$, $\g\in \LdBd$ and $\d\in\Hub$, find $\U\in\VV$ such that
		\begin{equation}
			\Bil(\U,\V;\Xbar) = (\f,\v)_\Omega + (\g,\Y)_\B + \c(\mmu,\d)\qquad\forall\V\in\VV.
		\end{equation}
	\end{pro}
	
	Combining, in a classical way, the well-posedness theory presented in~\cite{stat} with the results in~\cite{xu2003some}, $\Bil$~satisfies the following inf-sup condition.
	
	\begin{prop}
		There exists a positive constant $\eta$ such that
		\begin{equation*}
			\inf_{\U\in\VV}\sup_{\V\in\VV} \frac{\Bil(\U,\V;\Xbar)}{\|\U\|_\VV\|\V\|_{\VV}} \ge \eta.
		\end{equation*}
	\end{prop}
	
	In order to introduce the finite element discretization of
	Problem~\ref{pro:stat_gen_unified}, we consider two meshes $\T_h^\Omega$
	and $\T_h^\B$ decomposing $\Omega$ and $\B$, respectively, into triangles.
	We denote by $h_\Omega$ and $h_\B$
	the mesh sizes of $\T_h^\Omega$ and $\T_h^\B$, respectively.
	For fluid velocity and pressure, we choose a pair of discrete spaces
	$(\Vline_h,Q_h)\subset\Huo\times\Ldo$ satisfying the
	discrete inf-sup condition for the Stokes problem. For the solid unknowns,
	we consider $\Sline_h\subset\Hub$ and $\LL_h\subset\LL$. For simplicity, we
	assume that $\Sline_h=\LL_h$ even if more general cases have been studied
	(see~\cite{najwa}). From now on, we set
	\begin{equation}\label{eq:element}
		\begin{aligned}
			\Vline_h &= \{ \v\in\HuOm: \v_{\mid
				\tri}\in[\Pcal_1(\tri)]^2 \quad \forall \tri\in\T_{h/2}^\Omega \} \\
			Q_h &= \{q\in \Ldo\cap H^1(\Omega): q_{\mid \tri}\in \Pcal_1(\tri)\quad\forall \tri\in\T_{h}^\Omega\}\\
			\S_h &= \{\w\in\Hub:\w_{\mid \tri}\in [\Pcal_1(\tri)]^2\quad\forall \tri\in\mathcal{T}_{h}^\B\}\\
			\LL_h&=\{\mmu\in\Hub:\mmu_{\mid \tri}\in [\Pcal_1(\tri)]^2\quad\forall \tri\in\mathcal{T}_{h}^\B\}.
		\end{aligned}
	\end{equation}
	The pair of Stokes spaces $(\Vline_h,Q_h)$ corresponds to the Bercovier--Pironneau element introduced in~\cite{bercovier}, also known as $\Pcal_1$-iso-$\Pcal_2/\Pcal_1$. The velocity mesh $\T_{h/2}^\Omega$ is a refinement of $\T_h^\Omega$ obtained by partitioning each $\tri\in\T_{h}^\Omega$ into four sub triangles by connecting the three mid points of the edges of $\tri$.
	
	We make the following assumption on the considered meshes.
	\begin{assump}\label{assump:mesh}
		We assume that $\T_h^\Omega$ and $\T_h^\B$ are quasi-uniform meshes with size $h_\Omega,h_\B<1$.
	\end{assump}
	
	Notice that, when $\c=\c_0$ and if $\mmu\in\mathbf{L}^1_{loc}(\B)$, the duality pairing in $\Hub$ can be identified with the scalar product in $\LdBd$, so that we set
	\begin{equation}\label{eq:c_l2_inner}
		\c_0(\mmu_h,\Y_h) = (\mmu_h,\Y_h)_\B \quad \forall\mmu_h\in\LL_h,\forall\Y_h\in\Sline_h.
	\end{equation}
	Similarly to the continuous case, we now introduce the discrete product space $\VV_h=\Vline_h\times\Sline_h\times\LL_h\times Q_h$, so that the discrete counterpart of Problem~\ref{pro:stat_gen_unified} reads as follows.
	
	\begin{pro}\label{pro:dis_stat_gen_unified}
		Let $\Xbar\in\Winftyd$ be invertible with Lipschitz inverse, given $\f\in \LdOd$, $\g\in \LdBd$ and $\d\in\Hub$, find $\U_h\in\VV_h$ such that
		\begin{equation}
			\Bil(\U_h,\V_h;\Xbar) = (\f,\v_h)_\Omega + (\g,\Y_h)_\B - \c(\mmu_h,\d)\qquad\forall\V_h\in\VV_h.
		\end{equation}
	\end{pro}
	
	This discrete problem is well-posed thanks to the theory discussed
in~\cite{stat} combined with~\cite{xu2003some}. The following proposition
holds \fc true independently of how the fluid and the mapped solid meshes
intersect each other.\cf
	
	\begin{prop}\label{prop:infsup_discrete}
		There exists a positive constant $\tilde\eta$, independent of $h$, such that
		\begin{equation*}
			\inf_{\U_h\in\VV_h}\sup_{\V_h\in\VV_h} \frac{\Bil(\U_h,\V_h;\Xbar)}{\|\U_h\|_\VV\|\V_h\|_{\VV}} \ge \tilde\eta.
		\end{equation*}
	\end{prop}
	
	Therefore, the convergence theorem directly follows.
	
	\begin{teo}
		\label{theo:brezzi}
		Let $\Vline_h$ and $Q_h$ satisfy the usual compatibility
		conditions for the Stokes problem. If $\U=(\u,\X,\llambda,p)$ and
		$\U_h=(\u_h,\X_h,\llambda_h,p_h)$ denote the solution for the
		continuous and the discrete problem,  respectively, then the following error estimate holds true
		\begin{equation*}
			\begin{aligned}
				\norm{\u-\u_h}_{1,\Omega} &+ \norm{p-p_h }_{0,\Omega}+\norm{\X-\X_h }_{1,\B}+\norm{ \llambda-\llambda_h}_{\LL}\\
				&\hspace{-0.5cm}\leq C\big( \inf_{\v_h\in\Vline_h}\norm{\u-\v_h}_{1,\Omega} + \inf_{q_h\in Q_h}\norm{p-q_h}_{0,\Omega} + \inf_{\Y_h\in\Sline_h}\norm{\X-\Y_h}_{1,\B} + 
				\inf_{\mmu_h\in\LL_h}\norm{\llambda-\mmu_h}_{\LL} \big).
			\end{aligned}
		\end{equation*}
	\end{teo}
		
	\section{Computational aspects regarding the coupling term}\label{sec:coupling_term}
	
	In the previous sections, we recalled the main results regarding stability and well-posedness of continuous and discrete problems. We now focus on the coupling term, which needs particular care during the assembly phase. To begin with, we rewrite Problem~\ref{pro:dis_stat_gen_unified} in algebraic form as
	\begin{equation}\label{eq:lin_sys}
		\Sys\,\U_h = \mathcal{F}
	\end{equation}
	where	
	\begin{equation}\renewcommand{\arraystretch}{1.15}
		\label{eq:matrix}
		\Sys=\left[\begin{array}{@{}cccc@{}}
			\Amatr_f & 0 & \Cmatr_f^\top & -\Bmatr_f^\top\\ 
			0  & \Amatr_s & -\Cmatr_s^\top &0 \\
			-\Cmatr_f &  \Cmatr_s & 0 & 0\\
			\Bmatr_f & 0 & 0 & 0 \\
		\end{array}\right],\qquad
		\mathcal{F}=
		\left[ \begin{array}{c}
			\f\\
			\g\\
			\d\\
			\mathbf{0}
		\end{array}\right],
	\end{equation}
	and by abuse of notation $\U_h$ denotes the vector associated to the
	unknowns.
	Except for $\Cmatr_f$, each sub-matrix of $\Sys$ can be assembled exactly provided that a sufficiently precise quadrature rule is employed. Indeed, while $\Amatr_f$ and $\Bmatr_f$ are constructed on the fluid mesh $\T_h^\Omega$ and $\Amatr_s$,  $\Cmatr_s$ are defined on the solid mesh $\T_h^\B$, the interface matrix $\Cmatr_f$ combines the behavior of fluid and solid in the fictitious region of the fluid domain.
	
	$\Cmatr_f$ is the discrete counterpart of $\c(\mmu_h,\v_h(\Xbar))$, which involves the computation of integrals on $\T_h^\B$ of $\mmu_h\in\LL_h$ and $\v_h\in\Vline_h$, i.e. of functions defined on two independent meshes. More precisely,
	\begin{subequations}\label{eq:ch}
		\begin{align}
			&\label{eq:ch_l2}\c_0(\mmu_h,\v_h(\Xbar)) = \sum_{\tris\in\T^\B_h}\int_{\tris} \mmu_h\cdot \v_h(\Xbar)\,\ds,\\
			&\label{eq:ch_h1} \c_1(\mmu_h,\v_h(\Xbar)) 
			= \sum_{\tris\in\T^\B_h}\int_{\tris} \left(\mmu_h\cdot \v_h(\Xbar) + \grads\mmu_h:\grads\v_h(\Xbar)\right)\,\ds.
		\end{align}
	\end{subequations}	
	
	\begin{figure}
		\centering
		
		\vspace{3.5mm}
		
		\begin{overpic}[width=0.59\textwidth]{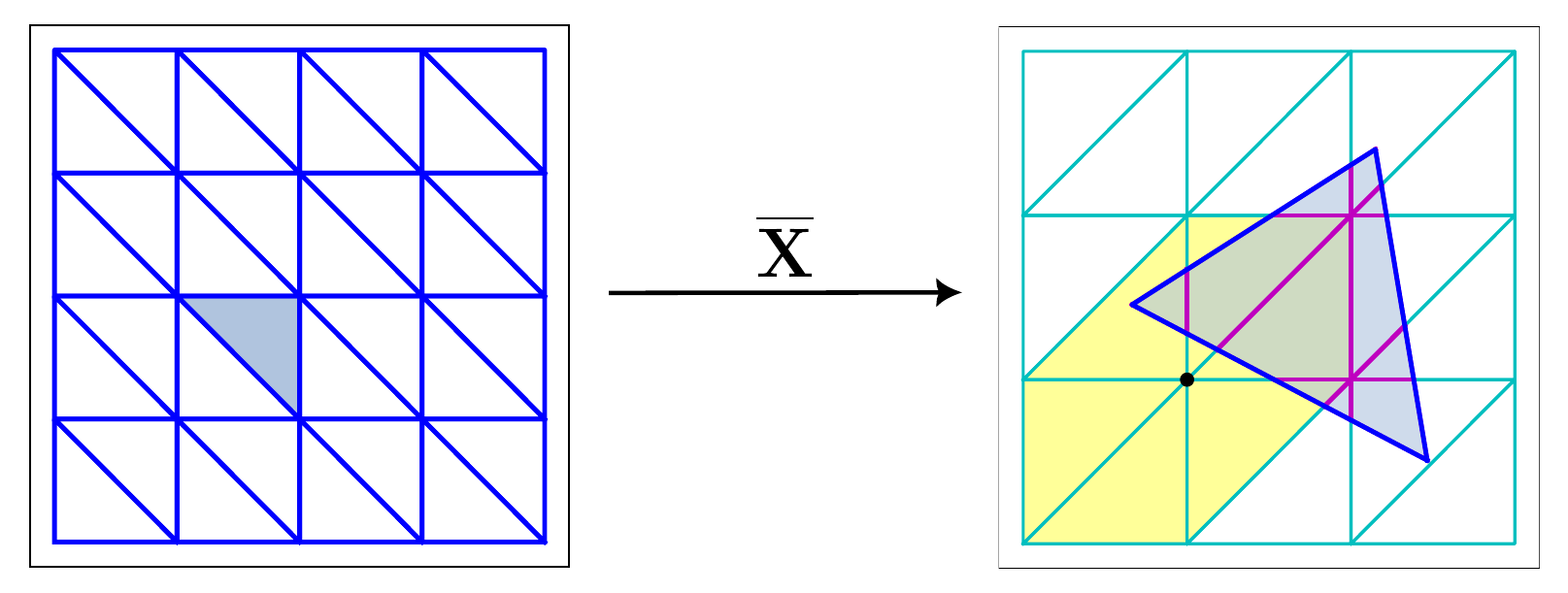}
			\put (-1.5,35) {$\B$}
		\end{overpic}
		\begin{overpic}[width=0.22\textwidth]{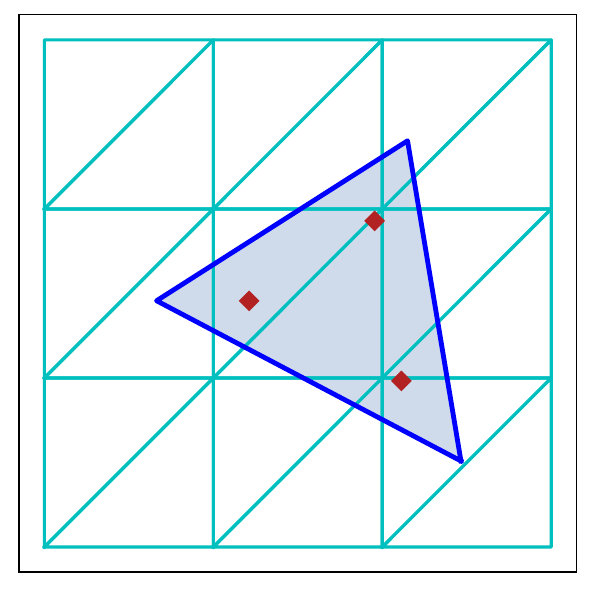}
			\put (-95,102) {Exact integration}
			\put (7,102) {Inexact integration}
		\end{overpic}
		\caption{ Mapping of a solid element $\tri_s\in\T_h^\B$ into the fluid mesh $T_{h/2}^\Omega$ showing the support mismatch of fluid basis function (yellow) with the immersed solid element and quadrature nodes for inexact integration.}
		\label{fig:coupling_term}
	\end{figure}	
	
	The composition of $\v_h$ with the map $\Xbar$ takes care of the
	actual configuration of the  solid. This fact is depicted in
	Figure~\ref{fig:coupling_term}, where a triangle $\tris\in\T_h^\B$ is
	immersed into the fluid mesh~$\T_{h/2}^\Omega$. Such operation introduces a mismatch
	between the mapped solid element and the support of fluid basis functions.
	An exhaustive discussion regarding the possible techniques for computing
	the integrals in~\eqref{eq:ch} can be found in~\cite{boffi2022interface}.
	The exact computation of the integrals can be done with the help of a composite quadrature
	rule defined on the intersection between the fluid and the mapped solid
	mesh (see purple lines in Figure~\ref{fig:coupling_term}) 
	which takes into consideration that $\v_h(\Xbar)$ is a piecewise
	polynomial in each element $\tris\in\T_h^\B$. 
	
	To this aim, we introduce a triangulation of the mapped elements by subdividing each intersection polygon into triangles, then we apply a suitable quadrature rule exact for the degree of the involved polynomials.
	This ``new mesh'' is thus an implementation tool and it does not appear in the formulation of the problem. We emphasize that the well-posedness of Problem~\ref{pro:dis_stat_gen_unified} is completely independent of $h$ (see Proposition~\ref{prop:infsup_discrete}). Therefore, the problem is naturally stable without the need of any artificial penalization term dealing with small cut cells.
	
	An alternative procedure consists in approximating the above integrals
	without taking into account that $\v_h(\Xbar)$ is piecewise polynomial in
	$\tri_s\in\T_h^\B$. In the rightmost picture of Figure~\ref{fig:coupling_term}, we show quadrature nodes of an inexact quadrature rule of order two. Given a generic
	triangle, we denote by $\{(\quadnode_k^0,\quadweigth_k^0)\}_{k=1}^{K_0}$
	and $\{(\quadnode_k^1,\quadweigth_k^1)\}_{k=1}^{K_1}$ quadrature rules for
	the $\LdBd$ scalar product and the $\LdBd$ scalar product of gradients,
	respectively. In this case, the expressions in~\eqref{eq:ch} become
	\begin{subequations}\label{eq:noint}
		\begin{equation}\label{eq:noint_l2}
			\c_{0,h}(\mmu_h,\v_h(\Xbar)) = \sum_{\tris\in\T^\B_h}\abs{\tris} \sum_{k=1}^{K_0} \quadweigth_k^0\,\mmu_h(\quadnode_k^0)\cdot \v_h(\Xbar(\quadnode_k^0)),
		\end{equation}
		{and}
		\begin{equation}\label{eq:noint_h1}
			\begin{aligned}
				&\c_{1,h}(\mmu_h,\v_h(\Xbar)) \\
				&\quad= \sum_{\tris\in\T^\B_h}\abs{\tris}\left(  \sum_{k=1}^{K_0} \quadweigth_k^0\,\mmu_h(\quadnode_k^0)\cdot \v_h(\Xbar(\quadnode_k^0)) + \sum_{k=1}^{K_1} \quadweigth_k^1\,\grads\mmu_h(\quadnode_k^1):\grads\v_h(\Xbar(\quadnode_k^1))\right).
			\end{aligned}
		\end{equation}
	\end{subequations}
	
	It is clear that, if the coupling term is assembled with the formulas given
	in~\eqref{eq:noint}, then a quadrature error is produced. This topic has
	been addressed in~\cite[Sec. 7]{BCG24}. We report here the associated
	quadrature error estimates.
	
	Taking into account the finite element spaces we are using, we make the
	following assumption.
	\begin{assump}\label{assump:quadrature}
		We assume that the quadrature rule
		$\{(\quadnode_k^0,\quadweigth_k^0)\}_{k=1}^{K_0}$ is exact for quadratic
		polynomials, whereas $\{(\quadnode_k^1,\quadweigth_k^1)\}_{k=1}^{K_1}$ is exact
		for constants.
	\end{assump}
	The quadrature errors can be bounded in terms of the mesh sizes as
	follows:
	\begin{equation}\label{eq:quaderr_l2}
		\begin{aligned}
			\abs{\c_0(\mmu_h,\v_h(\Xbar))-\c_{0,h}(\mmu_h,\v_h(\Xbar))}
			&\le C h_\B^{1/2}|\log h_\B| \norm{\mmu_h}_{\LL}\norm{\v_h}_{1,\Omega}
		\end{aligned}
	\end{equation}
	and
	\begin{equation}\label{eq:h1_est_fin}
		\abs{\c_1(\mmu_h,\v_h(\Xbar))-\c_{1,h}(\mmu_h,\v_h(\Xbar))}\le
		C \bigg( h_\B^{1/2}|\log h_\B| +
		\frac{h_\B}{h_\Omega} \bigg)
		\norm{\mmu_h}_{1,\B}\norm{\v_h}_{1,\Omega}.
	\end{equation}
	
	As we did in the continuous setting, we adopt the generic notation $\c_h$ for
	the inexact construction of both choices of coupling term.
	
	In the next section, we recall well-posedness results and convergence estimates for the problem with inexact coupling term.
	
	\subsection{Results on inexact coupling}\label{sec:inexact_analysis}
	
	Before rewriting Problem~\ref{pro:dis_stat_gen_unified} to take into account the inexact assembly of the coupling term, we define the inexact version of the bilinear form $\Bil$ and we study its properties. Given $\Xbar$, we introduce the new bilinear form $\Bil_h:\VV_h\times\VV_h\rightarrow\R$ defined as
	\begin{equation}\label{eq:def_Bil_h}
		\begin{aligned}
			\Bil_h(\U_h,\V_h;\Xbar) = & \,\a_f(\u_h,\v_h) + \a_s(\X_h,\Y_h)\\
			&+ \c(\mmu_h,\X_h)-\c_h(\mmu_h,\u_h(\Xbar)) - \c(\llambda_h,\Y_h)+\c_h(\llambda_h,\v_h(\Xbar))\\
			&-(\div\v_h,p_h)_\Omega + (\div\u_h,q_h)_\Omega.
		\end{aligned}
	\end{equation}
	
	The counterpart of Problem~\ref{pro:dis_stat_gen_unified} with inexact assembly of the coupling term reads as follows.	
	\begin{pro}\label{pro:ch_stat_gen_unified}
		Let $\Xbar\in\Winftyd$ be invertible with Lipschitz inverse, given $\f\in \LdOd$, $\g\in \LdBd$ and $\d\in\Hub$, find $\U_h^\star\in\VV_h$ such that
		\begin{equation}
			\Bil_h(\U_h^\star,\V_h;\Xbar) = (\f,\v_h)_\Omega + (\g,\Y_h)_\B - \c(\mmu_h,\d)\qquad\forall\V_h\in\VV_h.
		\end{equation}
	\end{pro}
	
	By replacing $\Bil$ with $\Bil_h$, we can rewrite the linear system~\eqref{eq:lin_sys} as
	\begin{equation}\label{eq:lin_sys_h}
		\Sys_h\,\U_h^\star = \mathcal{F}
	\end{equation}
	where
	\begin{equation}\renewcommand{\arraystretch}{1.5}
		\label{eq:matrix_h}
		\Sys_h=\left[\begin{array}{@{}cccc@{}}
			\Amatr_f & 0 & \Cmatr_{f,h}^\top & -\Bmatr_f^\top\\ 
			0  & \Amatr_s & -\Cmatr_s^\top &0 \\
			-\Cmatr_{f,h} &  \Cmatr_s & 0 & 0\\
			\Bmatr_f & 0 & 0 & 0 \\
		\end{array}\right].
	\end{equation}
	
	We recall the main results from~\cite{BCG24} about estimates for the
quadrature error as well as the analysis of
Problem~\ref{pro:ch_stat_gen_unified}. The well-posedness of
Problem~\ref{pro:ch_stat_gen_unified} is a consequence of the results
in~\cite{BCG24} combined with the theory in~\cite{xu2003some}. \fc Also in
this case, the following proposition states the stability of the problem
without any assumption on the intersection between the fluid and the mapped
solid meshes.\cf
	
	\begin{prop}\label{prop:infsup_discrete_ch}
		Under Assumptions~\ref{assump:mesh} and~\ref{assump:quadrature}, there exists a positive constant $\hat\eta$, independent of $h$, such that
		\begin{equation*}
			\inf_{\U_h\in\VV_h}\sup_{\V_h\in\VV_h} \frac{\Bil_h(\U_h,\V_h;\Xbar)}{\|\U_h\|_\VV\|\V_h\|_{\VV}} \ge \hat\eta.
		\end{equation*}
	\end{prop}
	
	Hence, taking into account the quadrature errors, the following convergence theorems hold in the case of $\c_{0,h}$ and $\c_{1,h}$, respectively.
	
	\begin{teo}\label{theo:c0h}
		Let $\c_h=\c_{0,h}$. Let $\U=(\u,\X,\llambda,p)$ be the solution of the continuous Problem~\ref{pro:stat_gen_unified}, $\U_h=(\u_h,\X_h,\llambda_h,p_h)$ the solution of the discrete Problem~\ref{pro:dis_stat_gen_unified}, and $\U_h^\star=(\u_h^\star,\X_h^\star,\llambda_h^\star,p_h^\star)$ the solution of Problem~\ref{pro:ch_stat_gen_unified}. Under Assumptions~\ref{assump:mesh} and~\ref{assump:quadrature}, the following estimate holds true
		\begin{equation*}
			\begin{aligned}
				&\norm{\u-\u_h^\star}_{1,\Omega} + \norm{p-p_h^\star}_{0,\Omega} + \norm{\X-\X_h^\star}_{1,\B} + \norm{\llambda-\llambda_h^\star}_{\LL} \\
				&\quad\le C\,\bigg(\inf_{\v_h\in\Vline_h}\norm{\u-\v_h}_{1,\Omega} + \inf_{q_h\in Q_h}\norm{p-q_h}_{0,\Omega} + \inf_{\Y_h\in\Sline_h}\norm{\X-\Y_h}_{1,\B} + \inf_{\mmu_h\in\LL_h}\norm{\llambda-\mmu_h}_{\LL} \\
				&\qquad\qquad+h_\B^{1/2}|\log h_\B|\|\u_h\|_{1,\Omega}+ h_\B^{1/2}|\log h_\B|\|\llambda_h\|_{\LL}\bigg).
			\end{aligned}
		\end{equation*}
	\end{teo}
	
	\begin{teo}\label{theo:c1h}
		Let $\c_h=\c_{1,h}$. Let $\U=(\u,\X,\llambda,p)$ be the solution of the continuous Problem~\ref{pro:stat_gen_unified}, $\U_h=(\u_h,\X_h,\llambda_h,p_h)$ the solution of the discrete Problem~\ref{pro:dis_stat_gen_unified}, and $\U_h^\star=(\u_h^\star,\X_h^\star,\llambda_h^\star,p_h^\star)$ the solution of Problem~\ref{pro:ch_stat_gen_unified}. Under Assumptions~\ref{assump:mesh} and~\ref{assump:quadrature}, the following estimate holds true
		\begin{equation*}
			\begin{aligned}
				&\norm{\u-\u_h^\star}_{1,\Omega} + \norm{p-p_h^\star}_{0,\Omega} + \norm{\X-\X_h^\star}_{1,\B} + \norm{\llambda-\llambda_h^\star}_{1,\B} \\
				&\qquad\le C\,\bigg(\inf_{\v_h\in\Vline_h}\norm{\u-\v_h}_{1,\Omega} + \inf_{q_h\in Q_h}\norm{p-q_h}_{0,\Omega} + \inf_{\Y_h\in\Sline_h}\norm{\X-\Y_h}_{1,\B} + \inf_{\mmu_h\in\LL_h}\norm{\llambda-\mmu_h}_{1,\B} \\
				&\quad\qquad\qquad+\big( h_\B^{1/2}|\log h_\B| + \frac{h_\B}{h_\Omega} \big)(\|\u_h\|_{1,\Omega}+\|\llambda_h\|_{1,\B})\bigg).
			\end{aligned}
		\end{equation*}
	\end{teo}
	
	We observe that the inexact assembly of the coupling form $\c_0$ gives a convergent method whenever the mesh sizes $h_\Omega$ and $h_\B$ decay to zero. On the other hand, if $\c_1$ is considered, we obtain an optimal method under the additional condition that $h_\B/h_\Omega$ tends to zero fast enough. This behavior was confirmed by the numerical tests we presented in~\cite{boffi2022interface,boffi2022parallel,BCG24}.
	
	In the statement of Problem~\ref{pro:ch_stat_gen_unified}, without
	affecting the generality of the results, we assumed
	that the three terms on the right hand side are computed exactly. 
	Actually, the error produced by inexact integration can be easily estimated
	by standard arguments as in~\cite[Chap. 4, Sect. 4.1]{ciarlet2002finite}. The interested reader
	can find these results in~\cite[Lemma 1]{BCG24}.
	
	\section{Estimate of the condition number}\label{sec:conditioning}
	
	In this section, we study the condition number of the matrices $\Sys$ and $\Sys_h$ associated with Problem~\ref{pro:dis_stat_gen_unified} and Problem~\ref{pro:ch_stat_gen_unified}, respectively.
	
	A general framework for studying the condition number of linear systems arising from finite element approximations was presented by Ern and Guermond in~\cite{ern-guermond}.  The Euclidean condition number of our system behaves as described by the following theorem, where the notation $\Sys_\square$ indicates both $\Sys$ and $\Sys_h$, as well as $\Bil_\square$ stands for both $\Bil$ and $\Bil_h$.
	
	\begin{teo}\label{teo:ern_guermond}
		Let $C_\star,\, C^\star$ be two positive constants, independent of the mesh sizes, arising from the equivalence between the $L^2$ norm of a finite element function and the Euclidean norm of the corresponding vector of degrees of freedom (see~\cite[Sect. 2.3]{ern-guermond}). Then, the following bounds hold
		\begin{equation}
			C_\star \frac\omega\tau \le \cond{\Sys_\square} \le C^\star \frac\omega\tau,
		\end{equation}
		where
		\begin{equation}\label{eq:omega_tau}
			\omega = \sup_{\U_h\in\VV_h}\sup_{\V_h\in\VV_h} \frac{\Bil_\square(\U_h,\V_h;\Xbar)}{\|\U_h\|_{0}\|\V_h\|_{0}},\qquad
			\tau = \inf_{\U_h\in\VV_h}\sup_{\V_h\in\VV_h} \frac{\Bil_\square(\U_h,\V_h;\Xbar)}{\|\U_h\|_{0}\|\V_h\|_{0}}
		\end{equation}
		and
		\begin{equation}\label{eq:l2_VV}
			\|\V_h\|_{0}^2 = \|\v_h\|_{0,\Omega}^2 + \|\Y_h\|_{0,\B}^2 + \|\mmu_h\|_{0,\B}^2 + \|q_h\|_{0,\Omega}^2.
		\end{equation}
	\end{teo}
	
	In the remainder of this section, in order to bound $\cond{\Sys_\square}$ in terms of $h_\Omega$ and $h_\B$, we study $\tau$ and $\omega$.
	The structure of the proof is the same for both the exact and inexact computation of the coupling term and for the two choices of $\c$, and can be organized into the following three steps:
	
	\begin{enumerate}[i)]
		\item we prove an $h$-dependent norm equivalence between the $L^2$-norm defined in \eqref{eq:l2_VV} and~$\|\cdot\|_\VV$,
		\item we prove bounds for $\omega$ and $\tau$ by exploiting the well-posedness of the problem and continuity of the bilinear form $\Bil_\square$,
		\item we apply Theorem~\ref{teo:ern_guermond}.
	\end{enumerate}
	
	The first step, which does not depend on the assembly technique for the coupling term, is presented in the following lemma.
	
	\begin{lem}\label{lem:equiv_norme}
		Under Assumption~\ref{assump:mesh}, there exist two positive constants $\underline{C},\,\overline{C}$ such that for all $\V_h\in\VV_h$ we have
		\begin{equation*}
			\underline{C} h_\B^{2-2\ell} \|\V_h\|_0^2 \le \|\V_h\|_\VV^2 \le \overline{C} ( h_\Omega^{-2} + h_\B^{-2}) \|\V_h\|_0^2\qquad\text{for $\LL=\LL_\ell$},
		\end{equation*}
		where $\ell=0,1$.
	\end{lem}
	
	\begin{proof}
		We take first $\LL=\LL_0$. We start proving the upper bound. By taking into account the following standard inverse inequalities (see~\cite{ciarlet2002finite})
		\begin{equation}\label{eq:inv_ineq}
			\begin{aligned}
				&\|\v_h\|_{1,\Omega}\le Ch_{\Omega}^{-1}\|\v_h\|_{0,\Omega}&&\forall\v_h\in\Vline_h,\\
				&\|\Y_h\|_{1,\B}\le Ch_{\B}^{-1}\|\Y_h\|_{0,\B}&&\forall\Y_h\in\Sline_h,\\
				&\|\mmu_h\|_{0,\B} \le Ch_{\B}^{-1}\|\mmu_h\|_{\LL_0}&&\forall\mmu_h\in\LL_h,\\
			\end{aligned}
		\end{equation}		
		the inclusion $\LdBd\subset\LL_0$, and the assumption on the mesh sizes, we can write
		\begin{equation}\label{eq:c0_upper}
			\|\V_h\|_\VV^2 \le C\,(h_{\Omega}^{-2}\|\v_h\|_{0,\Omega}^2 + h_\B^{-2}\|\Y_h\|_{0,\B}^2 + \|\mmu_h\|_{0,\B}^2 + \|q\|_{0,\Omega}^2)\le C\,(h_{\Omega}^{-2}+ h_\B^{-2})\|\V_h\|_0^2,
		\end{equation}
		so that the bound is proved.
		For the lower bound, we exploit an inverse inequality relating $\LL_0$ with $\LdBd$, so that we have
		\begin{equation}
			\begin{aligned}
				\|\V_h\|_\VV^2 &\ge \|\v_h\|_{1,\Omega}^2 + \|\Y_h\|_{1,\B}^2 +  h_\B^{2}\|\mmu_h\|_{0,\B}^2 + \|q\|_{0,\Omega}^2\\
				&\ge C\,h_\B^{2}(\|\v_h\|_{1,\Omega}^2+\|\Y_h\|_{1,\B}^2+\|\mmu_h\|_{0,\B}^2 + \|q\|_{0,\Omega}^2)\\
				&\ge C\,h_\B^{2}\|\V_h\|_0^2.
			\end{aligned}
		\end{equation}
		
		Now, we consider $\LL=\LL_1$. The proof of the lower bound is trivial and, for the upper bound, we exploit again the inverse inequalities in~\eqref{eq:inv_ineq}. By working as in~\eqref{eq:c0_upper}, the proof is concluded.
	\end{proof}

	\subsection{Condition number in case of exact coupling}	
	Let us consider the case of exact integration of the coupling term. In the following proposition we estimate $\omega$ and $\tau$.
	
	\begin{prop}\label{prop:tau_omega_ex}
		Under Assumption~\ref{assump:mesh}, there exist two positive constants $\underline{\gamma},\,\overline{\gamma}$ such that the following inequalities hold true
		\begin{equation*}
			\omega=\sup_{\U_h\in\VV_h}\sup_{\V_h\in\VV_h} \frac{\Bil(\U_h,\V_h;\Xbar)}{\|\U_h\|_{0}\|\V_h\|_{0}} \le \overline{\gamma}\,( h_\Omega^{-2} + h_\B^{-2}),\qquad\quad
			\tau=\inf_{\U_h\in\VV_h}\sup_{\V_h\in\VV_h} \frac{\Bil(\U_h,\V_h;\Xbar)}{\|\U_h\|_{0}\|\V_h\|_{0}} \ge \underline{\gamma}\,h_\B^{2-2\ell},
		\end{equation*}
		where $\ell=0,1$ is associated to the choice of $\c$.
	\end{prop}
	
	\begin{proof}
		
		The bound of $\omega$ is a consequence of the continuity of $\Bil$, which is easily obtained by combining the continuity of each term at the right hand side of~\eqref{eq:def_Bil} with the Cauchy--Schwarz inequality
		\begin{equation}\label{eq:B_cont_cor}
			\Bil(\U,\V;\Xbar) \le M\|\U\|_\VV\|\V\|_\VV \qquad\forall\U,\V\in\VV.
		\end{equation}
		Using the upper bound in Lemma~\ref{lem:equiv_norme}, we find
		\begin{equation*}
			\sup_{\U_h\in\VV_h}\sup_{\V_h\in\VV_h} \frac{\Bil(\U_h,\V_h;\Xbar)}{\|\U_h\|_{0}\|\V_h\|_{0}} = \sup_{\U_h\in\VV_h}\sup_{\V_h\in\VV_h} \frac{\Bil(\U_h,\V_h;\Xbar)}{\|\U_h\|_{\VV}\|\V_h\|_{\VV}} \cdot \frac{\|\U_h\|_{\VV}}{\|\U_h\|_{0}}\cdot\frac{\|\V_h\|_{\VV}}{\|\V_h\|_{0}}\le C ( h_\Omega^{-2} + h_\B^{-2}).
		\end{equation*}
		On the other hand, Proposition~\ref{prop:infsup_discrete} and the lower bound in Lemma~\ref{lem:equiv_norme} give us
		\begin{equation*}
			\inf_{\U_h\in\VV_h}\sup_{\V_h\in\VV_h} \frac{\Bil(\U_h,\V_h;\Xbar)}{\|\U_h\|_{0}\|\V_h\|_{0}} =
			\inf_{\U_h\in\VV_h}\sup_{\V_h\in\VV_h} \frac{\Bil(\U_h,\V_h;\Xbar)}{\|\U_h\|_{\VV}\|\V_h\|_{\VV}} \cdot \frac{\|\U_h\|_{\VV}}{\|\U_h\|_{0}}\cdot\frac{\|\V_h\|_{\VV}}{\|\V_h\|_{0}}\ge C h_\B^{2-2\ell}.
		\end{equation*}
	\end{proof}
	
	\subsection{Condition number in case of inexact coupling}	
	
	We move to the case of inexact computation of the coupling term. We follow the same lines of the previous section. We stress the fact that one of the main arguments consists in the continuity of the bilinear form $\Bil_h$ and, thus, of the bilinear form $\c_h$.
	
	\begin{prop}\label{prop:continuity_ch}
		Under Assumptions~\ref{assump:mesh} and~\ref{assump:quadrature} and given $\Xbar\in\Winftyd$, there exists a constant $C$, independent of $h$, such that the following inequality holds true
		\begin{equation}
			\c_{h}(\mmu_h,\v_h(\Xbar)) \le C  \|\mmu_h\|_\LL\,\|\v_h\|_{1,\Omega}\qquad\forall\mmu_h\in\LL_h,\,\v_h\in\Vline_h.
		\end{equation}
	\end{prop}
	
	\begin{proof}
		\textsl{Case 1.} We set $\c_h=\c_{0,h}$. The continuity of
		$\c_{0,h}(\mmu_h,\v_h(\Xbar))$ is a direct consequence of the
		continuity of $\c_0$ combined with~\eqref{eq:quaderr_l2} and the
		inclusion $\Xbar(\B)\subset\Omega$; indeed it holds
		\[
		\begin{aligned}
			\c_{0,h}(\mmu_h,\v_h(\Xbar)) &= \c_0(\mmu_h,\v_h(\Xbar)) +
			\big[\c_{0,h}(\mmu_h,\v_h(\Xbar)) -
			\c_0(\mmu_h,\v_h(\Xbar))\big]\\
			& \le C \big( \|\mmu_h\|_\LL\,\|\v_h\|_{1,\Omega} +
			h_\B^{1/2}|\log h_\B| \norm{\mmu_h}_{\LL}\norm{\v_h}_{1,\Omega}
			\big)\\
			& \le C \big(1+h_\B^{1/2}|\log h_\B|\big)\norm{\mmu_h}_{\LL}\norm{\v_h}_{1,\Omega}.
		\end{aligned}
		\]
		
		\textsl{Case 2.} We consider $\c_h=\c_{1,h}$. Since now $\LL=\Hub$ we
		can bound directly the integral associated to the $\LdBd$ scalar
		product by applying the discrete Cauchy--Schwarz inequality and the
		accuracy of the quadrature rule, hence
		\begin{equation}\label{eq:procedure}
			\begin{aligned}
				\sum_{k=1}^{K_0} \quadweigth_k^0\,\mmu_h(\quadnode_k^0)\cdot&
				\v_h(\Xbar(\quadnode_k^0))
				\le \left(\sum_{k=1}^{K_0}
				\quadweigth_k^0\,\mmu_h^2(\quadnode_k^0)\right)^{1/2}\left(\sum_{k=1}^{K_0}
				\quadweigth_k^0\,\v_h^2(\Xbar(\quadnode_k^0))\right)^{1/2}\\
				&=\|\mmu_h\|_{0,\tri}\left(\sum_{k=1}^{K_0}
				\quadweigth_k^0\,\v_h^2(\Xbar(\quadnode_k^0))\right)^{1/2}
				\le C|\tri|^{1/2}\|\mmu_h\|_{0,\tri}\|\v_h(\Xbar)\|_{\infty,\tri}.
			\end{aligned}
		\end{equation}
		We combine~\eqref{eq:procedure} with the inverse inequality
		$\|\v_h(\Xbar)\|_{\infty,\tri}\le C_I h_\tri^{-1}\|\v_h(\Xbar)\|_{0,\tri}$, so
		that we find
		\begin{equation}\label{eq:procedure_2}
			\sum_{k=1}^{K_0} \quadweigth_k^0\,\mmu_h(\quadnode_k^0)\cdot
			\v_h(\Xbar(\quadnode_k^0))
			\le C|\tri|^{1/2}h_\tri^{-1}\|\mmu_h\|_{0,\tri}\|\v_h(\Xbar)\|_{0,\tri}
			\le C\|\mmu_h\|_{0,\tri}\|\v_h(\Xbar)\|_{0,\tri}.
		\end{equation}	
		
		We now focus on the term involving gradients. By applying the same
		procedure as in~\eqref{eq:procedure} and~\eqref{eq:procedure_2}, we obtain
		\begin{equation}\label{eq:c1h_t}
			\sum_{k=1}^{K_1} \quadweigth_k^1\,\grads\mmu_h(\quadnode_k^1):\grads\v_h(\Xbar(\quadnode_k^1))
			\le C\,\|\grads\mmu_h\|_{0,\tri}\|\grads\v_h(\Xbar)\|_{0,\tri}.
		\end{equation}
		We sum up~\eqref{eq:procedure_2} and~\eqref{eq:c1h_t} over all $\tri\in\T_h^\B$. The inclusion $\Xbar(\B)\subset\Omega$, together with the Cauchy--Schwarz inequality, gives
		\[
		\c_{1,h}(\mmu_h,\v_h(\Xbar)) \le
		C\,\left(\|\mmu_h\|_{0,\B}\,\|\v_h(\Xbar)\|_{0,\B}
		+\|\grads\mmu_h\|_{0,\B}\|\grads\v_h(\Xbar)\|_{0,\B}\right)
		\le C\,\|\mmu_h\|_{\LL}\,\|\v_h\|_{1,\Omega},
		\]
		so that the result is proved.
	\end{proof}
	
	The continuity of $\Bil_h$ is a direct consequence of the above proposition, and is stated as follows.
	
	\begin{prop}\label{prop:Bh_cont_cor}
		Under Assumptions~\ref{assump:mesh} and~\ref{assump:quadrature}, there exists a constant $M>0$, independent of $h$, such that
		\begin{equation}\label{eq:Bh_cont_cor_ch1}
			\Bil_h(\U_h,\V_h;\Xbar)\le M \|\U_h\|_\VV\|\V_h\|_\VV \qquad \forall\U_h,\V_h\in\VV_h.\\
		\end{equation}	
		
	\end{prop}
	
	\begin{proof}		
		The result is obtained by combining Proposition~\ref{prop:continuity_ch} with the continuity of the other terms in definition~\eqref{eq:def_Bil_h}.
	\end{proof}
	
	We now estimate $\omega$ and $\tau$ (see~\eqref{eq:omega_tau}) for $\Bil_h$.
	
	\begin{prop}\label{prop:tau_omega_inex}
		Under Assumptions~\ref{assump:mesh} and~\ref{assump:quadrature}, there exist two positive constants, denoted again by $\underline{\gamma},\,\overline{\gamma}$, such that the following inequalities hold true
		\begin{equation*}
			\omega=\sup_{\U_h\in\VV_h}\sup_{\V_h\in\VV_h} \frac{\Bil_h(\U_h,\V_h;\Xbar)}{\|\U_h\|_{0}\|\V_h\|_{0}} \le \overline{\gamma}\,( h_\Omega^{-2} + h_\B^{-2}),\qquad\quad
			\tau=\inf_{\U_h\in\VV_h}\sup_{\V_h\in\VV_h} \frac{\Bil_h(\U_h,\V_h;\Xbar)}{\|\U_h\|_{0}\|\V_h\|_{0}} \ge \underline{\gamma}\,h_\B^{2-2\ell},
		\end{equation*}
		where $\ell=0,1$ is associated to the choice of $\c_h$.
	\end{prop}	
	
	\begin{proof}
		The proof is similar to the one of Proposition~\ref{prop:tau_omega_ex}.
	\end{proof}
	
	\subsection{Main result}
	
	We conclude this section by presenting the estimate of the condition number in both cases of exact and inexact integration of the coupling term.
	
	\begin{teo}\label{theo:conditioning}
		Under Assumption~\ref{assump:mesh}, for $\ell=0,1$ depending on the definition of $\c$, we have
		\begin{equation*}
			\cond{\Sys} \le C h_\B^{-2+2\ell}( h_\Omega^{-2}+h_\B^{-2}).
		\end{equation*}
		Moreover, with the additional hypotheses stated in Assumption~\ref{assump:quadrature}, we have
		\begin{equation*}
			\cond{\Sys_h} \le C h_\B^{-2+2\ell}( h_\Omega^{-2}+h_\B^{-2}).
		\end{equation*}
	\end{teo}
	
	\begin{proof}
		The result is a direct consequence of Theorem~\ref{teo:ern_guermond} combined with Proposition~\ref{prop:tau_omega_ex} and  Proposition~\ref{prop:tau_omega_inex} for the exact and inexact integration, respectively.
	\end{proof}
	
	\red
	
	\begin{rem}\label{rem:cond_h}
		We observe that, if $h = h_\B\approx h_\Omega$, then Theorem~\ref{theo:conditioning} gives the following bounds
		\begin{equation}
			\begin{aligned}
				&\cond{\Sys_\square} \le C h^{-4} &&\quad \text{if } \c=\c_0 \text{ or } \c_h=\c_{0,h},\\
				&\cond{\Sys_\square} \le C h^{-2} &&\quad\text{if }\c=\c_1 \text{ or } \c_h=\c_{1,h}.
			\end{aligned}
		\end{equation} 
	\end{rem}
	
	\nored
	
	\begin{rem}
		The condition number only depends on the mesh sizes $h_\Omega$ and $h_\B$, while its growth rate depends on the choice of the coupling term, regardless of the implementation technique. We observed in Section~\ref{sec:coupling_term} that the well-posedness of the discrete problem is independent of the size of the intersected cell. This fact is still true for the condition number, thus its behavior does not depend on the position of the interface as we show later in our numerical experiments.
	\end{rem}
	
	\begin{rem}
		
		We observe that the choice of coupling form $\c$ and related assembly technique should be done with particular care. On the one hand, looking at Theorems~\ref{theo:c0h} and~\ref{theo:c1h}, we see that choosing $\c_0$ instead of $\c_1$ may be more convenient from the computational point of view: indeed, $\c_0$ can be cheaply assembled using the inexact procedure, without negative effects on the accuracy. On the other hand, Theorem~\ref{theo:conditioning} says that the condition number increases faster in $h$ when $\c_0$ or $\c_{0,h}$ are considered.
		
	\end{rem}
	
	\section{Numerical tests}\label{sec:tests}
	
	This section collects a wide range of numerical tests confirming our theoretical results for the condition number and the convergence of the method. We first describe the problem we are going to solve throughout the entire section.
	
	\begin{pro}
		\label{pro:tests}
		Given $\Xbar\in\Winftyd$ invertible with Lipschitz inverse and given
		$\f\in\LdOd$, ${\g\in \LdBd}$, and $\d\in\Hub$, find $(\u,p)\in\Huo\times\Ldo$,
		$\X\in\Hub$, and $\llambda\in\LL$ such that
		\begin{subequations}
			\begin{align}
				&\a_f(\u,\v)-(\div\v,p)_\Omega+\c(\llambda,\v(\Xbar))=(\f,\v)_\Omega&&
				\forall\v\in\Huo\\
				&(\div\u,q)_\Omega=0&&\forall q\in \Ldo\\
				&\a_s(\X,\Y)-\c(\llambda,\Y)=(\g,\Y)_\B&&\forall\Y\in\Hub\\
				&\c(\mmu,\X-\u(\Xbar))=\c(\mmu,\d)&&\forall\mmu\in\LL
			\end{align}
		\end{subequations}
		where $\a_f$ and $\a_s$ are defined in~\eqref{eq:forme}, with $\alpha = \frac{\rho_f}{\Delta t}$, $\beta = \frac{\dr}{\Delta t}$, $\gamma = \kappa\Delta t$.
	\end{pro}
	
	In each test, we compute the right hand sides $\f$, $\g$ and $\d$ starting from manufactured solutions. The viscosity $\nu$ is always equal to one, while the values of $\alpha$, $\beta$ and $\gamma$ differ from test to test.
	
	Since the definition of $\LL$ changes with the choice of coupling term $\c$, we point out that if $\c=\c_1$ then the error computation for $\llambda$ is done in the standard $\Hub$ norm. On the other hand, when $\c=\c_0$ the error is computed by using the norm of the dual space $\Hubd$: such norm is evaluated by solving the auxiliary problem $-\Delta\bm{\psi}+\bm{\psi}=\llambda-\llambda_h$ with homogeneous Neumann boundary conditions and by computing the $H^1$-norm of $\bm{\psi}$.
	
	We consider four different geometrical configurations: square, disk, flower, and stretched annulus. The first one is treated with more details, whereas the results for the other three are presented all together.
	
	\subsection{The shifted square}
	
	The goal of this first test is to assess the (in)dependence of the method from the interface position. We consider an immersed square placed at different positions with respect to the fluid mesh. We show that condition number and rate of convergence do not depend either on the interface position or on the presence of small intersections.
	
	\begin{figure}
		\centering
		\subfloat[Shift $\sigma$]{\includegraphics[width=0.207\linewidth]{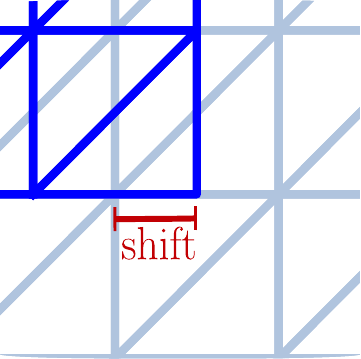}}
		\subfloat[$\sigma<0$]{\includegraphics[width=0.275\linewidth]{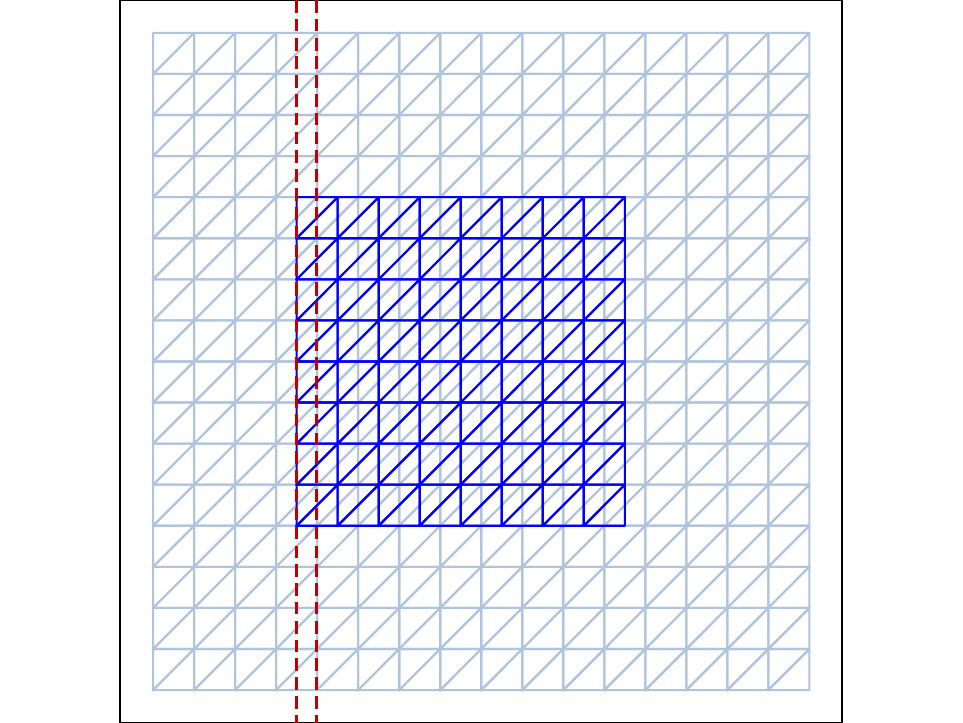}}
		\subfloat[$\sigma=0$]{\includegraphics[width=0.22\linewidth]{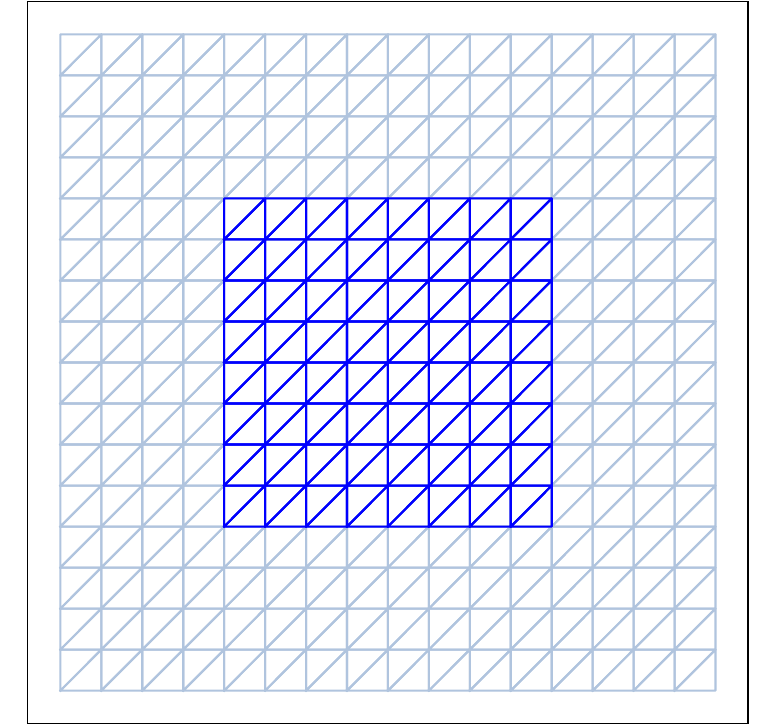}}
		\subfloat[$\sigma>0$]{\includegraphics[width=0.25\linewidth]{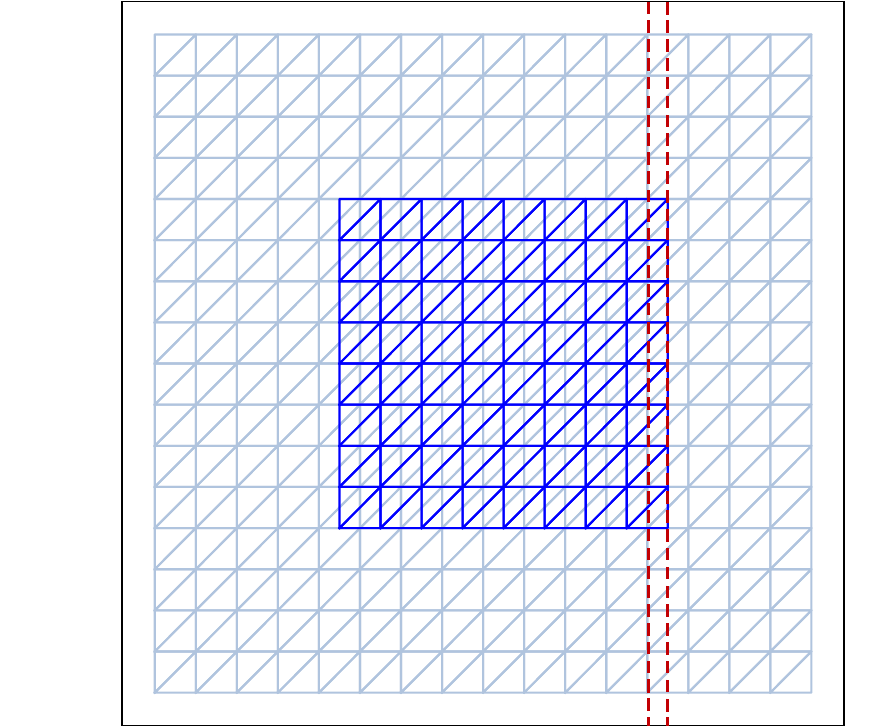}}
		\caption{The geometrical configuration of the shifted square. The value of $\sigma$ gives the shift between the fluid mesh and the mapped solid one.}
		\label{fig:shift}
	\end{figure}
	
	We consider $\Omega=[-2,2]^2$ and $\B=[0,1]^2$. $\Xbar$ is computed in such a way that the actual position of the solid is given by the square $\Os=[-1+\sigma,1+\sigma]\times[-1,1]$, where $\sigma$ is a shift parameter denoting the distance between the fluid and solid meshes. If $\sigma=0$, then $\T_{h/2}^\Omega$ and the mapped solid mesh $\Xbar(\T_h^\B)$ perfectly match. If $\sigma>0$ ($\sigma<0$), then the square $\Os$ is shifted to the right (left). A sketch of this geometrical configuration is shown in Figure~\ref{fig:shift}. Clearly, for small values of the shift, we have small intersected cells. From our theory, the presence of small cut cells should not affect the optimality of the method.
	
	We set $\alpha = \beta = 0$, so that no mass terms are present, and $\gamma=1$. We choose the following solution
	\begin{equation}
		\begin{aligned}
			&\u(x,y) = \curl\big( (4-x^2)^2(4-y^2)^2 \big)\\
			&p(x,y) = 150\,\sin(x)&&\qquad\text{for }\x=(x,y)\in\Omega\\
			&\X(s_1,s_2) = \u(s_1,s_2)\\
			&\llambda(s_1,s_2) = (e^{s_1},e^{s_2})&&\qquad\text{for }\s=(s_1,s_2)\in\B,
		\end{aligned}
	\end{equation}
	and we solve the problem for $\sigma=0$ and $\sigma=\pm10^{-j}$ with $j=3,4,\dots,15$.

	In the first test, both fluid and solid domains are discretized by fixed $256\times256$ uniform triangular meshes. The behavior of the condition number with respect to the shift $\sigma$ is reported in Figure~\ref{fig:shift_cond} for all choices of coupling term. Looking at the scale of the $y$ axis, it is clear that the condition number is not significantly affected by the value of $\sigma$. On the other hand, the condition number associated to the choice of $\c_0$ and $\c_{0,h}$ is five orders of magnitude larger than for $\c_1$ and $\c_{1,h}$. \red This is consistent with the theoretical analysis presented in Section~\ref{sec:conditioning}. \nored
	
	In Figure~\ref{fig:shift_err}, we report the value of the errors in terms of the shift. Also in this case, there is no influence of $\sigma$ for all possible choices of coupling term. We observe that, when $\c=\c_{1,h}$, the error for $\llambda$ is affected by the presence of the term $h_\B/h_\Omega$ in the quadrature error estimate as already discussed in our previous works~\cite{boffi2022interface,BCG24}.
	
	\begin{figure}
		\centering
		\includegraphics[width=0.45\linewidth]{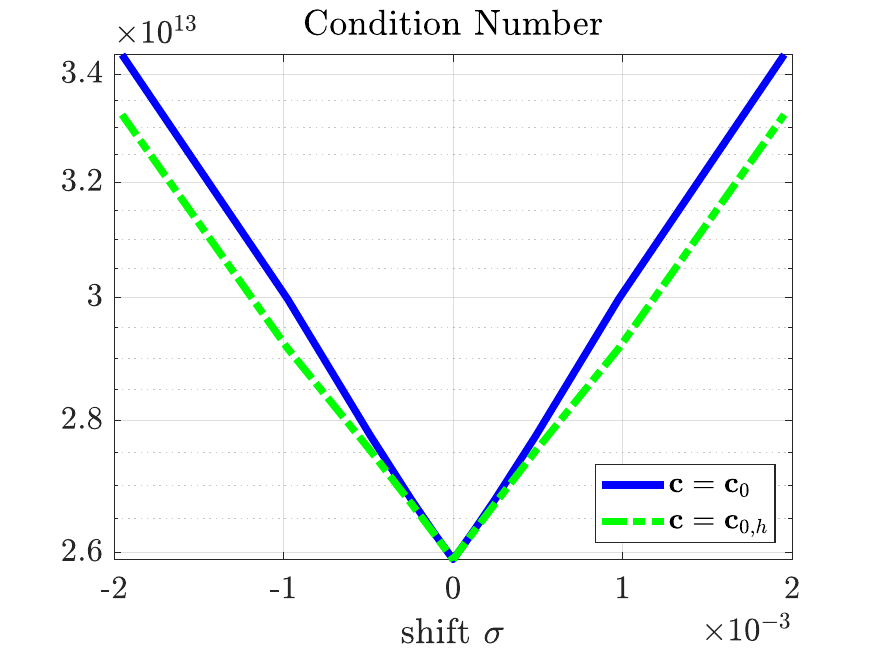}
		\includegraphics[width=0.45\linewidth]{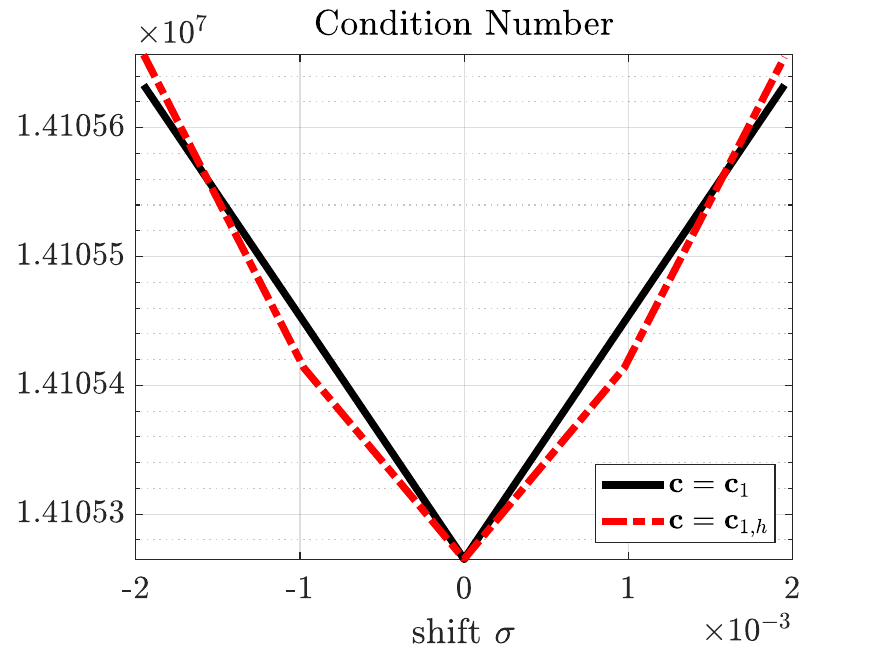}
		\caption{Condition number of the shifted square problem as a function of the shift $\sigma$. The value of $\sigma$ does not affect the condition number. For $\c=\c_0$ or $\c_{0,h}$, $\cond{\Sys_\square}$ is five order of magnitude larger than for $\c=\c_1$ or $\c_{1,h}$.}
		\label{fig:shift_cond}
	\end{figure}
	
	\begin{figure}
		\centering
		\includegraphics[width=0.24\linewidth]{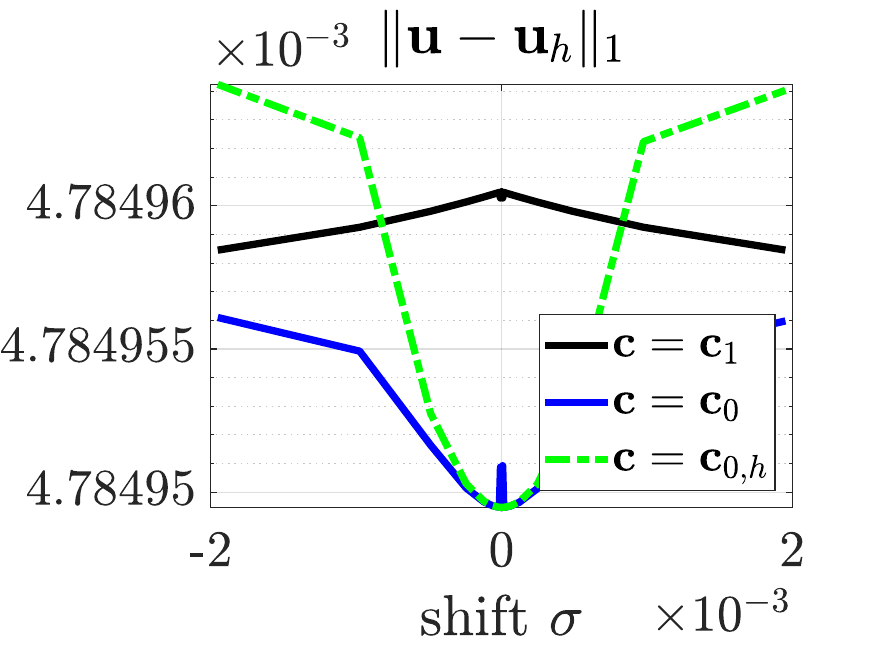}
		\includegraphics[width=0.24\linewidth]{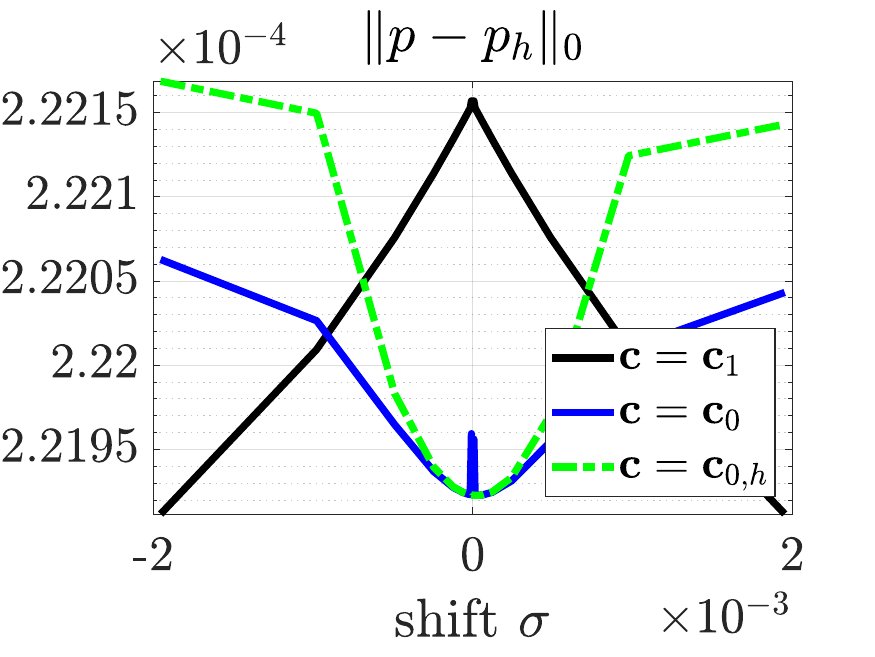}
		\includegraphics[width=0.24\linewidth]{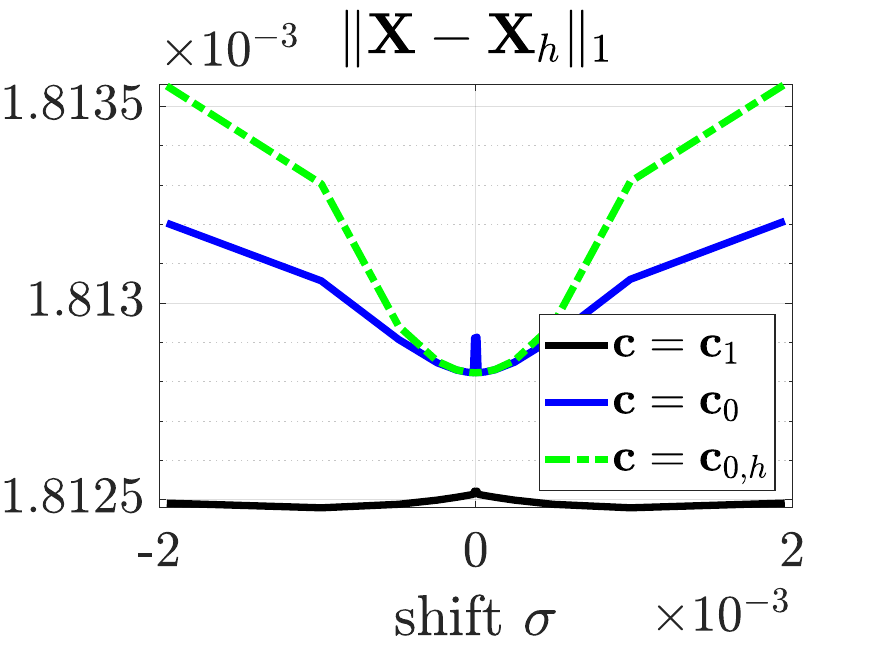}
		\includegraphics[width=0.24\linewidth]{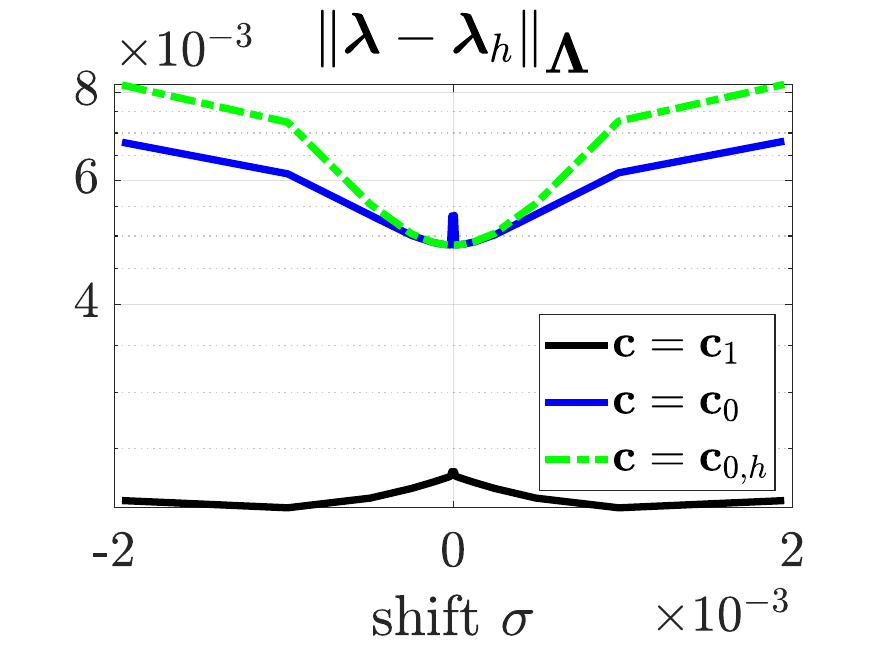}
		
		\
		
		\includegraphics[width=0.24\linewidth]{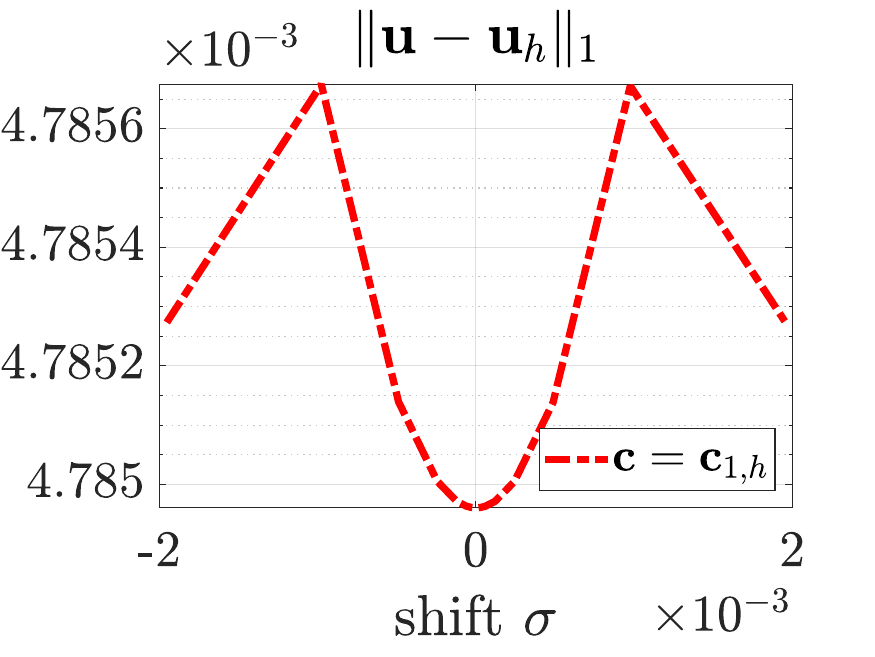}
		\includegraphics[width=0.24\linewidth]{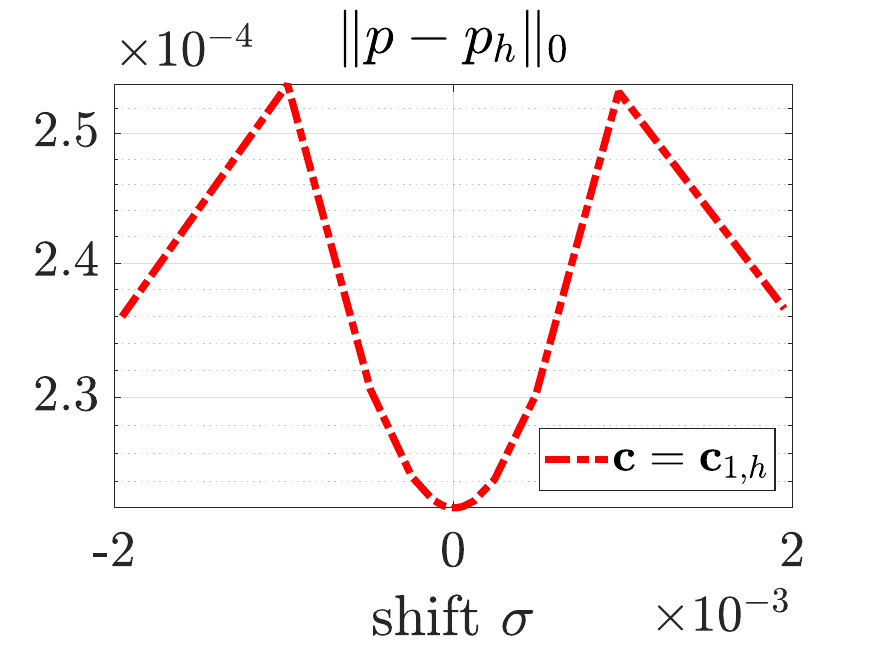}
		\includegraphics[width=0.24\linewidth]{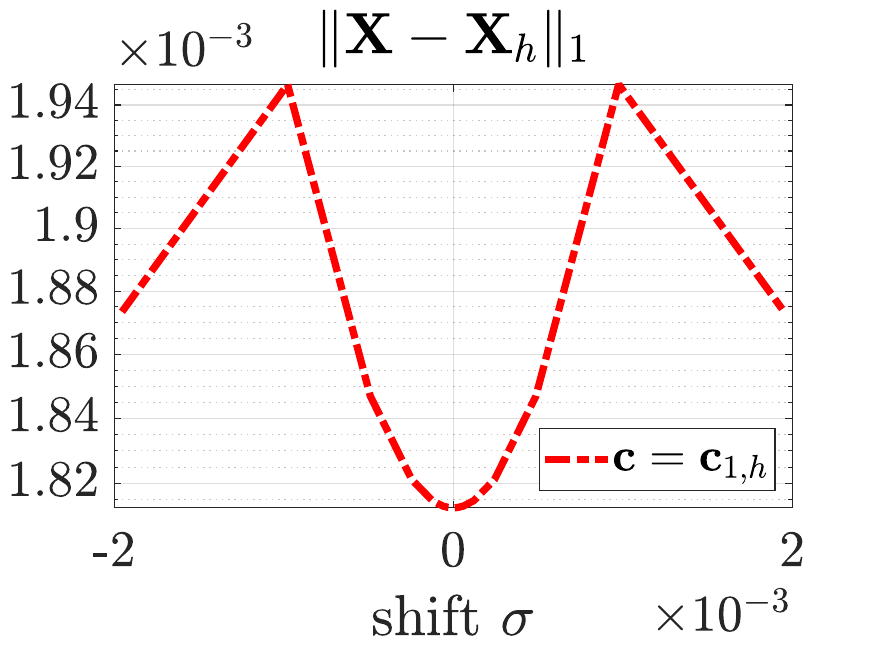}
		\includegraphics[width=0.24\linewidth]{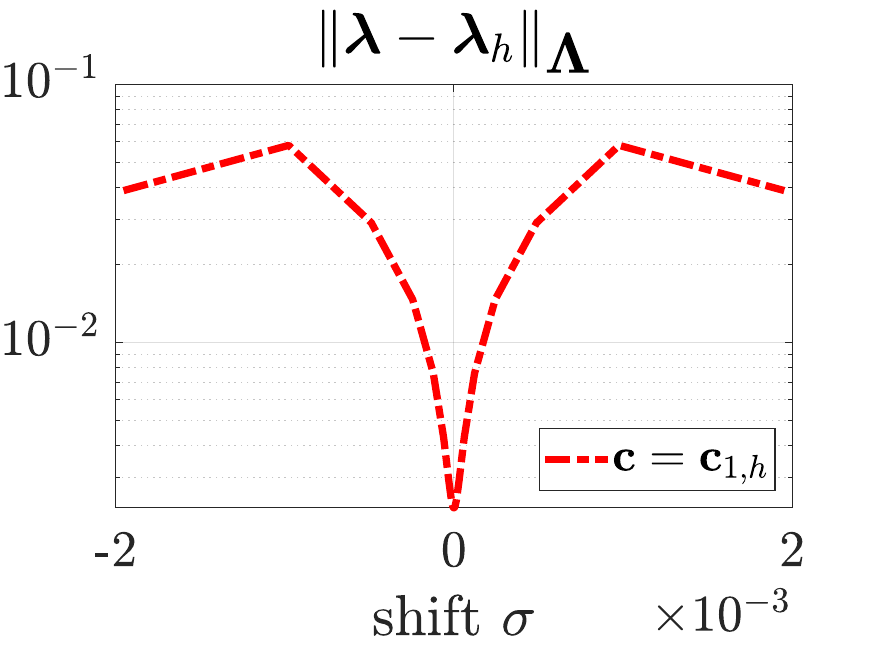}
		\caption{Error for the shifted square test plotted with respect to $\sigma$. In all cases, results are not affected by the value of the shift. The curves related to the coupling $\c_{1,h}$ are plotted separately since they have a different scale due to the quadrature error.}
		\label{fig:shift_err}
	\end{figure}
	
	At this point, we study error and condition number with respect to $h$ refinement for three fixed values of $\sigma$. In particular, we choose $\sigma=0$ (the matching situation), $\sigma=\pi\cdot10^{-13}$, and $\sigma=\pi\cdot10^{-3}$. Initially, $\T_h^\Omega$ and $\T_h^\B$ are both $8\times8$ uniform meshes, which are then refined five times in such a way that $h_\Omega/h_\B$ is kept constant.
	Since $h_\B=C\,h_\Omega\approx h$, Theorem~\ref{theo:conditioning} gives the following theoretical behavior of the condition number, \red as already observed in Remark~\ref{rem:cond_h}\nored
	\begin{equation}\label{eq:cond_tests}
		\begin{aligned}
			&\cond{\Sys_\square} \le C h^{-4} &&\quad \text{if } \c=\c_0 \text{ or } \c_h=\c_{0,h},\\
			&\cond{\Sys_\square} \le C h^{-2} &&\quad\text{if }\c=\c_1 \text{ or } \c_h=\c_{1,h}.
		\end{aligned}
	\end{equation} 
	
	The numerical growth rates of the condition number are reported in Figure~\ref{fig:shift_cond_h}. The results agree with the theoretical findings and do not change when $\sigma$ changes. Once again, the absolute value of the condition number is larger when $\c_0$ or $\c_{0,h}$ are considered.
	
	\begin{figure}
		\centering
		\includegraphics[width=0.328\linewidth]{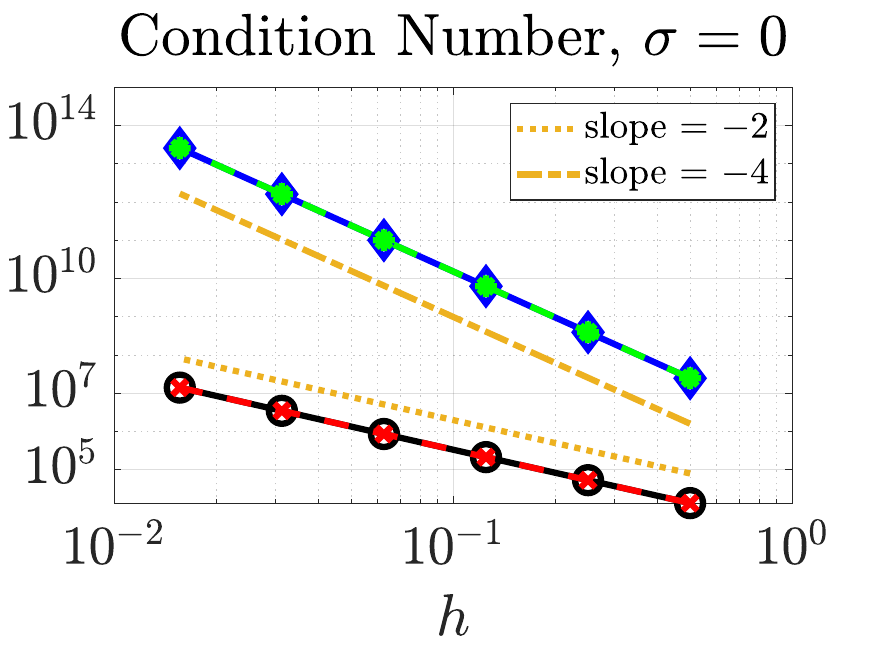}
		\includegraphics[width=0.328\linewidth]{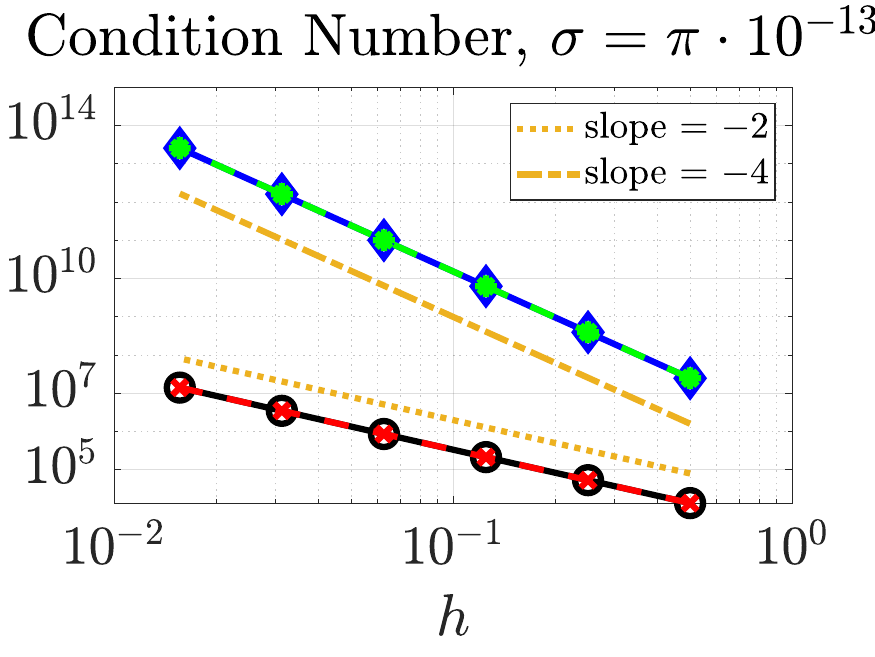}
		\includegraphics[width=0.328\linewidth]{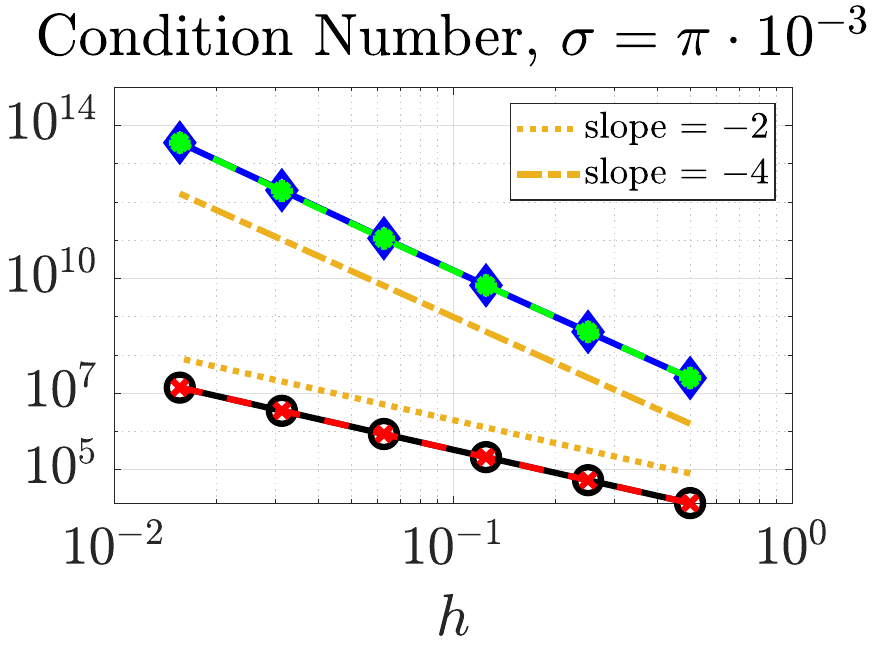}
		\includegraphics[width=0.35\linewidth]{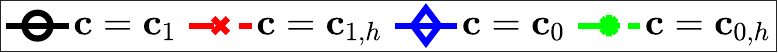}
		\caption{Condition number of the shifted square test plotted as a function of mesh size for fixed values of $\sigma$.}
		\label{fig:shift_cond_h}
	\end{figure}
	
	Regarding the error convergence, for $\sigma=0$ we obtain that the exact and inexact assembly of the coupling term produces the same results. This makes sense since the immersed solid mesh matches with the fluid one. Convergence plots are reported in Figure~\ref{fig:shift_sigma0} for both choices of $\c$ (the top line refers to $\c_0$, while the bottom line to $\c_1$); all variables converge with optimal rates. We obtain the same behavior for $\sigma=\pi\cdot10^{-13}$: this is justified by the fact that $\sigma$ is very close to machine precision. We do not report the associated convergence curves. Results for $\sigma=\pi\cdot10^{-3}$ are collected in Figure~\ref{fig:shift_sigma-3} (same format as for Figure~\ref{fig:shift_sigma0}). We first notice that $\c_{1,h}$ gives sub-optimal results in agreement with the chosen mesh refinement technique and Theorem~\ref{theo:c1h}. For all other cases, the method is optimal.
	
	We conclude that the value of the shift $\sigma$, i.e. the position of the interface between solid and fluid, does not affect either the condition number or the accuracy of the method. A lack of convergence can only be caused by an inappropriate choice of integration technique for the coupling matrix.
	
	\begin{figure}
		\centering
		\vspace{5mm}
		\textbf{Shifted square with $\sigma=0$: $\LdBd$ coupling \textit{vs} $\Hub$ coupling }\\
		\vspace{3mm}
		\includegraphics[width=0.24\linewidth]{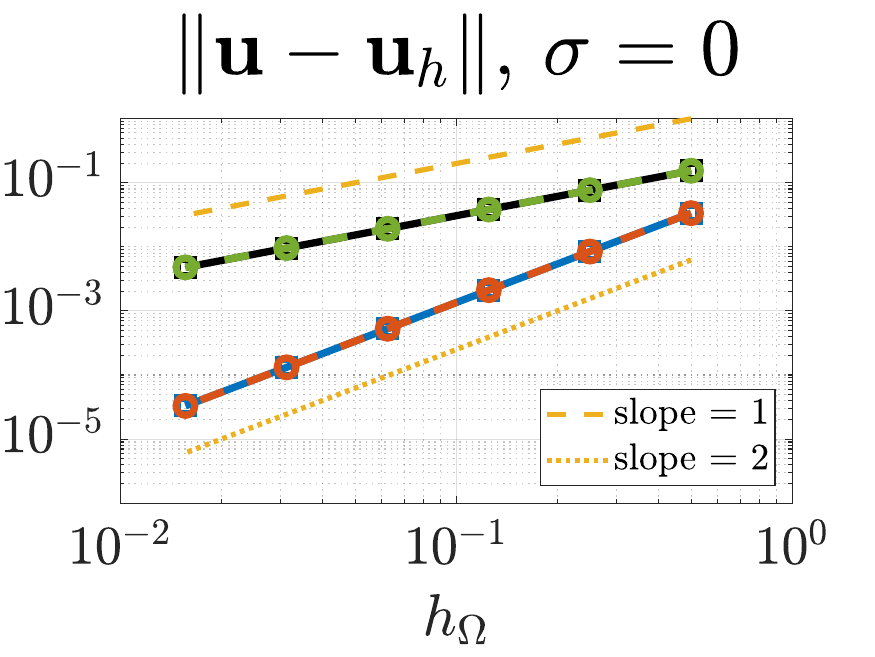}
		\includegraphics[width=0.24\linewidth]{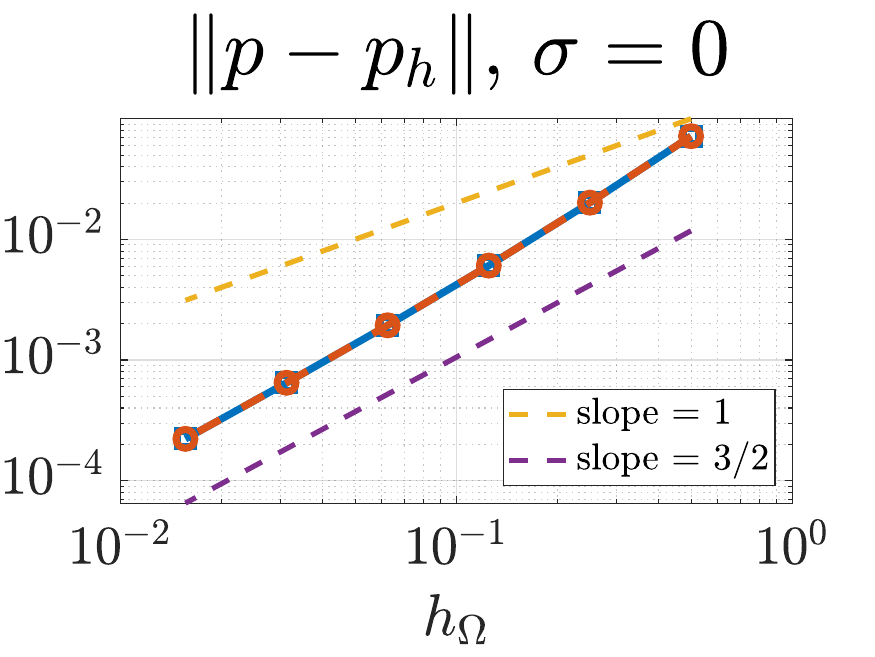}
		\includegraphics[width=0.24\linewidth]{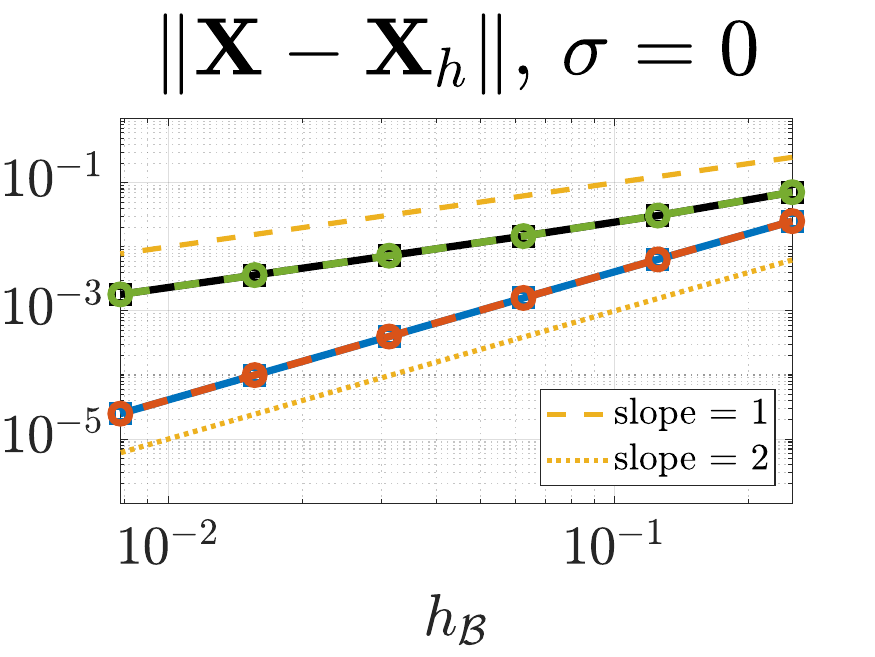}
		\includegraphics[width=0.24\linewidth]{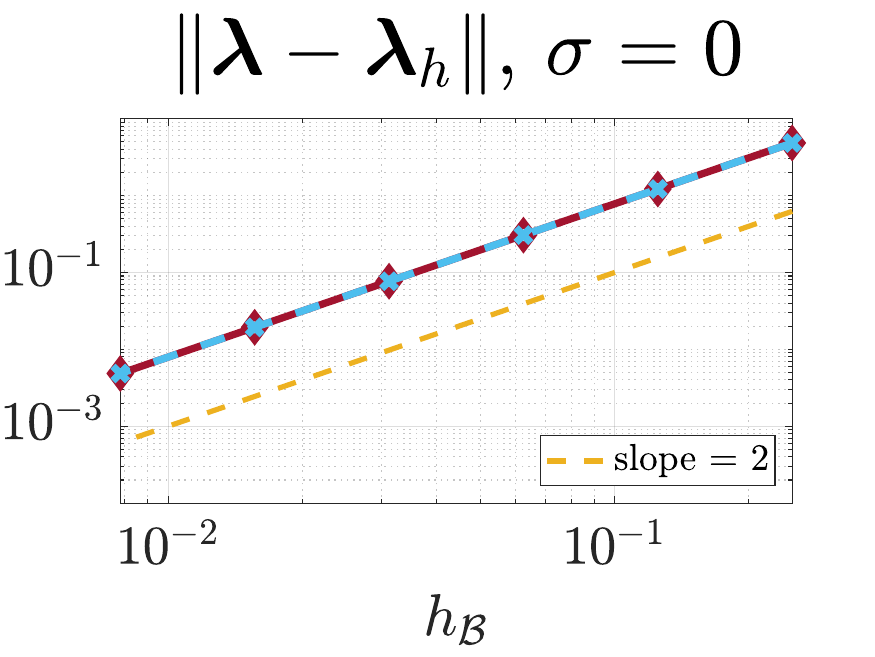}
		\includegraphics[width=0.94\linewidth]{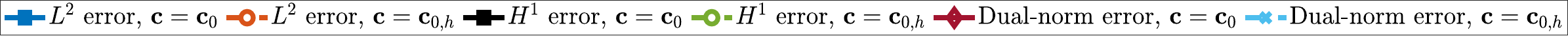}
		
		\
		
		\includegraphics[width=0.24\linewidth]{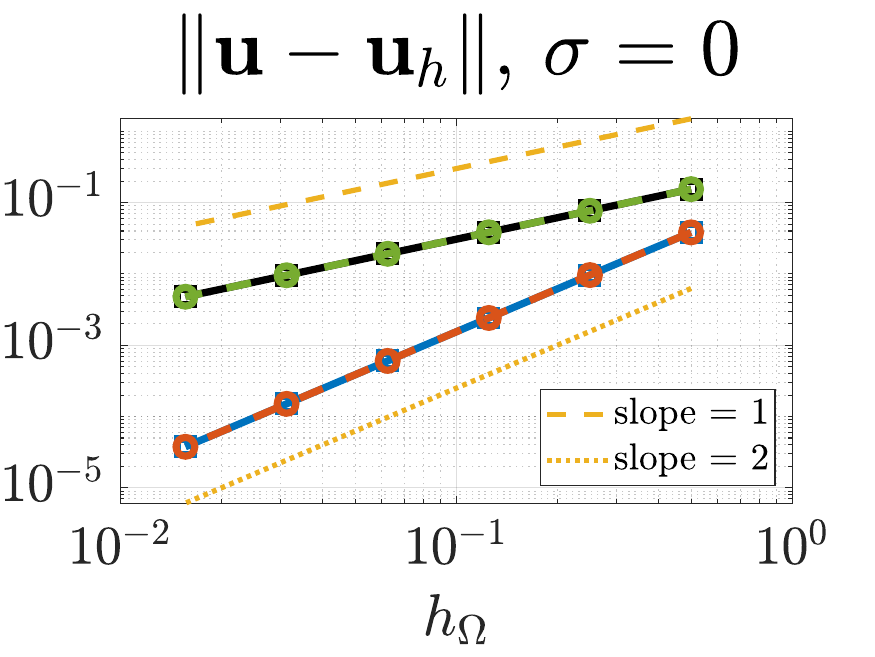}
		\includegraphics[width=0.24\linewidth]{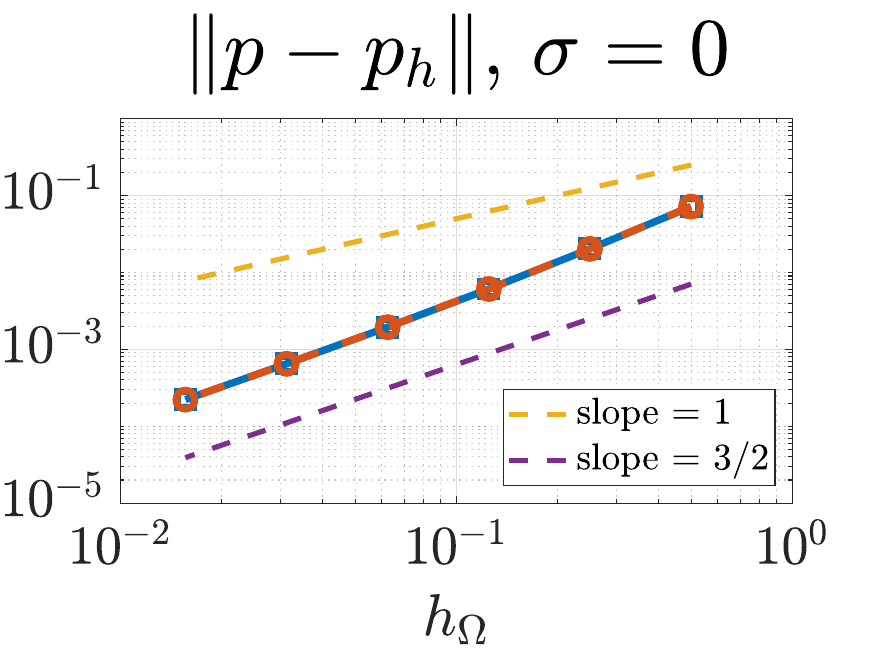}
		\includegraphics[width=0.24\linewidth]{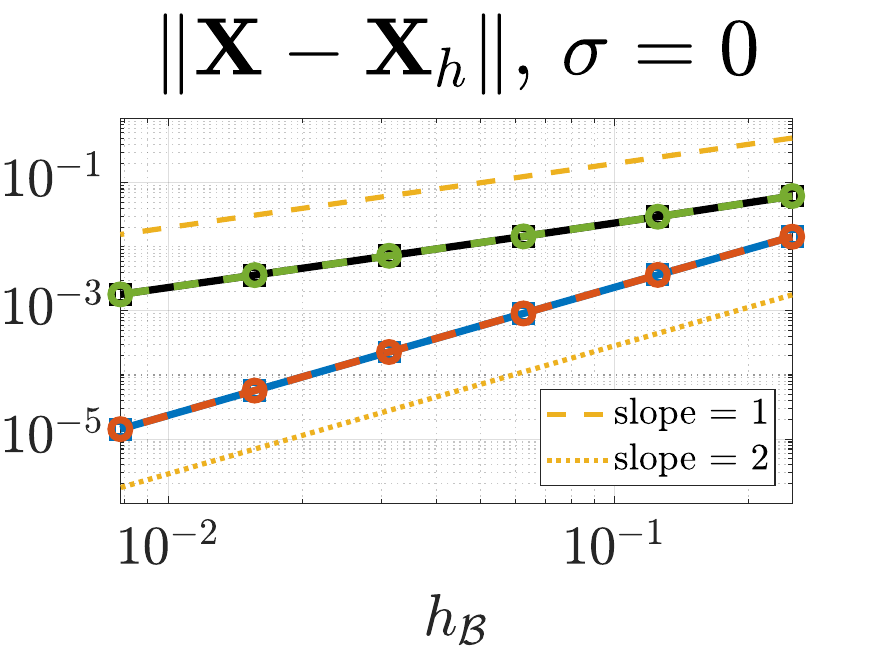}
		\includegraphics[width=0.24\linewidth]{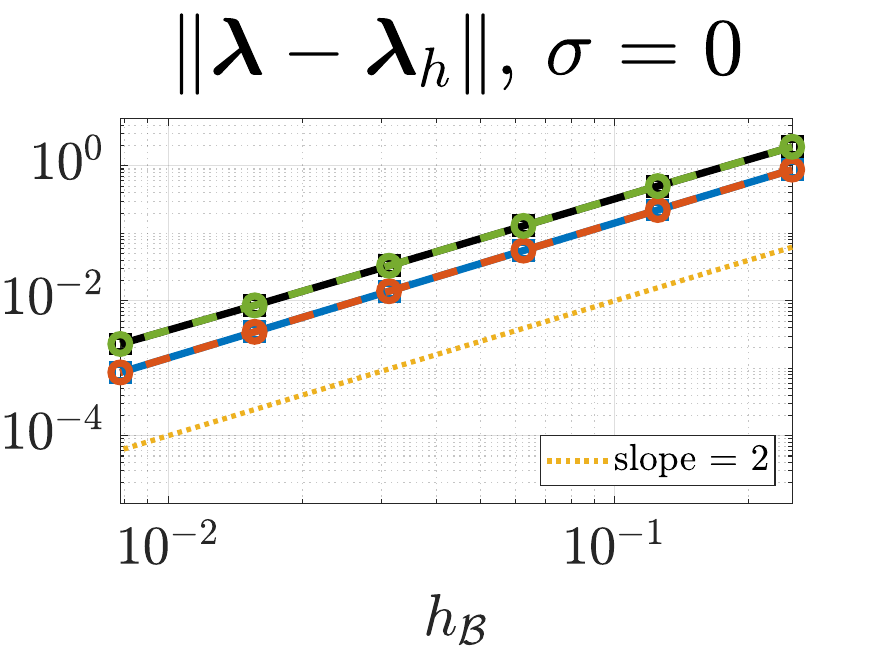}
		\includegraphics[width=0.55\linewidth]{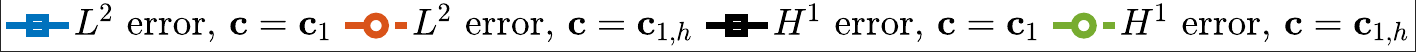}
		\caption{Error convergence for the shifted square test for $\sigma=0$. The immersed solid mesh perfectly matches with the fluid one. Thus, the coupling term is always assembled exactly provided that the employed quadrature rule is sufficiently precise.}
		\label{fig:shift_sigma0}	
	\end{figure}
	
	\begin{figure}
		\centering
		\vspace{5mm}
		\textbf{Shifted square with $\sigma=\pi\cdot10^{-3}$: $\LdBd$ coupling \textit{vs} $\Hub$ coupling}\\
		\vspace{3mm}
		\includegraphics[width=0.24\linewidth]{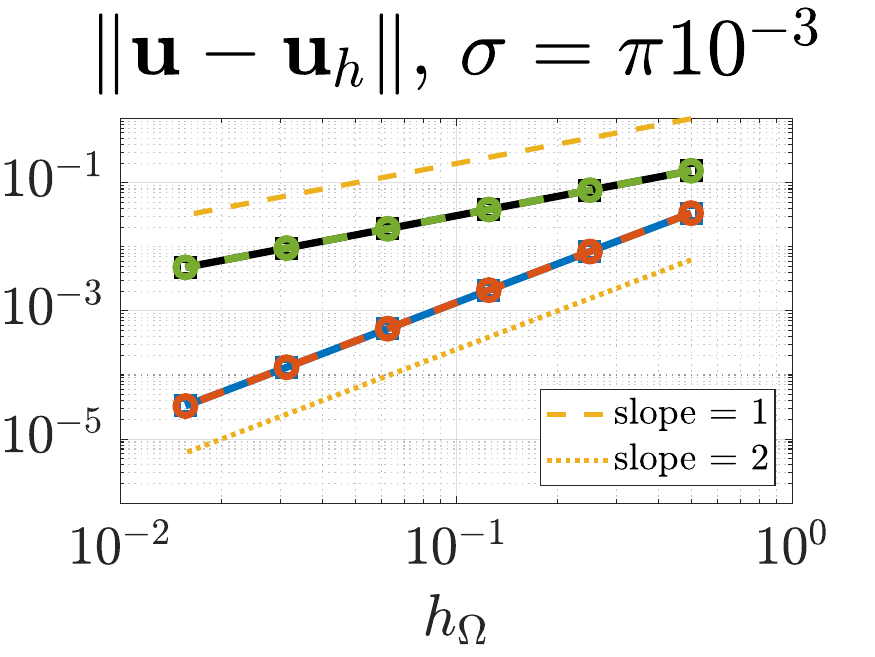}
		\includegraphics[width=0.24\linewidth]{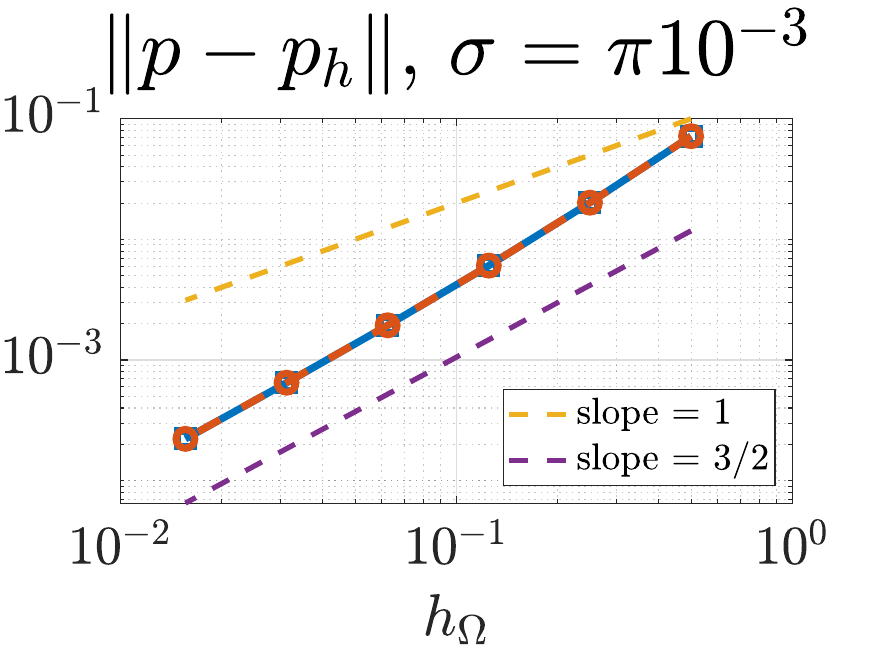}
		\includegraphics[width=0.24\linewidth]{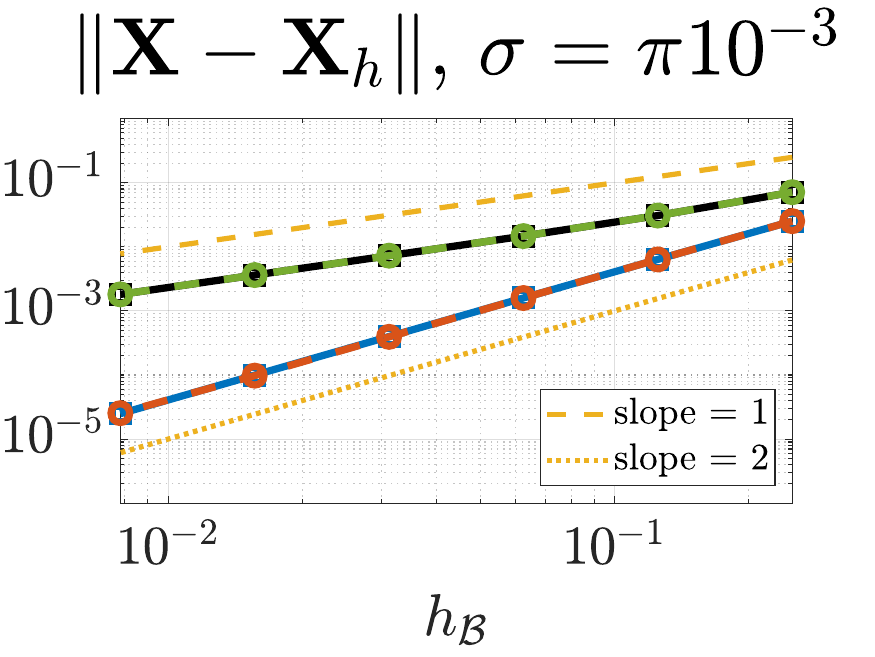}
		\includegraphics[width=0.24\linewidth]{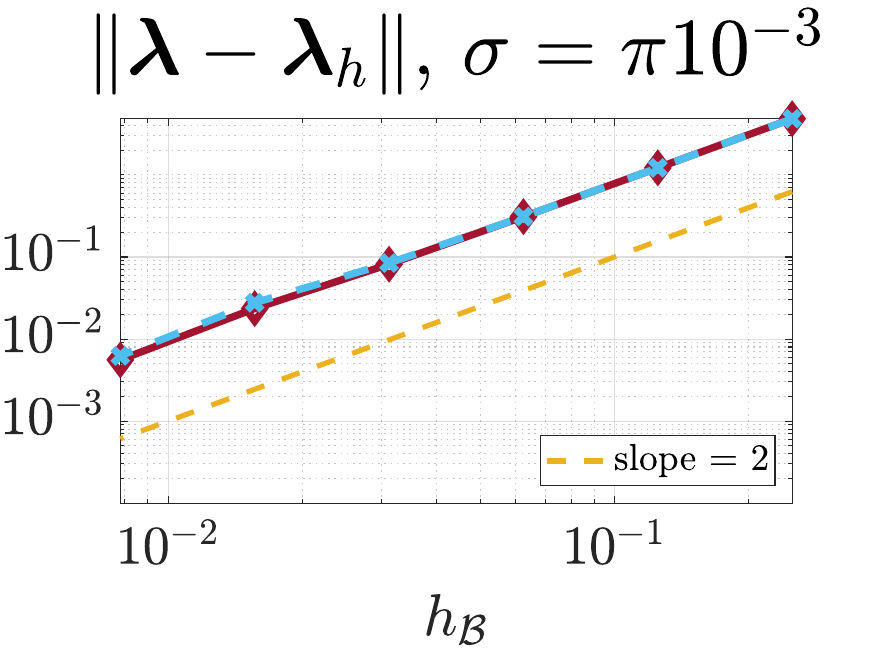}
		\includegraphics[width=0.94\linewidth]{figures_paper/legend_l2}
		
		\
		
		\includegraphics[width=0.24\linewidth]{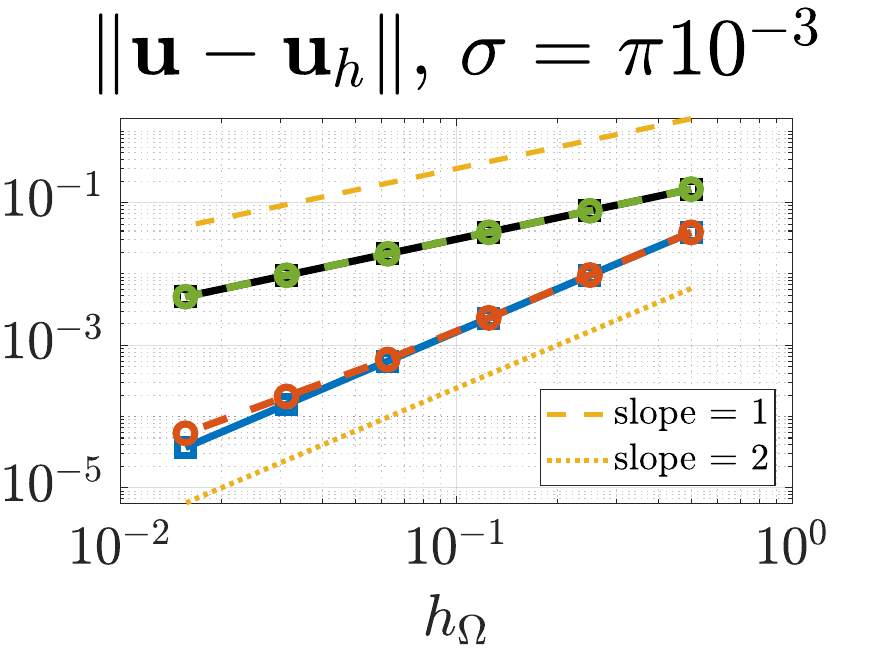}
		\includegraphics[width=0.24\linewidth]{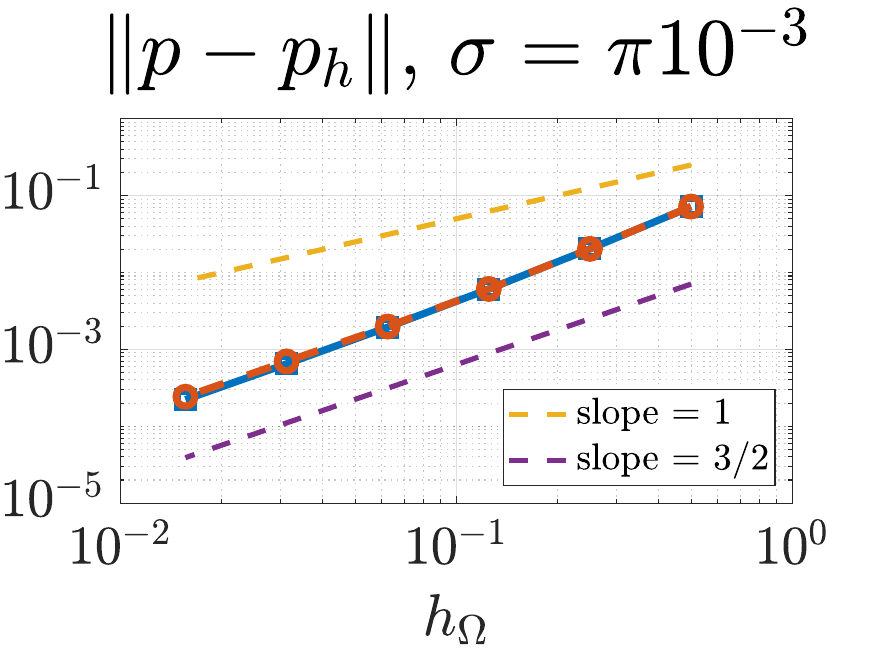}
		\includegraphics[width=0.24\linewidth]{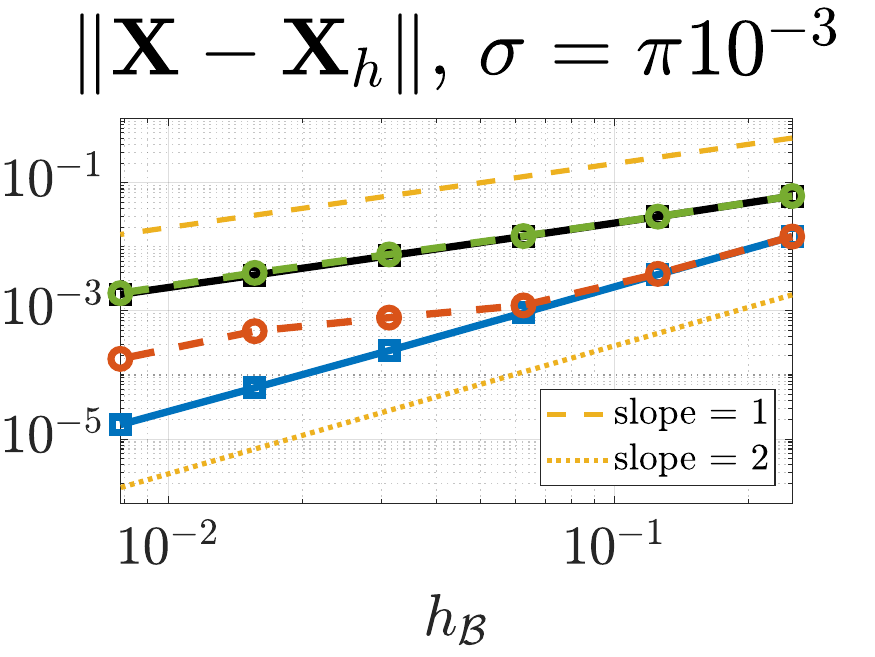}
		\includegraphics[width=0.24\linewidth]{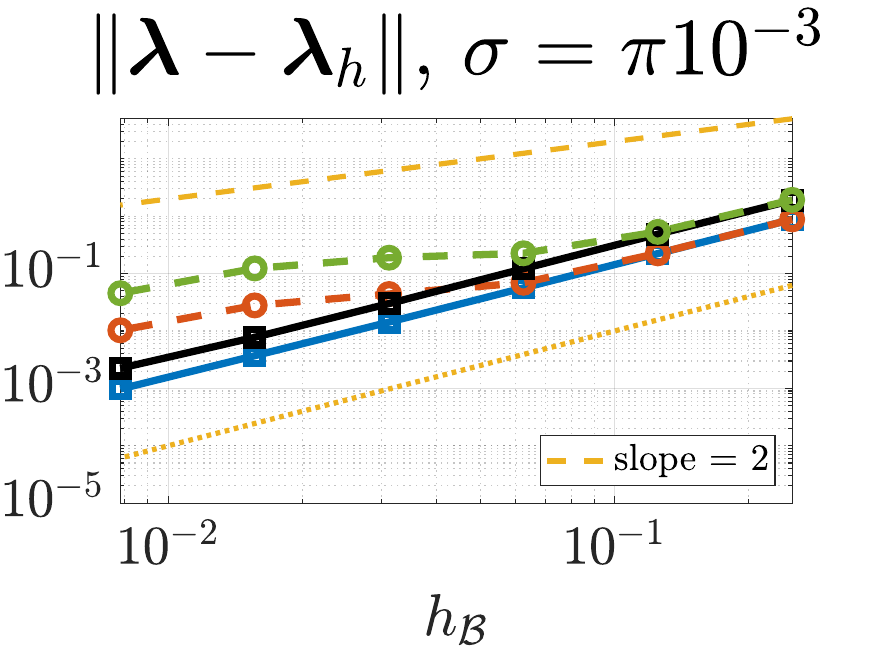}
		\includegraphics[width=0.55\linewidth]{figures_paper/legend_h1}
		\caption{Error convergence for the shifted square test for $\sigma=\pi\cdot10^{-3}$. The sub-optimality of $\c_{1,h}$ agrees with Theorem~\ref{theo:c1h} and $h_\Omega/h_\B$ constant.}
		\label{fig:shift_sigma-3}
		
	\end{figure}
	
	\subsection{Disk, flower and stretched annulus}
	
	We now consider three different geometries for the immersed solid: the disk, the flower-shaped domain and the stretched annulus. In all cases, the fluid domain is the unit square discretized by structured triangulations, whereas the solid mesh is unstructured, but quasi-uniform. The initial meshes are refined five times and $h_\Omega/h_\B$ is kept constant. 
	In Figure~\ref{fig:geo} we show the meshes corresponding to the first level. Moreover, 
	Table~\ref{tab:small_cells} reports the area of the smallest intersection of each refinement.
	
	\begin{figure}
		\centering
		\includegraphics[width=0.3\linewidth]{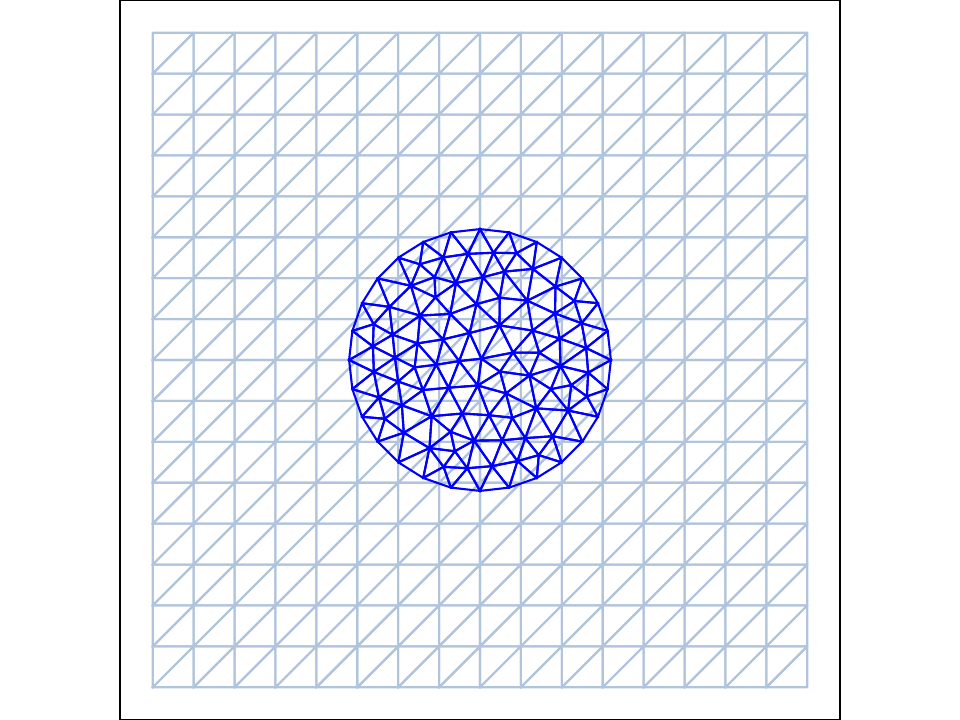}
		\includegraphics[width=0.3\linewidth]{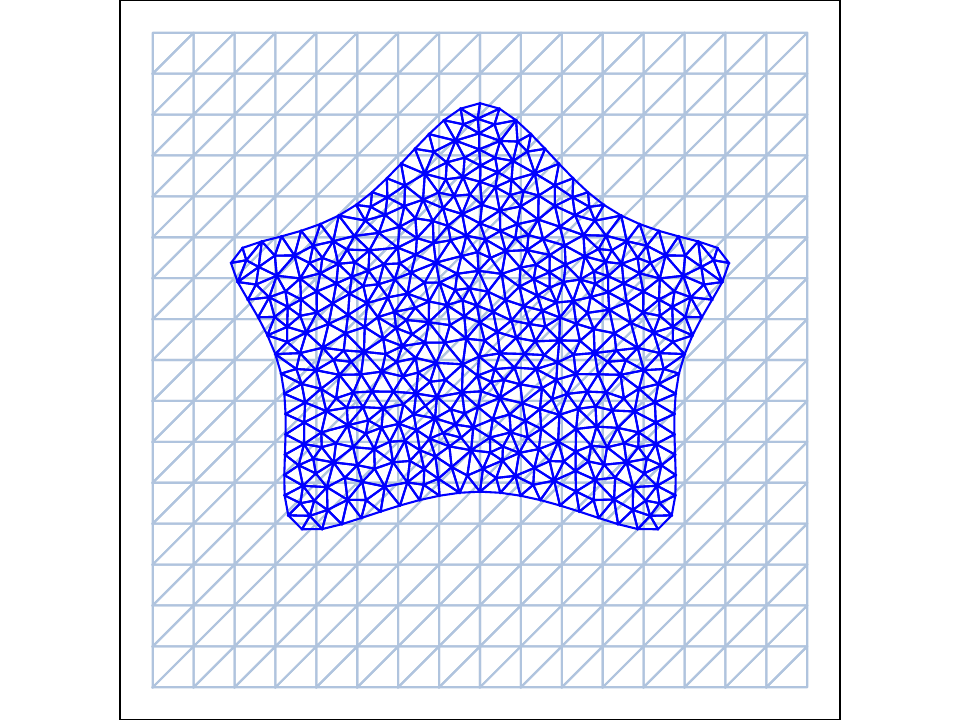}
		\includegraphics[width=0.3\linewidth]{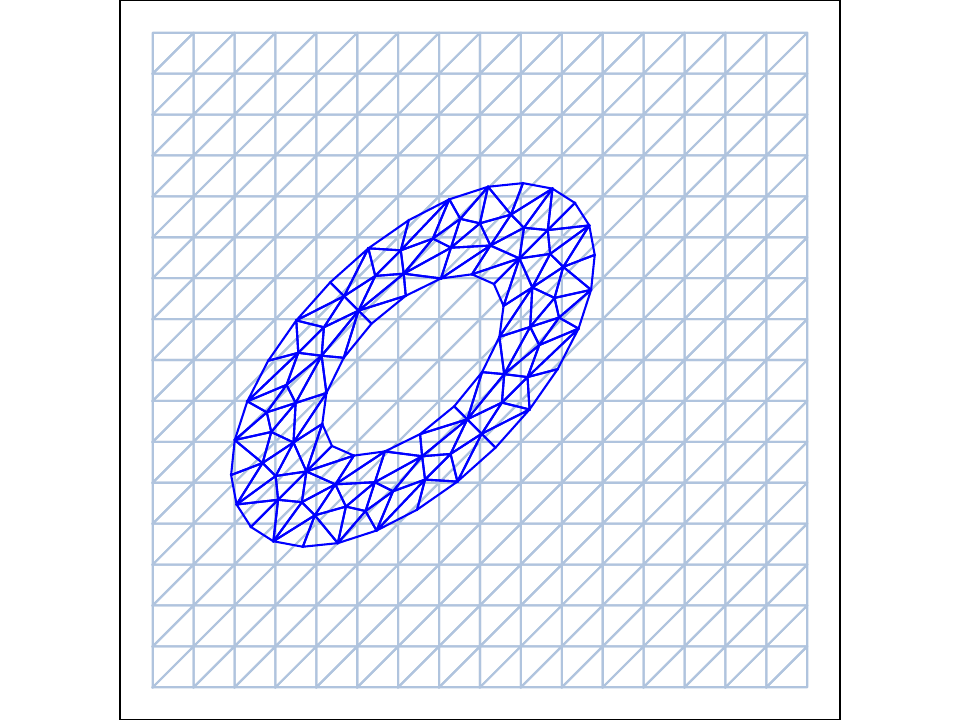}\\
		\caption{Geometric configuration of disk, flower, and annulus at the coarsest level. Solid meshes were generated with $\mathtt{Gmsh}$~\cite{geuzaine2009gmsh}.}
		\label{fig:disk_flower_annulus}
	\end{figure}
	
	\begin{table}[ht]\renewcommand{\arraystretch}{1.5}
		\centering
		\begin{tabular}{c|c|c|c}
			\multicolumn{4}{c}{\textbf{\textsc{Area of Smallest Intersection}}}\\
			\hline
			\textbf{Level} & \textbf{Disk} & \textbf{Flower} & \textbf{Annulus}\\
			\hline
			1 & 7.336\,e--10 & 1.149\,e--10 & 3.235\,e--10\\
			2 & 1.206\,e--10 & 2.043\,e--10 & 1.378\,e--10\\
			3 & 1.005\,e--10 & 1.112\,e--10 & 1.089\,e--10\\
			4 & 1.000\,e--10 & 1.009\,e--10 & 1.002\,e--10\\
			5 & 1.001\,e--10 & 1.001\,e--10 & 1.000\,e--10\\
			6 & 1.000\,e--10 & 1.000\,e--10 & 1.000\,e--10\\
			\hline
		\end{tabular}
		\caption{Area of the smallest cut cell arising from the mesh intersection $\Xbar(\T_h^\B)\cap\T_{h/2}^\Omega$. The first column indicates the level of mesh refinement.}
		\label{tab:small_cells}
	\end{table}

	\begin{figure}
		\centering
		\includegraphics[width=0.328\linewidth]{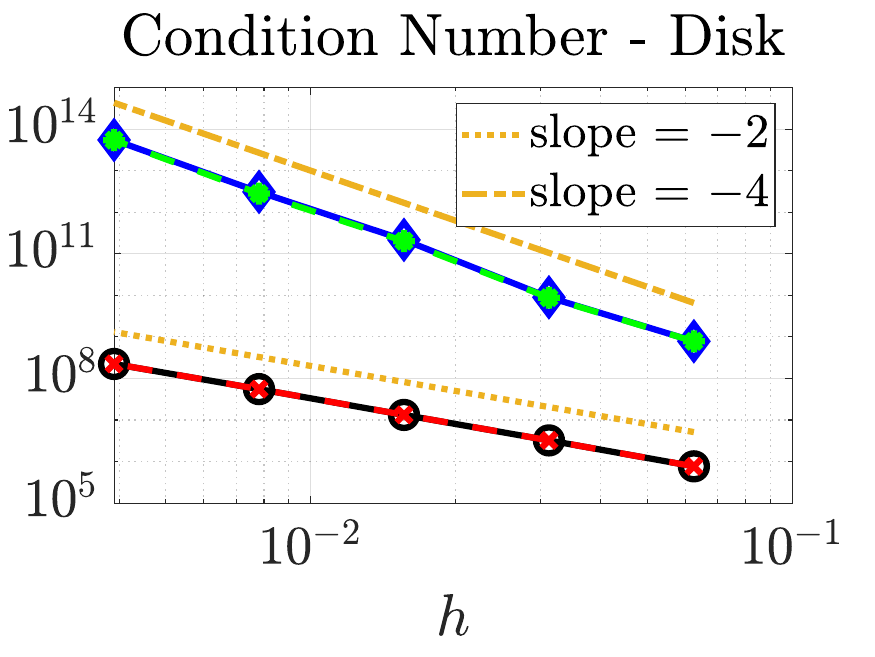}
		\includegraphics[width=0.328\linewidth]{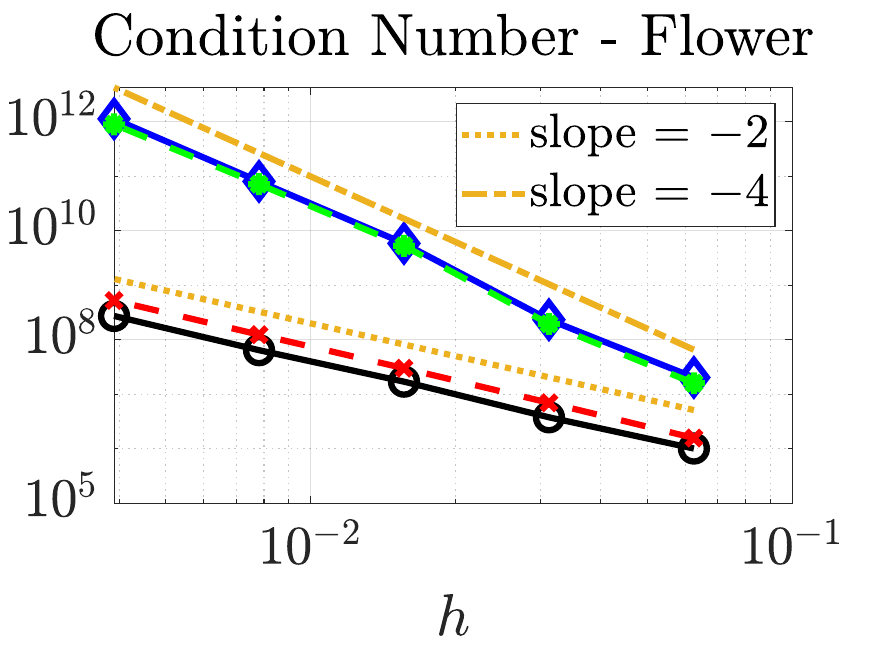}
		\includegraphics[width=0.328\linewidth]{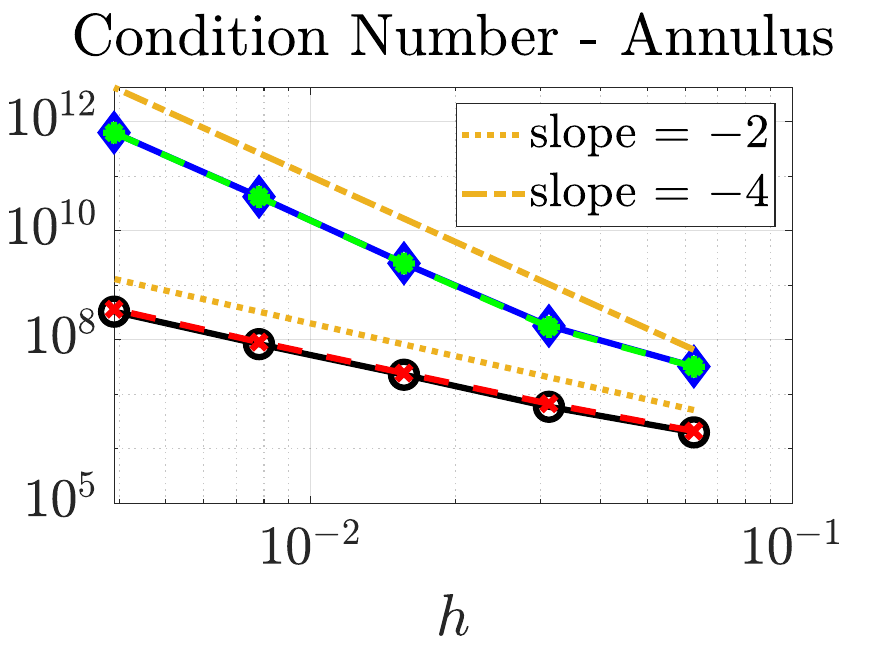}
		\includegraphics[width=0.35\linewidth]{figures_paper/legend_cond}
		\caption{Condition number of disk, flower, and annulus tests.}
		\label{fig:dfa_cond}
	\end{figure}
	
	We now briefly report the details regarding the geometric configuration and the manufactured solutions of each case.
	
	\subsubsection{The disk}
	
	The immersed solid is represented by the disk with radius $R=1/5$ centered at $(1/2,1/2)$. We assume that $\B$ corresponds to the actual position of the body, thus $\Xbar$ is the identity function. We set $\alpha=\beta=0$ and $\gamma=1$ and we compute the right hand sides according to the following solution
	\begin{equation*}
		\begin{aligned}
			&\u(x,y) = \big( 2x^2y(x-1)^2(y-1)(2y-1), -2xy^2(x-1)(2x-1)(y-1)^2 \big)\\
			&p(x,y) =
			\begin{cases}
				\sin(\pi x)\sin(\pi y) - {4}/{\pi^2} - |\Os|/(2|\Of|) & \text{in }\Of\\
				\sin(\pi x)\sin(\pi y) - {4}/{\pi^2} + 1/2 & \text{in }\Os=\B
			\end{cases}\\
			&\X(s_1,s_2) = \big(s_1^4 - 2s_1^3 + s_1^2,- 2s_2^3 + 3s_2^2 - s_2\big)\\
			&\llambda(s_1,s_2) = \big(s_2\sin(s_1),s_2\cos(s_1)\big).
		\end{aligned}
	\end{equation*}
	We point out that we are considering a discontinuous pressure.
	
	\subsubsection{The flower}
	
	For this test, the solid domain is identified by a flower-shaped boundary with inscribed circle centered at $(1/2,1/2)$ and radius $1/4$. We set again $\B=\Os$, $\alpha=\beta=0$ and $\gamma=1$. We choose the following solution
	\begin{equation*}
		\begin{aligned}
			&\u(x,y) = \big(-x\sin(xy),\,y\sin(xy) \big),
			&&\qquad p(x,y) = \cos(xy)-\int_\Omega\cos(x,y)\,\dx,\\
			&\X(s_1,s_2) = \u(s_1,s_2),
			&&\qquad\llambda(s_1,s_2) = (s_2\sin(s_1),s_2\cos(s_1)),
		\end{aligned}
	\end{equation*}
	and we compute the right hand sides accordingly.
	
	\begin{figure}
		\centering
		\textbf{Disk: $\LdBd$ coupling \textit{vs} $\Hub$ coupling }\\
		\vspace{3mm}
		\includegraphics[width=0.24\linewidth]{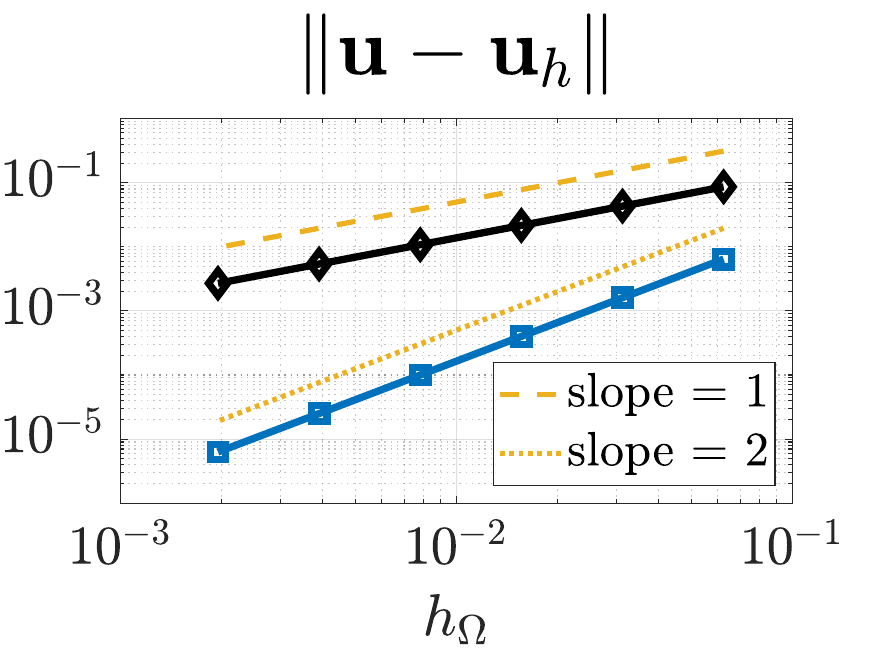}
		\includegraphics[width=0.24\linewidth]{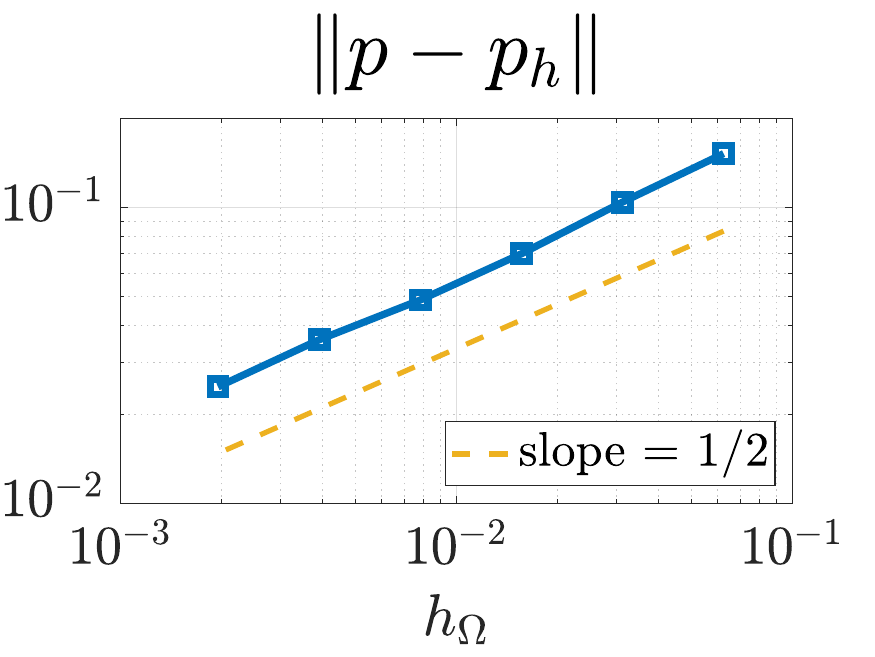}
		\includegraphics[width=0.24\linewidth]{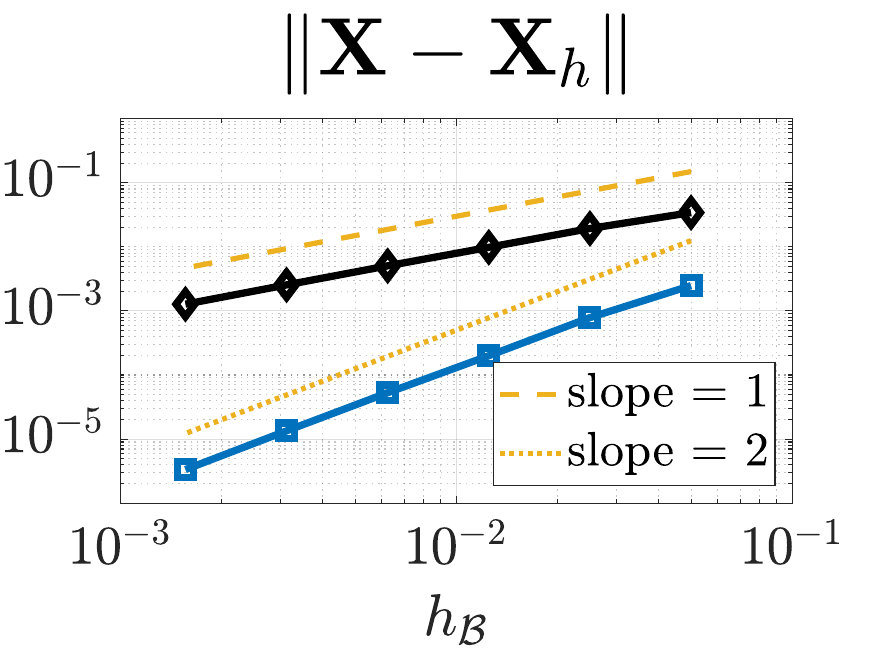}
		\includegraphics[width=0.24\linewidth]{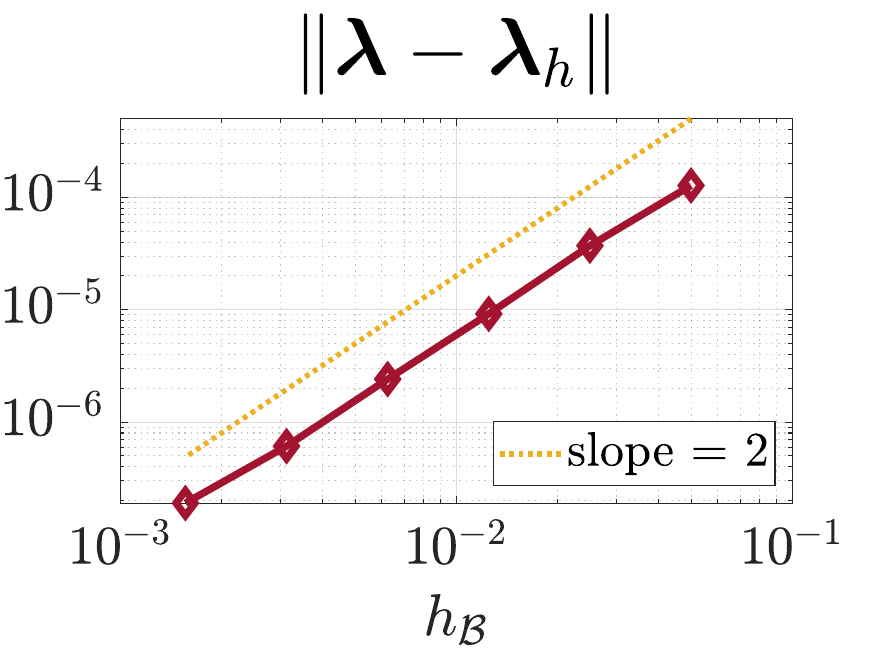}
		\includegraphics[width=0.425\linewidth]{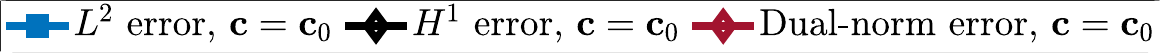}
		
		\
		
		\includegraphics[width=0.24\linewidth]{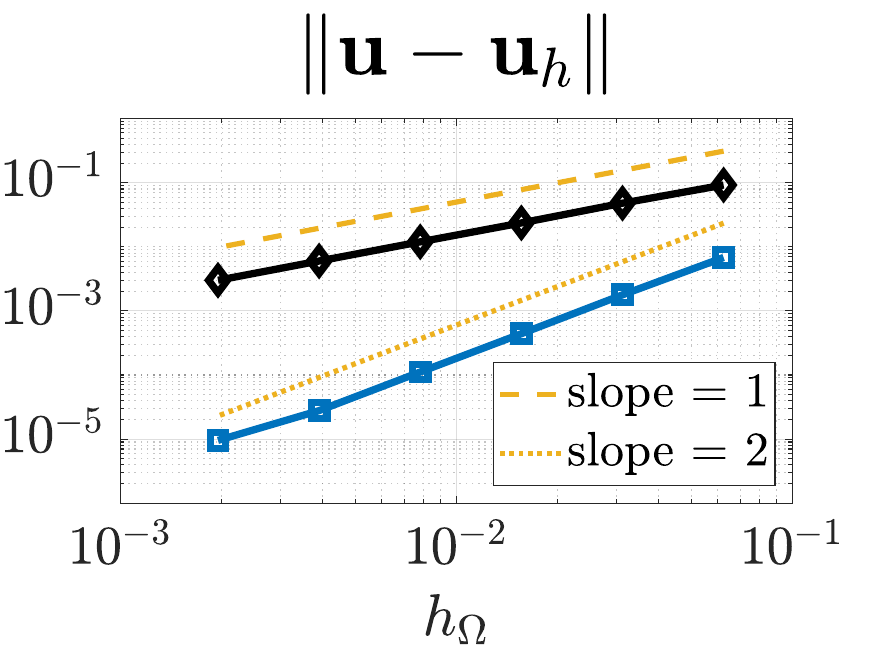}
		\includegraphics[width=0.238\linewidth]{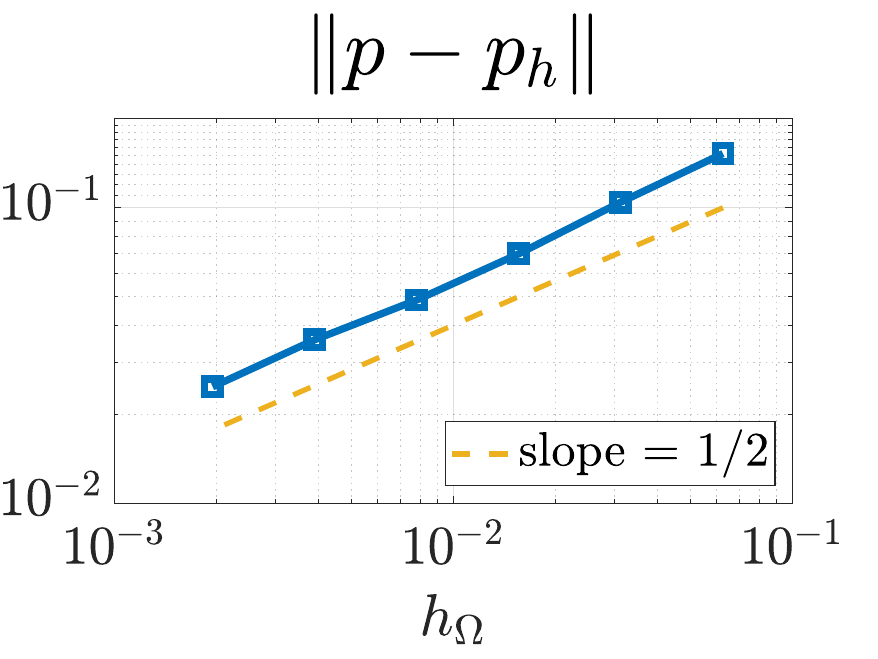}
		\includegraphics[width=0.24\linewidth]{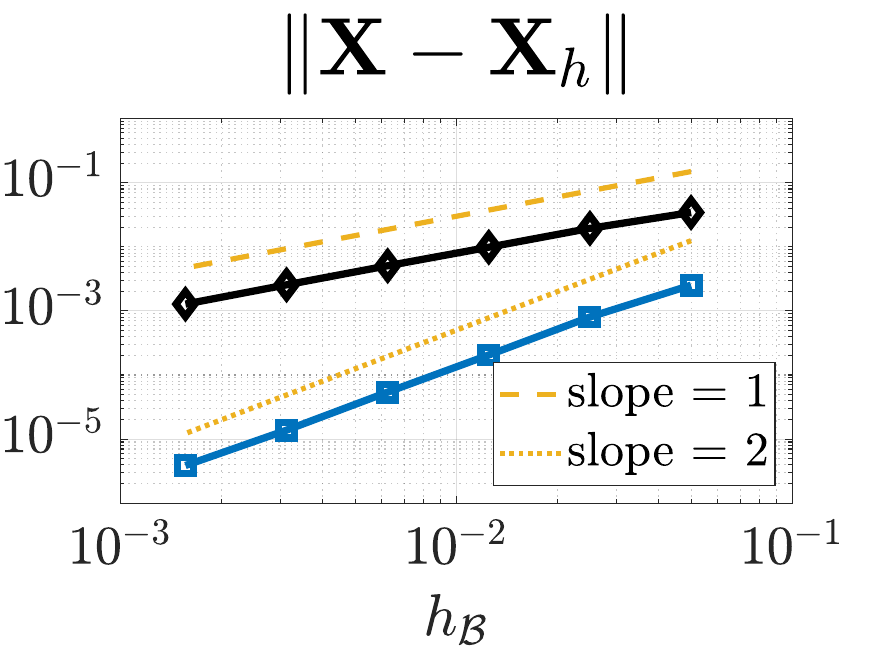}
		\includegraphics[width=0.24\linewidth]{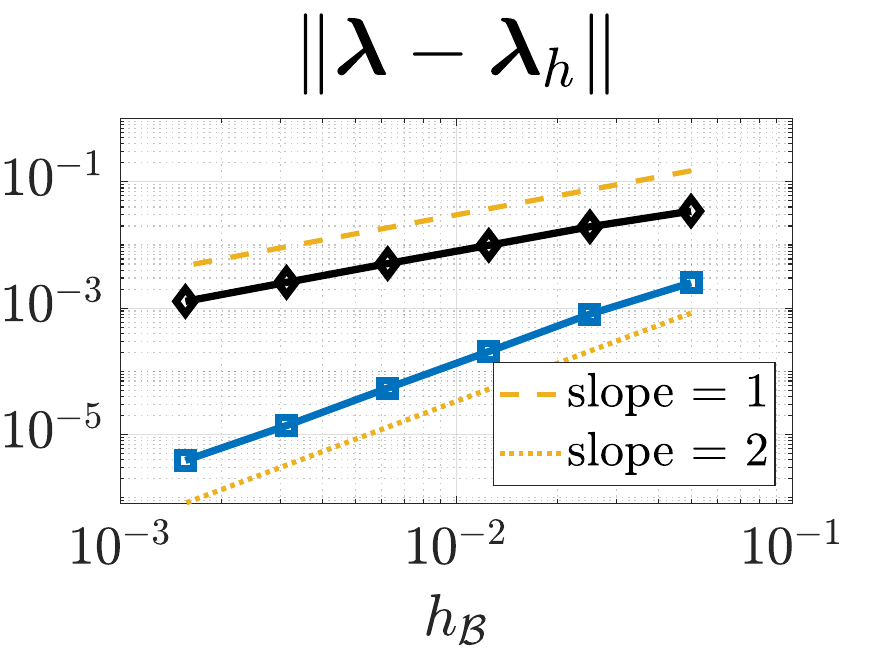}
		\includegraphics[width=0.25\linewidth]{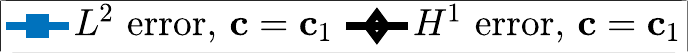}
		\caption{Error convergence for the disk test.}
		\label{fig:circle}
	\end{figure}
	
	\subsubsection{The stretched annulus}
	
	In this case, the actual position of the solid body is given by the stretched annulus obtained by mapping the reference annulus
	\begin{equation*}
		\B = \left\{\s\in\R^2: \frac1{16} \le \left(s_1-\frac12\right)^2+\left(s_2-\frac12\right)^2 \le \frac1{64}\right\},
	\end{equation*}
	by means of
	\begin{equation*}
		\Xbar(s_1,s_2) = \left( \frac{\sqrt{2}}{2}(s_1+s_2)-\frac7{20},\,\frac{\sqrt{2}}{2}(-s_1+s_2)-\frac14\right).
	\end{equation*}
	We set $\alpha=100$, $\beta=200$ and $\gamma=0.03$. We choose the right hand sides so that the problem is solved by the following functions
	\begin{equation*}
		\begin{aligned}
			&\u(x,y) = \big( 2x^2y(x-1)^2(y-1)(2y-1), -2xy^2(x-1)(2x-1)(y-1)^2 \big)\\
			&p(x,y) = x(x-1)(y-1)-1/12\\
			&\X(s_1,s_2) = (-s_1\sin(s_1s_2),\,s_2\sin(s_1s_2)))\\
			&\llambda(s_1,s_2) = (e^{s_1},e^{s_2}).
		\end{aligned}
	\end{equation*}
	
	Since the numerical results show similar behaviors, we describe our findings all together. In Figure~\ref{fig:dfa_cond} we display the condition number as a function of $h$. \lg We considered both the exact and approximate integration of the coupling term\gl. The growth rate is $4$ for $\c_0$ and $2$ for $\c_1$, thus giving the expected values reported in~\eqref{eq:cond_tests}. By combining the information collected in Table~\ref{tab:small_cells} with Figure~\ref{fig:dfa_cond}, it is clear that the condition number is not influenced by the presence of small cut cells.

	\begin{figure}
		\centering
		\textbf{Flower: $\LdBd$ coupling \textit{vs} $\Hub$ coupling }\\
		\vspace{3mm}
		\includegraphics[width=0.24\linewidth]{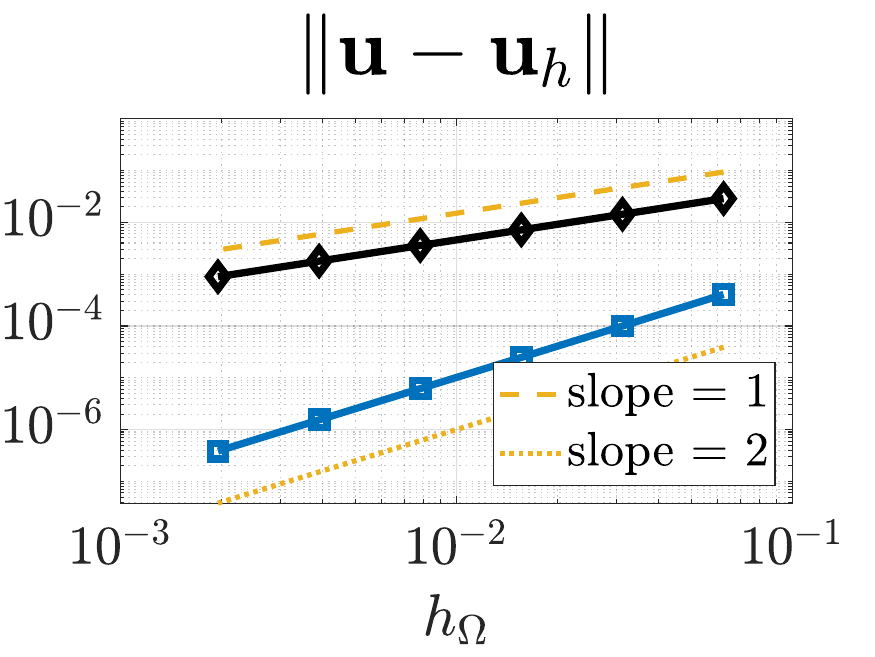}
		\includegraphics[width=0.24\linewidth]{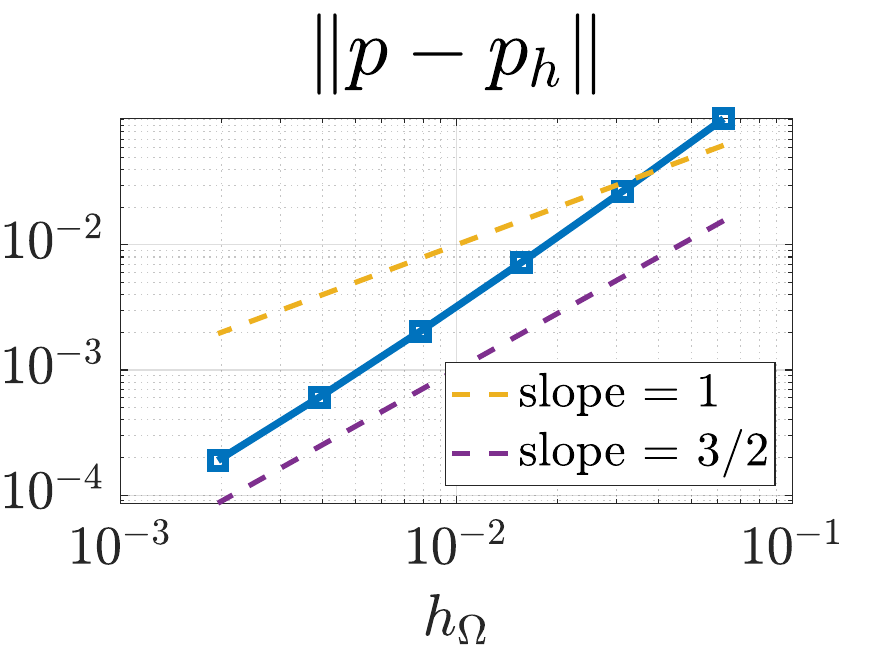}
		\includegraphics[width=0.24\linewidth]{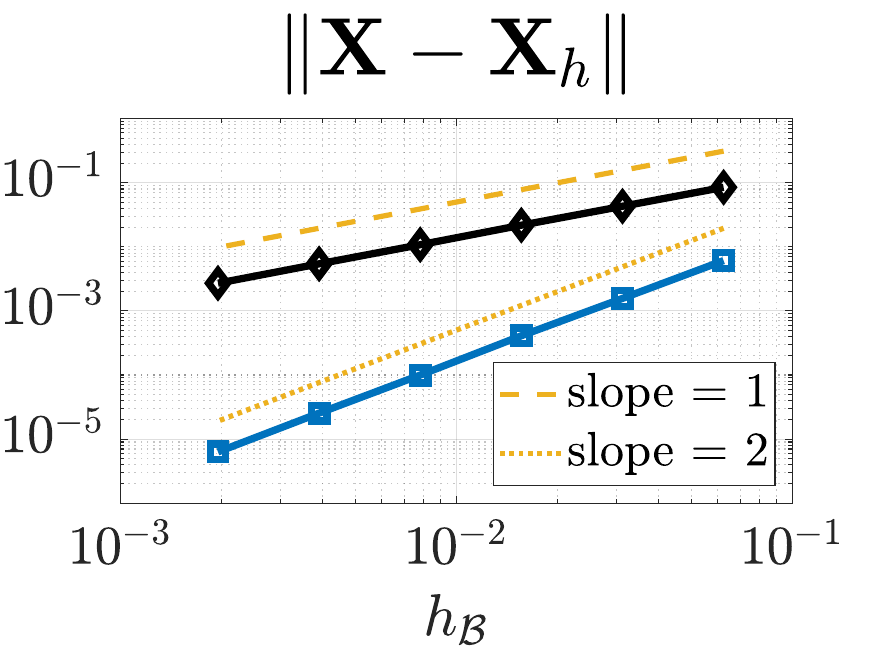}
		\includegraphics[width=0.24\linewidth]{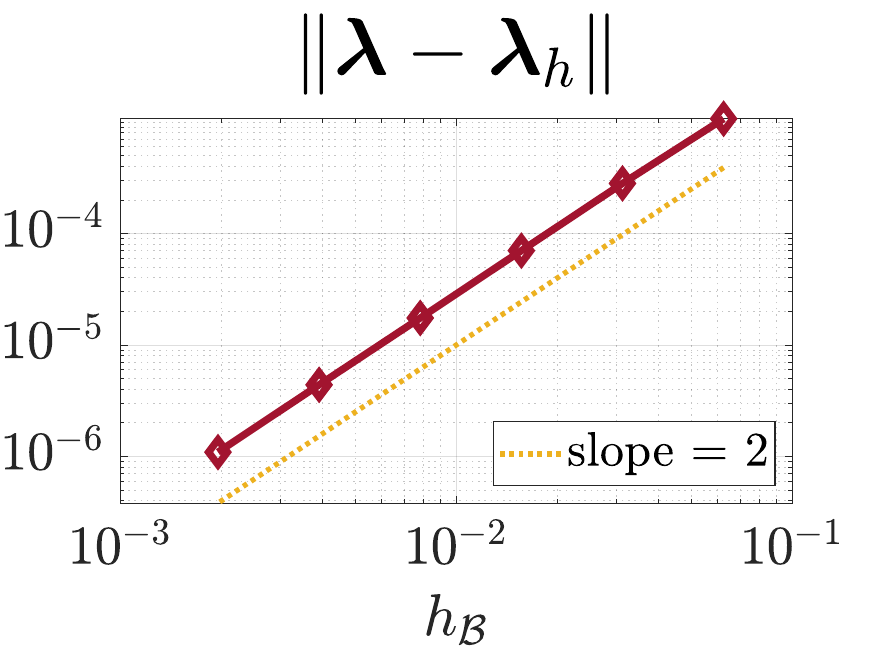}		
		\includegraphics[width=0.425\linewidth]{figures_paper/legend_l2_2}
		
		\
		
		\includegraphics[width=0.24\linewidth]{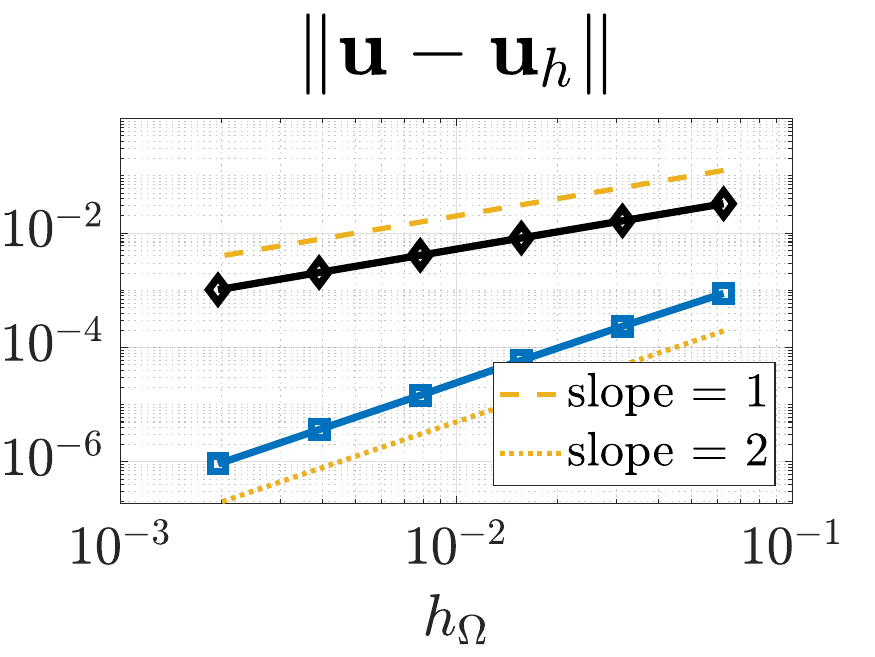}
		\includegraphics[width=0.24\linewidth]{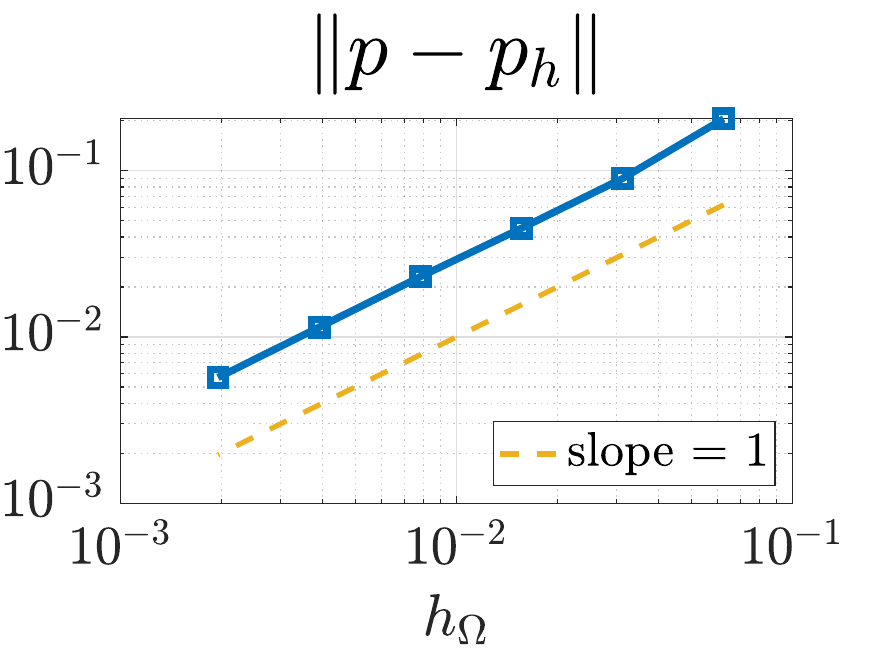}
		\includegraphics[width=0.24\linewidth]{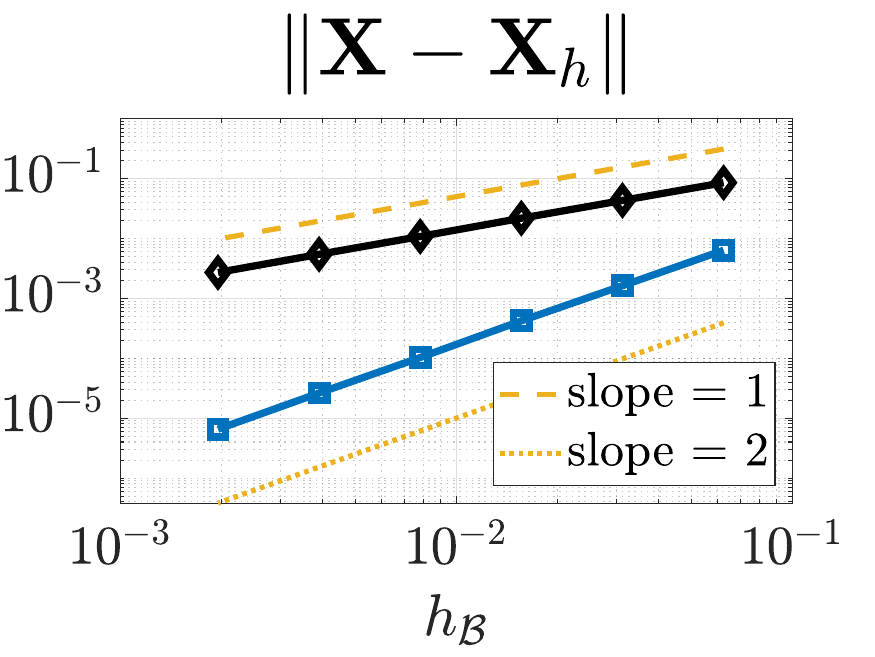}	
		\includegraphics[width=0.24\linewidth]{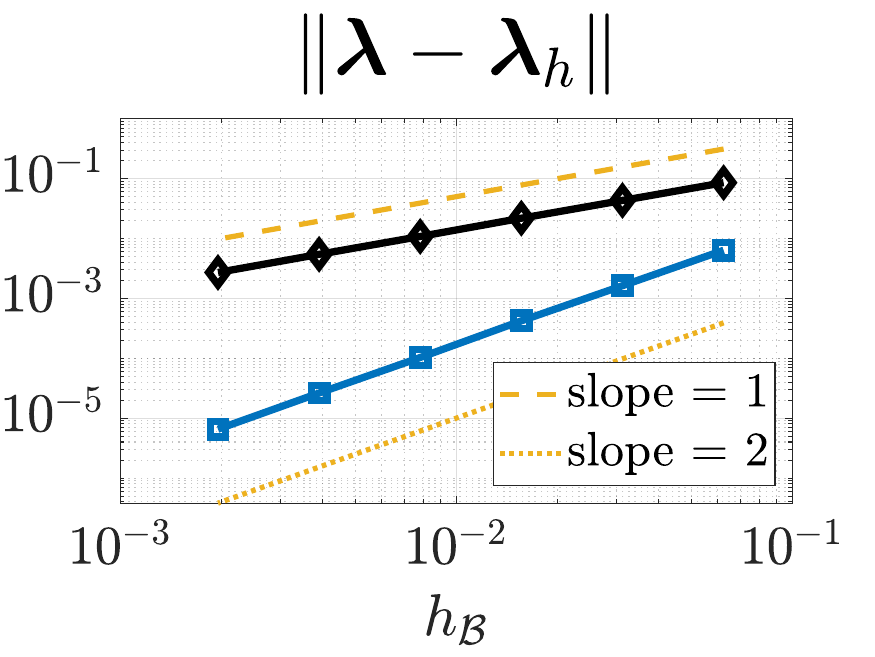}
		\includegraphics[width=0.25\linewidth]{figures_paper/legend_h1_2}
		\caption{Error convergence for the flower-shaped solid test.}
		\label{fig:flower}
	\end{figure}
	
	Figures~\ref{fig:circle}--\ref{fig:annulus} display the rate of
	convergence for the three examples under investigation \lg in case of exact computation of the coupling term\gl. As before, the results
	for $\c_0$ are in the top line, while those for $\c_1$ are in the bottom line.

	\begin{figure}
		\centering
		\textbf{Annulus: $\LdBd$ coupling \textit{vs} $\Hub$ coupling }\\
		\vspace{3mm}
		\includegraphics[width=0.24\linewidth]{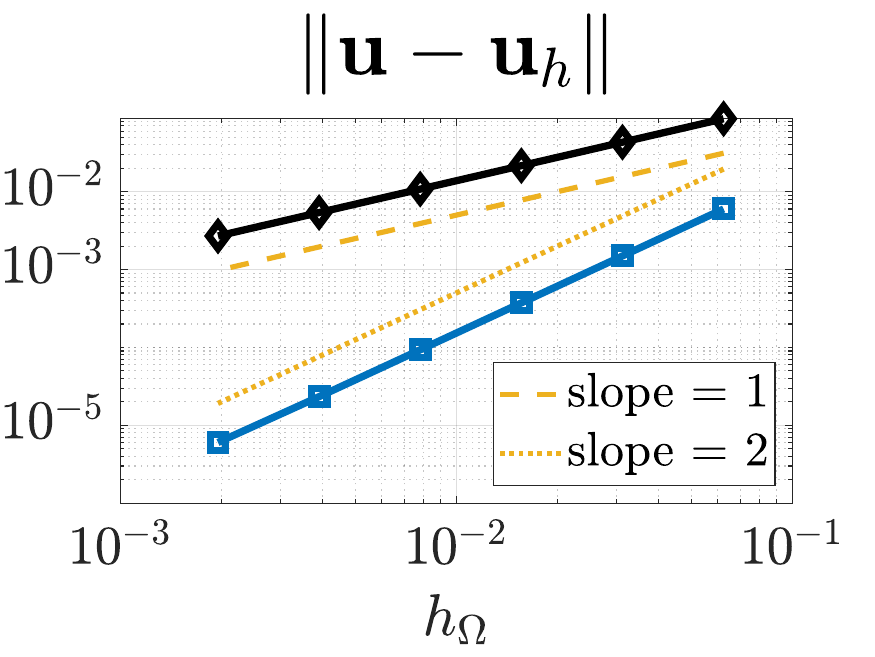}
		\includegraphics[width=0.24\linewidth]{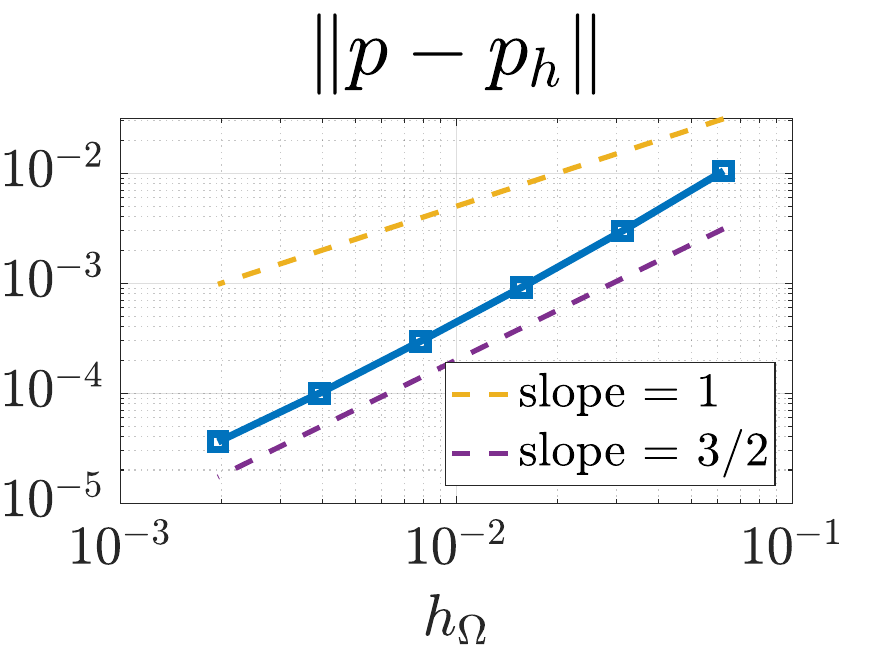}
		\includegraphics[width=0.24\linewidth]{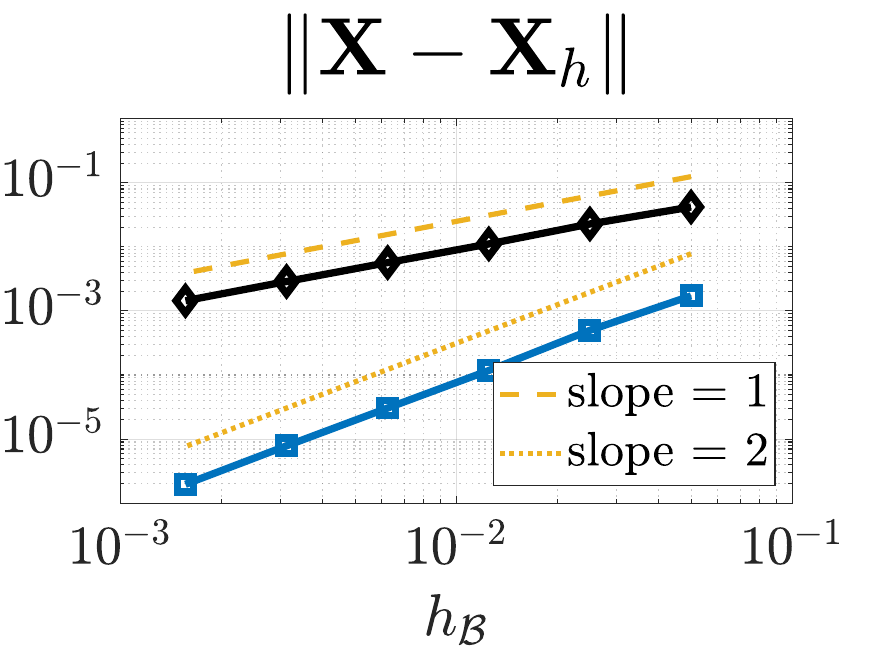}
		\includegraphics[width=0.24\linewidth]{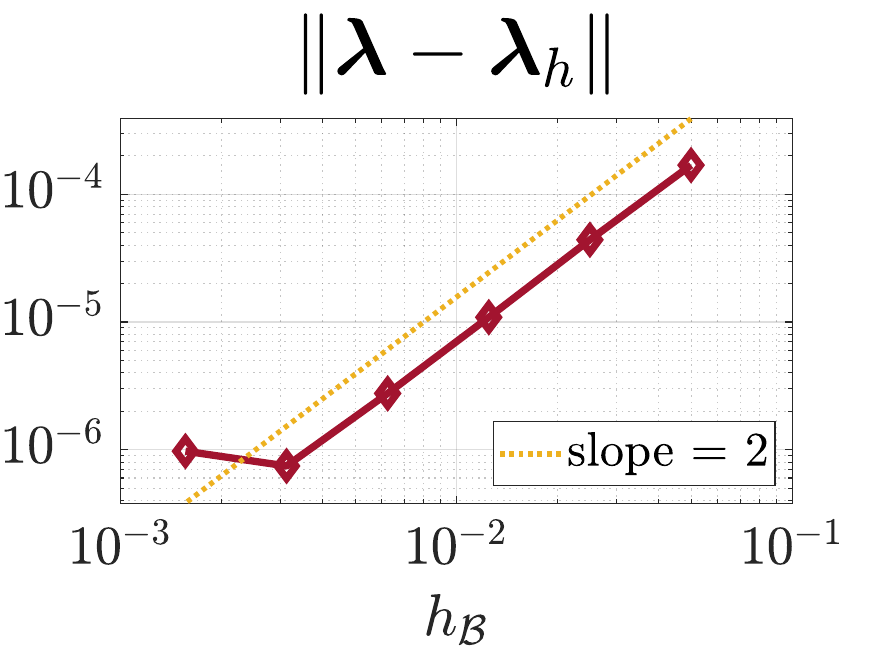}
		\includegraphics[width=0.425\linewidth]{figures_paper/legend_l2_2}
		
		\
		
		\includegraphics[width=0.24\linewidth]{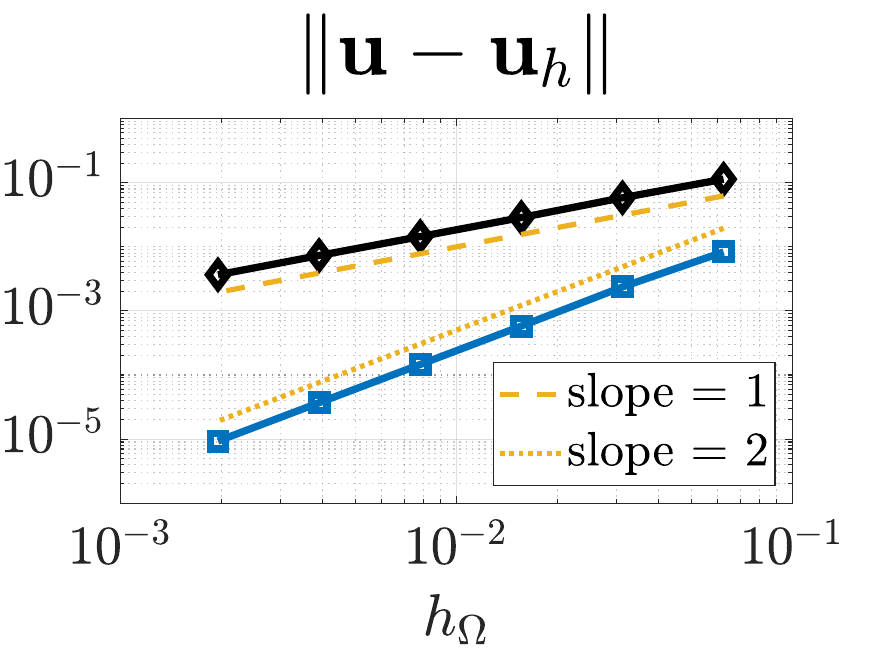}
		\includegraphics[width=0.24\linewidth]{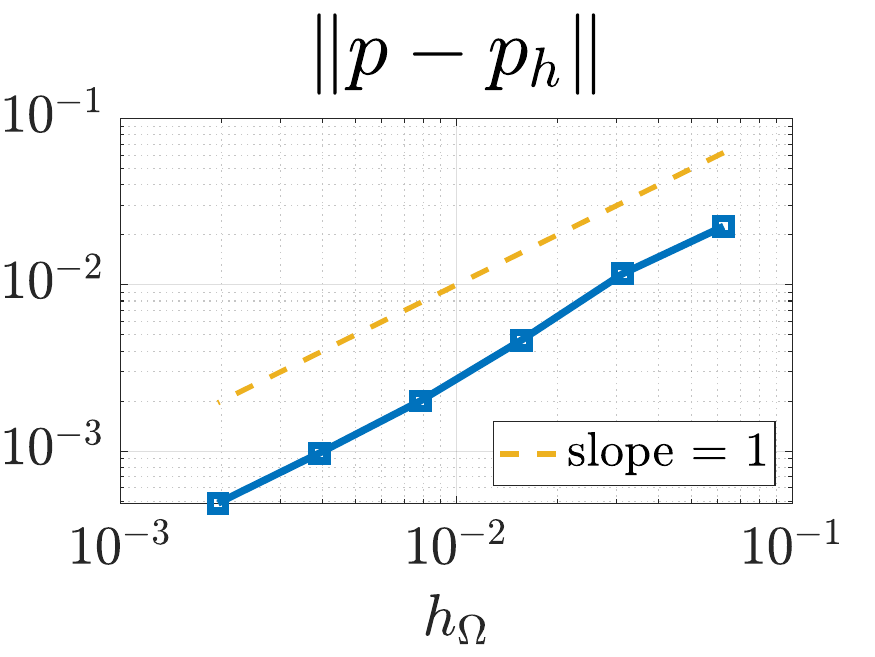}
		\includegraphics[width=0.24\linewidth]{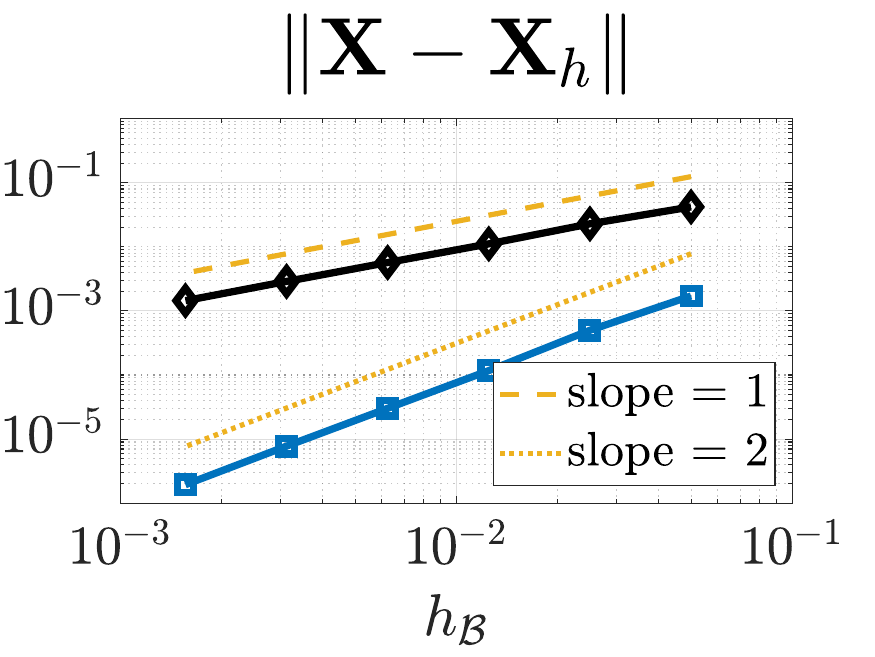}
		\includegraphics[width=0.24\linewidth]{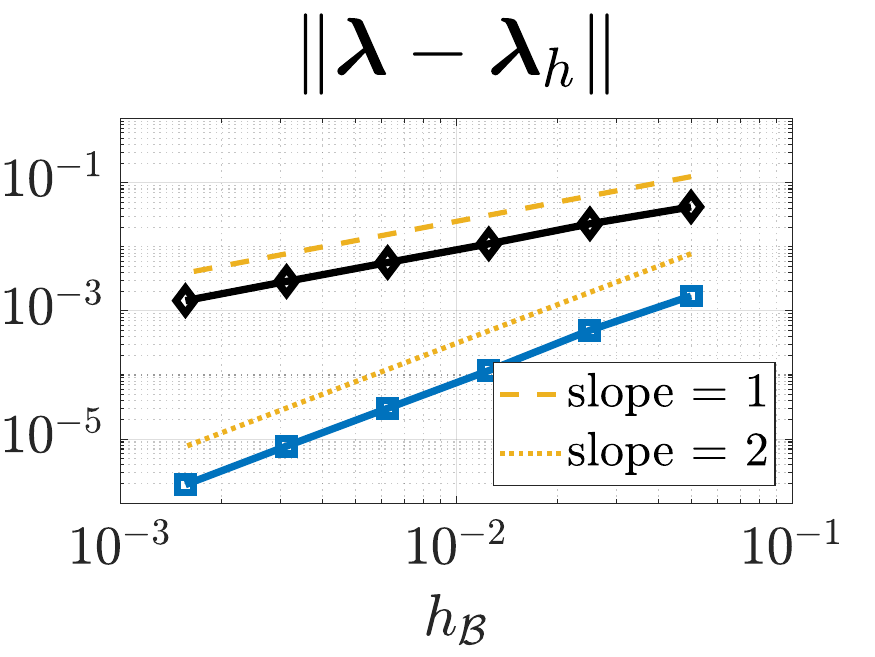}
		\includegraphics[width=0.25\linewidth]{figures_paper/legend_h1_2}
		\caption{Error convergence for the annulus test.}
		\label{fig:annulus}
	\end{figure}
	
	Concerning the disk example, we observe that $p$ is discontinuous along the interface, thus ${p\in H^{s}(\Omega)}$  with $0\le s<1/2$. Hence, we cannot expect optimal rate of convergence for this variable even if the coupling term is computed exactly.  In the other two cases, since the solution is regular, the method shows optimal convergence.
	The only critical case is given by the error estimate of the variable
	$\lambda$ in Figure~\ref{fig:annulus} that has a surprising behavior in the
	last level of refinement. This might be due to the ill-conditioning of the
	matrix: according to the proven estimate, the condition number is growing like
	$O(h^{-4})$ and in the last test we might have reached the critical value of
	$h$ so that the error stops decreasing.

	\subsection{Time dependent problem: the stretched annulus}
	\fc
	In this last subsection, we report some numerical results for the time dependent Problem~\ref{pro:problem_fictitious} of a stretched annulus immersed in a square fluid domain. We focus on the effect of small cut cells on the stability in time.
	
	Let $\Omega$ be the unit square and $\B=\{\s\in\R^2:\, 0.125\le|\s|\le0.25 \}$ the reference solid domain corresponding to the annulus at rest. 
	The physical parameters are densities $\rho_f=1.0$ and $\rho_s=1.1$, equal fluid and solid viscosities $\nu=0.01$, and elasticity constant $\kappa=0.2$. We assume that at $t=0$ the fluid is at rest, hence $\u_0=0$, while the annulus is stretched and its initial deformation is given by $\X_0 = \left({s_1}/{1.4}+0.5,1.4\,s_2+0.5\right)$. The internal elastic forces drive the annulus back to its resting position, thus generating the motion of the system, which is analyzed until $T=4$.
	
	The mesh for the pressure unknown, $\T_h^\Omega$, is obtained by dividing $\Omega$ into $N^2$ squares then split into two triangles, with $h_\Omega=1/N$. Consequently, the velocity unknown is approximated on $\T^\Omega_{h/2}$, see~\eqref{eq:element}. The grid $\T_h^\B$ is generated with \texttt{Gmsh}~\cite{geuzaine2009gmsh} in such a way that $h_\B\approx (4/3)h_\Omega$.
	We consider three configurations by setting $N=32\text{ (coarse)},\,64\text{ (medium)},\,128\text{ (fine)}$. For the time step, we choose $\dt=0.1,\,0.01$. For the coupling, we use $\c=\c_0$ computed exactly.
	
	In Figure~\ref{fig:snap}, some snapshots report the velocity contour lines, together with the position of the structure (top line) and the pressure profiles (bottom line). We plot the system configuration at time $t=0.5$, $t=1$ and $t=4$: this choice is due to the fact that the evolution of the immersed structure is faster at the beginning, while after $t=2$, it almost reaches the resting position. Moreover, since the pressure is discontinuous along the interface and it is discretized by continuous piecewise linear elements, the Gibbs phenomenon appears. Notice that the value of $p_h$ ranges in $(-0.05,0.05)$ at $t=0.5$, then it decreases to values in $(0,0.01)$ at $t=4$.
	
	In order to assess the unconditional stability with respect to the time step, we compute the system energy 
	\begin{equation*}
		\mathcal{E}(\u_h^n,\X_h^n) = \frac{\rho_f}{2}\|\u_h^n\|^2_{0,\Omega} + \frac{\dr}2\left\|\frac{\X_h^n-\X_h^{n-1}}{\dt}\right\|_{0,\B}^2+\frac\kappa 2|\X_h^n|_{1,\B}^2.
	\end{equation*}
	Figure~\ref{fig:energy} displays the behavior of the energy ratio $\mathcal{E}(\u_h^n,\X_h^n)/\mathcal{E}(\u_h^0,\X_h^0)$, computed by setting $\X_h^{-1}=0$. We observe that the energy is decreasing in the interval of interest and again that the evolution is faster for $t\in(0,2)$. After that, the resting position is almost reached. Essentially, the energy behaves independently of the considered meshes and time steps. It seems that the possible presence of small cut cells does not have any effect on the energy. In order to check if the computation of the coupling term actually produces small cut cells, we follow the evolution of a randomly chosen solid element. At each time step, the picked element is subdivided into several (polygonal) cells taking into account the intersections with the underlying velocity mesh. We evaluate the area of these sub-cells and we mark it in the semi-log scatter plots in Figure~\ref{fig:area}. We see that, for $\dt=0.1$, cells with area of the order $10^{-9}$ -- $10^{-10}$ appear for all meshes, while for $\dt=0.01$, the smallest cells can have an area ranging from $10^{-12}$ to $10^{-15}$, depending on the mesh refinement.
	
	\begin{figure}
		\centering
		
		\includegraphics[width=0.3\linewidth]{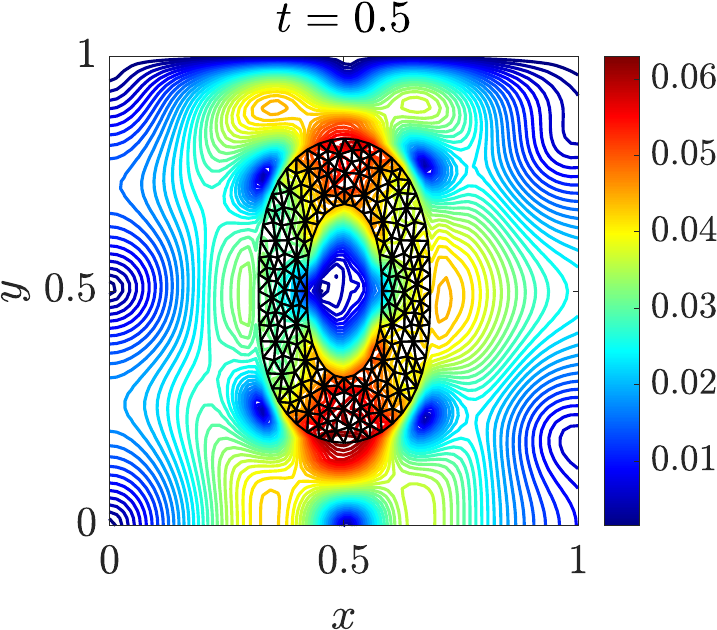}\quad
		\includegraphics[width=0.3\linewidth]{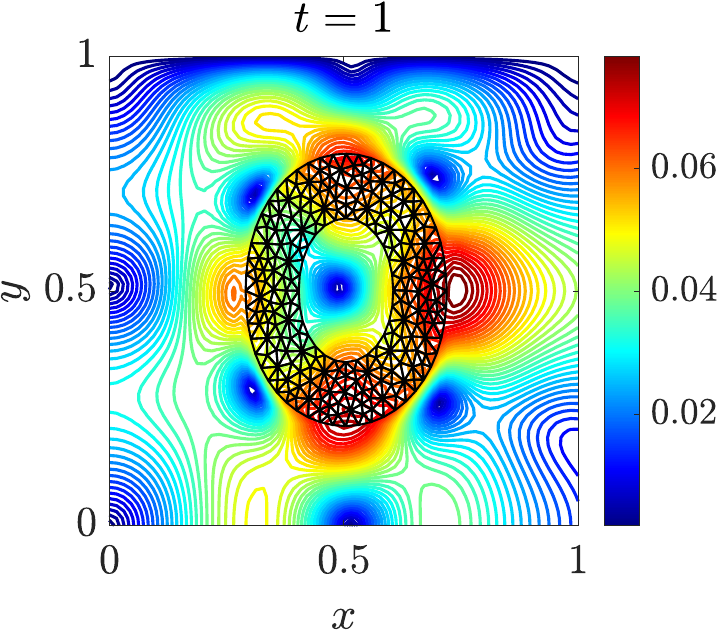}\quad
		\includegraphics[width=0.31\linewidth]{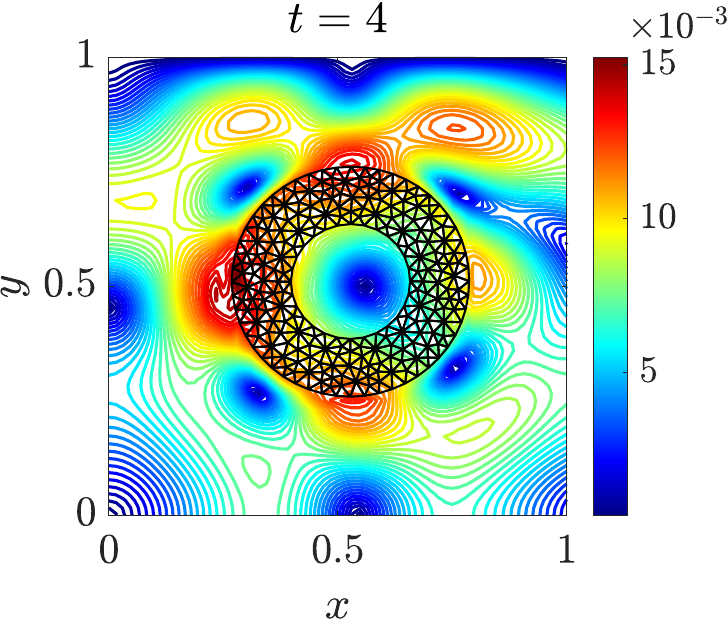}
		
		\includegraphics[width=0.3\linewidth]{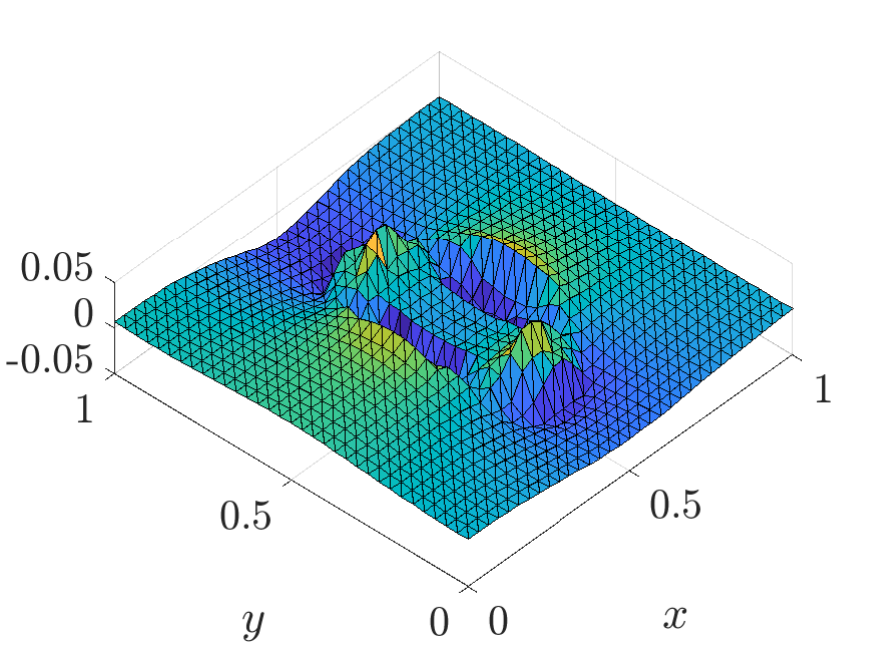}\quad
		\includegraphics[width=0.3\linewidth]{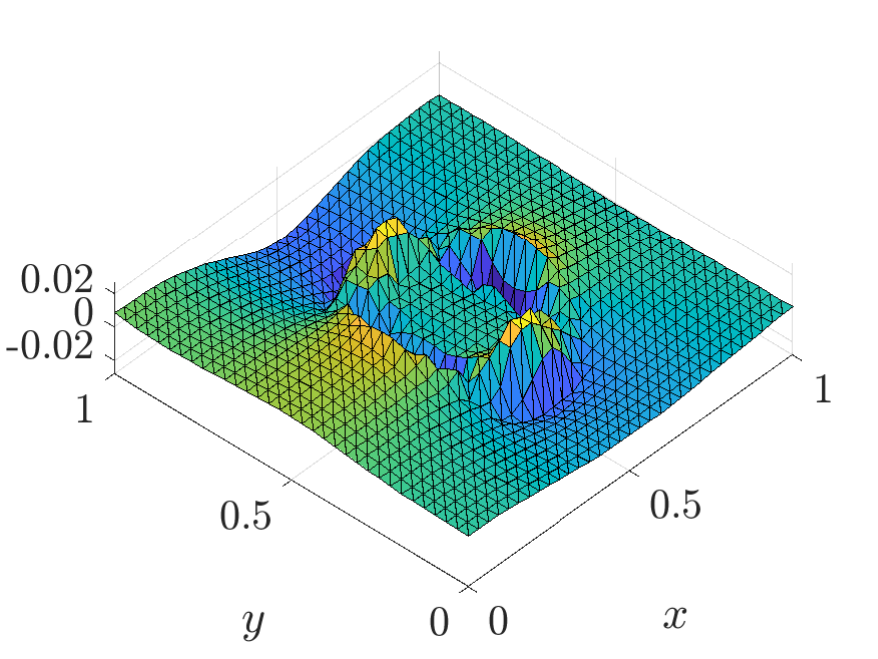}\quad
		\includegraphics[width=0.3\linewidth]{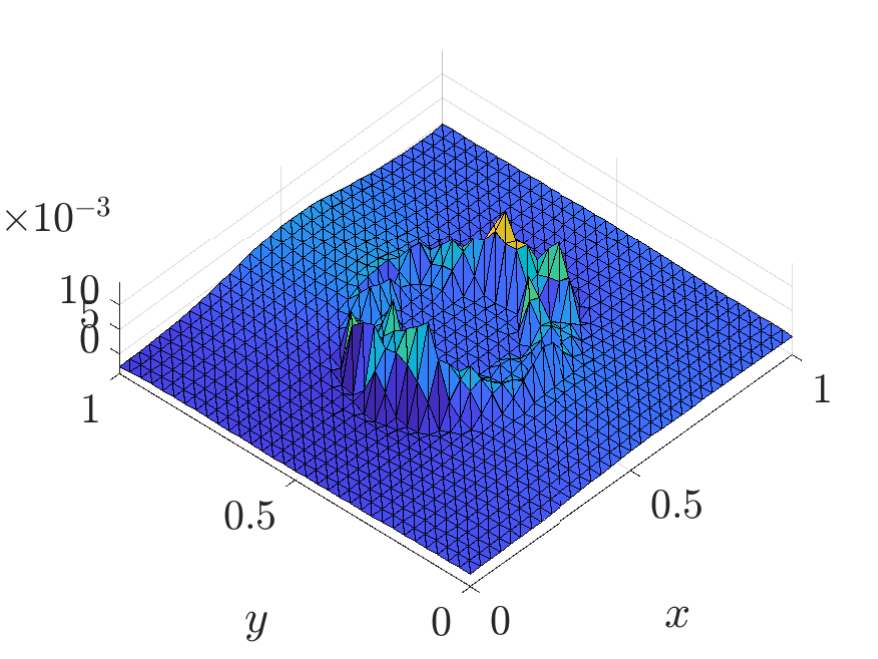}
		
		\caption{Snapshots of the annulus evolution (coarse test). Streamlines of $|\u_h|$ and structure position are reported on the top line, while the pressure field on the bottom line.}
		\label{fig:snap}
	\end{figure}
	
	\begin{figure}
		\centering
		\includegraphics[width=0.3\linewidth]{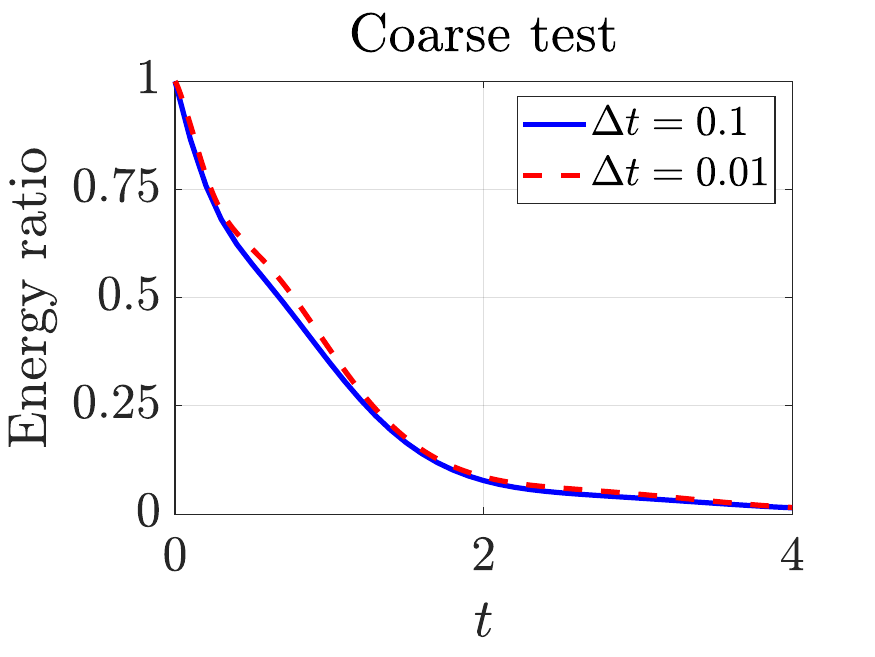}\quad
		\includegraphics[width=0.3\linewidth]{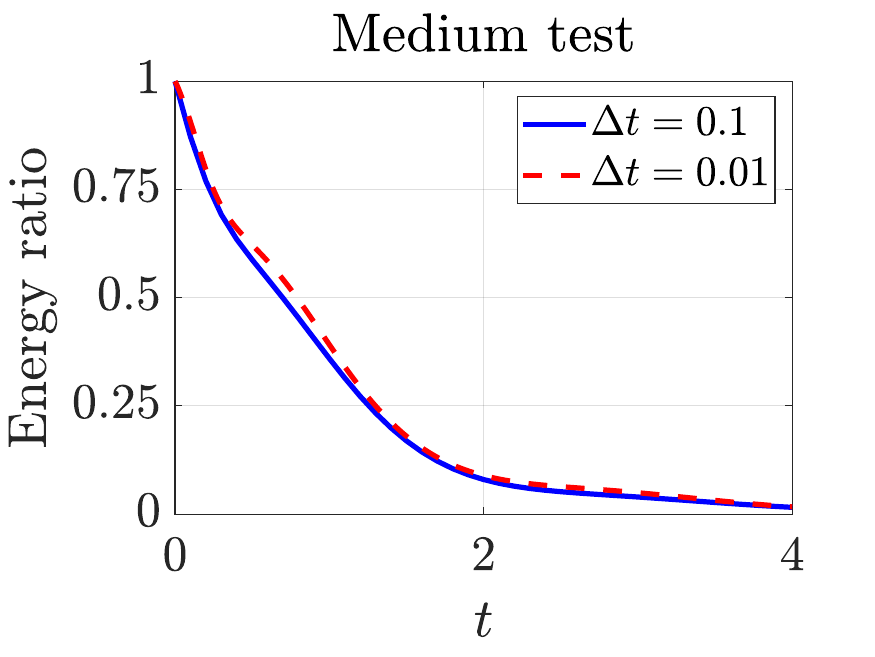}\quad
		\includegraphics[width=0.3\linewidth]{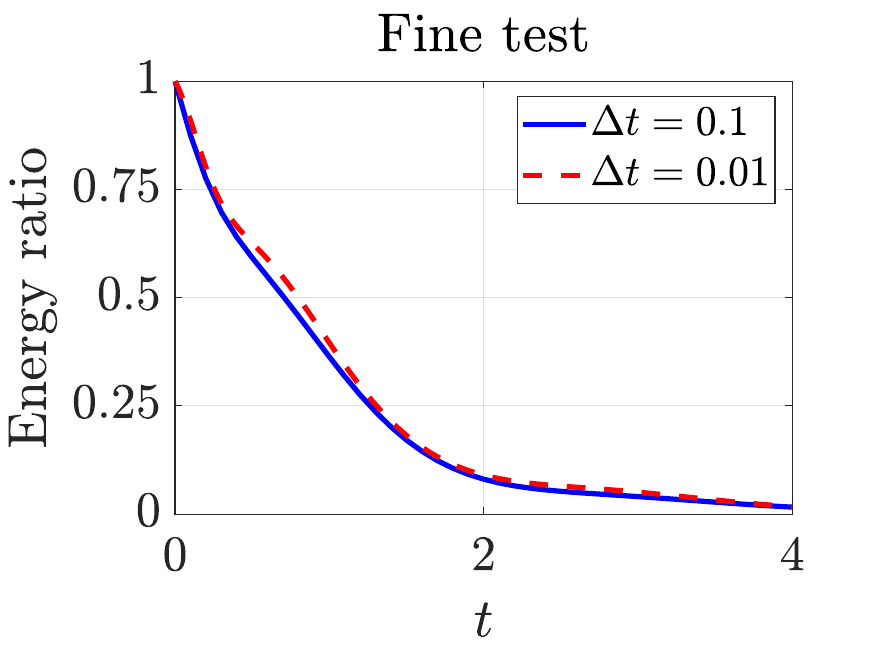}
		\caption{Energy ratio $\mathcal{E}(\u_h^n,\X_h^n)/\mathcal{E}(\u_h^0,\X_h^0)$ of the stretched annulus: evolution in time.}
		\label{fig:energy}
	\end{figure}

	\begin{figure}
		\includegraphics[width=0.3\linewidth]{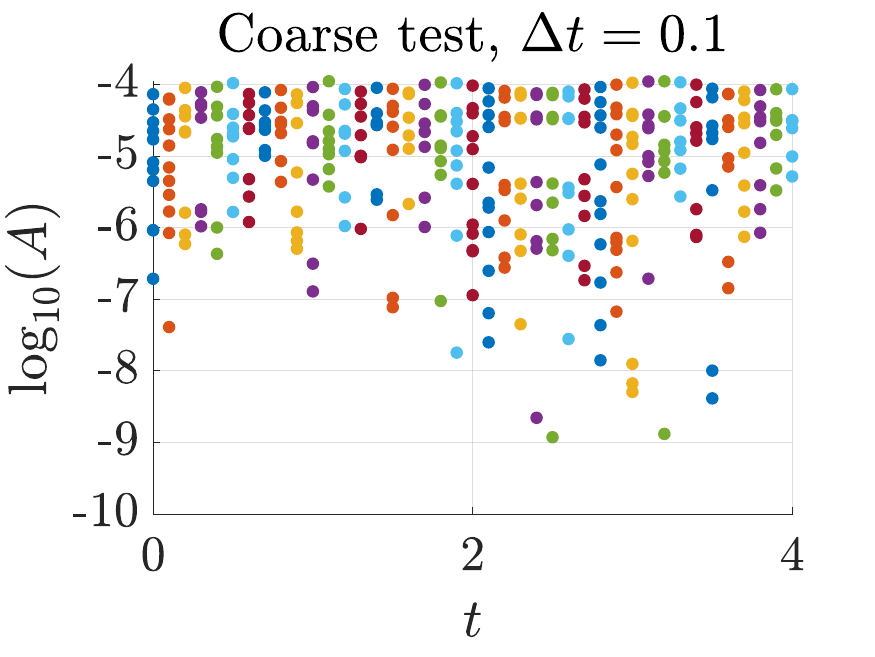}\qquad
		\includegraphics[width=0.3\linewidth]{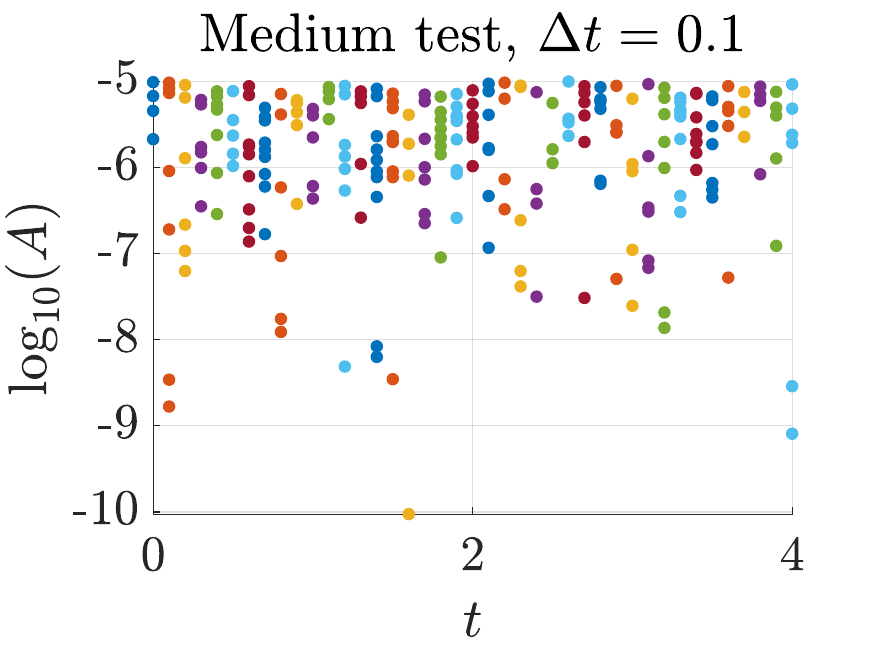}\qquad
		\includegraphics[width=0.3\linewidth]{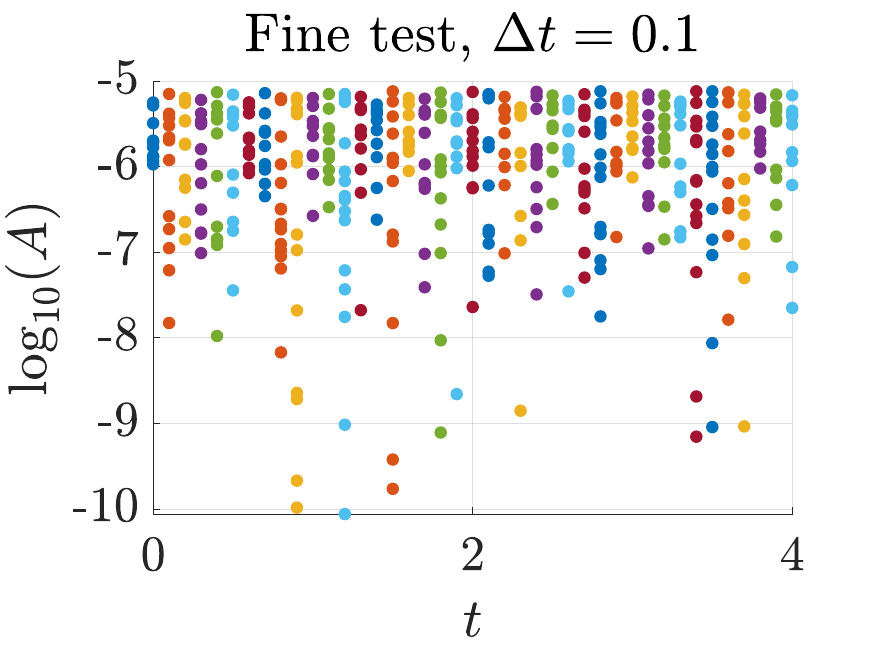}
		
		\
		
		\includegraphics[width=0.3\linewidth]{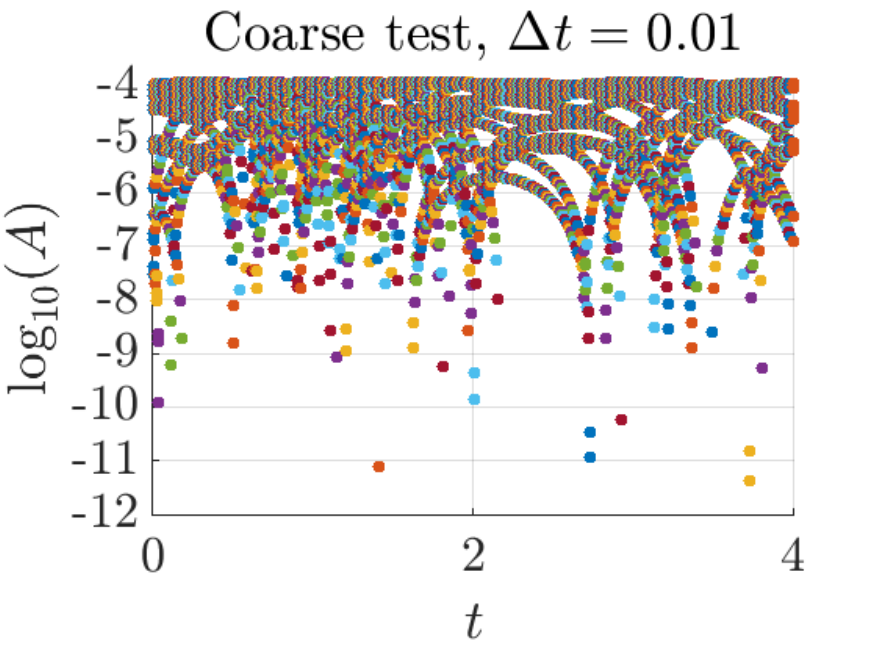}\qquad
		\includegraphics[width=0.3\linewidth]{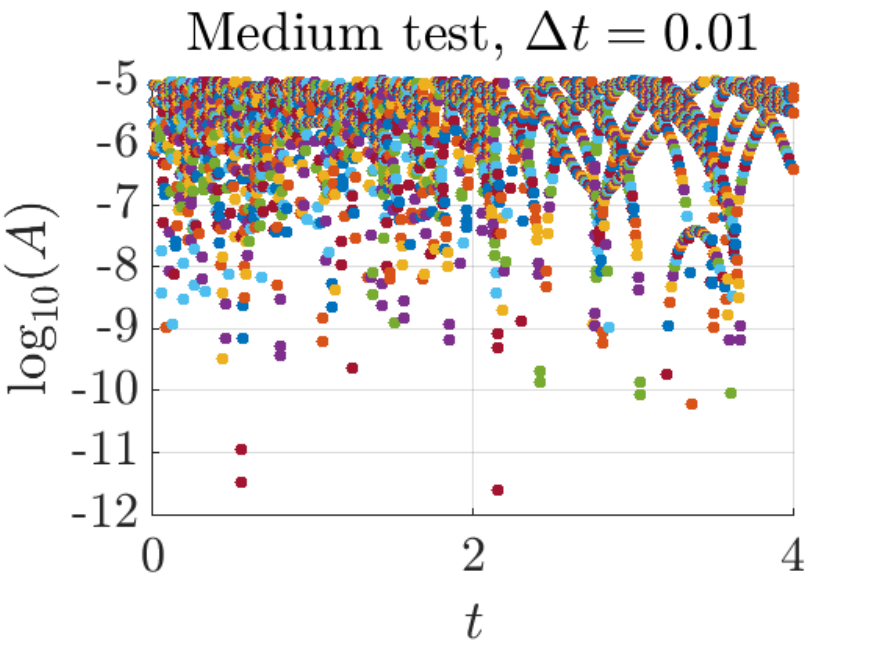}\qquad
		\includegraphics[width=0.3\linewidth]{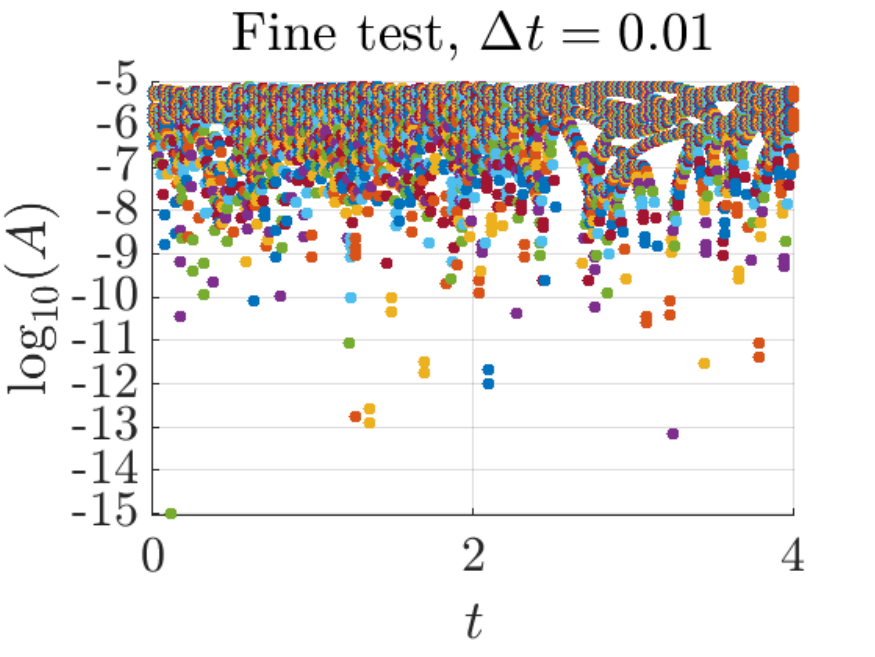}
	
		\caption{Area of the cut cells resulting from a fixed solid element. For each test, the solid element is randomly chosen and its evolution in time is studied.}
		\label{fig:area}
	\end{figure}
	\cf
	
	\section{Conclusion}\label{sec:conclusion}
	
	We considered the fictitious domain formulation for fluid-structure interaction problems introduced in~\cite{2015}. The approach uses a distributed Lagrange multiplier to enforce the kinematic condition and, consequently, to represent the interaction between fluid and solid. Thus, at discrete level, the so-called \textit{coupling term} is defined through functions constructed over non-matching grids. As extensively described in~\cite{boffi2022interface,BCG24}, the discrete coupling term can be constructed either exactly, i.e. by computing the intersection between the involved meshes, or in approximate way by accepting the presence of a quadrature error.
	
	After recalling the existing theoretical results concerning the unconditional stability in time and the well-posedness of the mixed finite element discretization, we emphasized that such results are independent of the position of the interface, hence not affected by the presence of small cut cells while assembling the coupling matrix. Our formulation is naturally stable without resorting to any penalization term.
	
	We then proved upper bounds for the condition number of the fluid-structure FEM system. To this aim, we made use of the theory provided by~\cite{ern-guermond}. We observed that the condition number depends on the choice of coupling term, but does not deteriorate in the presence of small cut cells. Several numerical tests confirmed our theoretical findings.
	
	\section*{Acknowledgments}
	The authors are member of INdAM Research group GNCS. The research of L. Gastaldi is partially supported by PRIN/MUR (grant No.20227K44ME) and IMATI/CNR.
	
	\bibliographystyle{abbrv}
	\bibliography{biblio}

\end{document}